\documentclass[letterpaper]{amsart}
\usepackage{amssymb}
\usepackage{amsthm}
\usepackage{amsmath}
\usepackage{bm}
\usepackage{amsfonts}
\usepackage{yfonts}
\usepackage{comment}
\usepackage[mathscr]{euscript}
\usepackage{tikz-cd,xfrac,MnSymbol,xargs,todonotes,soul,enumitem}
\usepackage[all]{xy}

\newcommand{\newdownfree}{\scalebox{1.5}[1]{\ensuremath{\downfree}}}
\newcommand{\newndownfree}{\scalebox{1.5}[1]{\ensuremath{\ndownfree}}}
\usepackage{xspace}

\newcommandx{\margin}[2][1=]{}%{\todo[linecolor=blue,backgroundcolor=blue!25,bordercolor=blue,#1]{#2}}

\def\Ind#1#2{#1\setbox0=\hbox{$#1x$} \kern\wd0    \hbox to 0pt{\hss$#1\newdownfree$\hss}\kern\wd0}
\def\Notind#1#2{#1\setbox0=\hbox{$#1x$}\kern\wd0 \hbox to 0pt{\hss$#1\newndownfree$\hss}\kern\wd0}
\newcommand{\ind}[1][]{\mathop{\mathpalette\Ind{}^{\!\!\!\!\rlap{$\scriptscriptstyle{#1}$}\,\,} }}
\newcommand{\nind}[1][]{\mathop{\mathpalette\Notind{}^{\!\!\!\rlap{$\scriptscriptstyle{#1}$}\,}}}
\newcommand{\eind}[1][]{\mathop{\mathpalette\Ind{}^{\!\!\!\!\rlap{$\scriptscriptstyle{#1}$}\,\,\,\,\,} }}
\newcommand{\enind}[1][]{\mathop{\mathpalette\Notind{}^{\!\!\!\rlap{$\scriptscriptstyle{#1}$}\,\,\,\,}}}

\newcommand{\F}{\mathbb{F}}

\newcommand{\acl}{\mathrm{acl}}
\newcommand{\dcl}{\mathrm{dcl}}
\newcommand{\bcl}{\mathrm{bcl}}
\newcommand{\ccl}{\mathrm{ccl}}

\newcommand{\cl}{\mathrm{cl}}

\newcommand{\eq}{\mathrm{eq}}

\newcommand{\elesub}{\preccurlyeq}
\newcommand{\nn}{\mathbb{N}}

\newcommand{\zz}{\mathbb{Z}}

\newcommand{\ACF}{\mathrm{ACF}}

\newcommand{\sM}{\mathscr{M}}
\newcommand{\sN}{\mathscr{N}}

\newcommand{\sK}{\mathscr{K}}
\newcommand{\sA}{\mathscr{A}}
\newcommand{\sB}{\mathscr{B}}
\newcommand{\sV}{\mathscr{V}}

\newcommand{\id}{\mathrm{id}}
\newcommand{\tp}{\mathrm{tp}}
\newcommand{\qftp}{\mathrm{qftp}}
\newcommand{\Aut}{\mathrm{Aut}}
\newcommand{\Th}{\mathrm{Th}}
\newcommand{\fdiag}{\mathrm{Fdiag}}
\newcommand{\ediag}{\mathrm{Ediag}}
\newcommand{\ACFA}{\mathrm{ACFA}}

\newcommand{\monster}{\boldsymbol{\sM}}

\newcommand{\monsterset}{\boldsymbol{M}}

\newcommand{\NSOP}{\mathrm{NSOP}}
\newcommand{\SOP}{\mathrm{SOP}}
\newcommand{\NIP}{\mathrm{NIP}}

\newcommand{\Fraisse}{Fra\"iss\'e\xspace}

\makeatletter
\newsavebox\myboxA
\newsavebox\myboxB
\newlength\mylenA

\newcommand*\xbar[2][0.75]{%
    \sbox{\myboxA}{$\m@th#2$}%
    \setbox\myboxB\null% Phantom box
    \ht\myboxB=\ht\myboxA%
    \dp\myboxB=\dp\myboxA%
    \wd\myboxB=#1\wd\myboxA% Scale phantom
    \sbox\myboxB{$\m@th\overline{\copy\myboxB}$}%  Overlined phantom
    \setlength\mylenA{\the\wd\myboxA}%   calc width diff
    \addtolength\mylenA{-\the\wd\myboxB}%
    \ifdim\wd\myboxB<\wd\myboxA%
       \rlap{\hskip 0.5\mylenA\usebox\myboxB}{\usebox\myboxA}%
    \else
        \hskip -0.5\mylenA\rlap{\usebox\myboxA}{\hskip 0.5\mylenA\usebox\myboxB}%
    \fi}
\makeatother

\newtheorem{thm}{Theorem}[section]
\newtheorem{lem}[thm]{Lemma}
\newtheorem*{thm*}{Theorem}
\newtheorem*{prop*}{Proposition}
\newtheorem*{v-conj}{Vague Conjecture}
\newtheorem{prop}[thm]{Proposition}
\newtheorem{cor}[thm]{Corollary}
\newtheorem{fact}[thm]{Fact}
\theoremstyle{definition}
\newtheorem{defn}[thm]{Definition}
\newtheorem{rem}[thm]{Remark}
\newtheorem{example}[thm]{Example}
\newtheorem{question}[thm]{Question}

\definecolor{red}{rgb}{1.0, 0, 0}

\begin{document}
\sloppy

\title[Interpolative fusions II] {Interpolative fusions II: Preservation results}
\author{Alex Kruckman, Chieu-Minh Tran, Erik Walsberg}
%\address{Department of Mathematics, University of Illinois at Urbana-
%Champaign, Urbana, IL 61801, U.S.A}
%\curraddr{}
\email{akruckman@wesleyan.edu, ewalsber@uci.edu, mtran6@nd.edu}
%\subjclass[2010]{Primary 03C65; Secondary 03B25, 03C10, 03C64}
\date{\today}

\begin{abstract}
We study interpolative fusion, a method of combining theories $T_1$ and $T_2$ in distinct languages in a ``generic'' way over a common reduct $T_\cap$, to obtain a theory $T_\cup^*$. When each $T_i$ is model-complete, $T_\cup^*$ is the model companion of the union $T_1\cup T_2$. Our goal is to prove preservation results, i.e., to  find sufficient conditions under which model-theoretic properties of $T_1$ and $T_2$ are inherited by $T_\cup^*$. 

We first prove preservation results for quantifier elimination, model-completeness, and related properties. We then apply these tools to show that, under mild hypotheses, including stability of $T_\cap$, the property $\mathrm{NSOP}_1$ is preserved. We also show that simplicity is preserved under stronger hypotheses on algebraic closure in $T_1$ and $T_2$. This generalizes many previous results;
for example, simplicity of $\mathrm{ACFA}$ and the random $n$-hypergraph are both non-obvious corollaries. We also address preservation of stability, $\mathrm{NIP}$, and $\aleph_0$-categoricity, and we describe examples which witness that these results are  sharp.
%Suppose $L_1,L_2$ are languages and $T_i$ is a model complete $L_i$-theory for $i \in \{1,2\}$ such that $T_\cup = T_1 \cup T_2$ is consistent, and the $T_i$ have a common set of consequencess $T_\cap$ in the language $L_\cap = L_1\cap L_2$. 
%Suppose $T_\cup$ has a model companion $T^*_\cup$.
%We show that if  $T_\cap$ is stable, $T_1,T_2$ are $\NSOP_1$, and Kim-independence in each $T_i$ satisfies the $T_\cap$-generic independence theorem, then $T^*_\cup$ is $\NSOP_1$.
%We also show that if each $T_i$ is simple and disintegrated relative to $T_\cap$ then $T^*_\cup$ is simple.
%This generalizes many previous results.
%For example, simplicity of $\ACFA$ and the random $n$-hypergraph are both non-obvious corollaries to the second statement.
%We describe a number of examples which witness that these results are reasonably sharp.
\end{abstract}

\maketitle

%\tableofcontents

\section{Introduction}
\noindent
It is a well-known idea that simplicity can be seen as ``stability $+$ randomness". The results in this paper suggest a slight variation of this idea: the property $\NSOP_1$ should be seen as ``stability $+$ randomness", while simplicity should be seen as ``stability $+$ controlled randomness''. In particular, we show that a generic (or ``random") fusion of stable theories (over a common reduct) is $\NSOP_1$.  This provides us with a conceptual explanation for the fact that several known theories are $\NSOP_1$ by also observing that these theories can be decomposed as generic fusions of stable theories. The above result is a special case of the following  preservation phenomenon (Theorem~\ref{thm: mainintro}):  Under some mild hypotheses, a generic fusion of $\NSOP_1$ theories over a common stable reduct is $\NSOP_1$. 
Simplicity, on the other hand, is preserved under the more restrictive setting of ``relatively disintegrated fusions" (Theorem~\ref{thm:main2}).
Several important examples do  fit into this setting: the simplicity of the random graph and simplicity of $\ACFA$ can both be seen as special cases of Theorem~\ref{thm:main2}. We also show that other model-theoretic properties (stability, NIP, $\aleph_0$-categoricity, etc.) of many well-known examples can be seen as special cases of results on preservation of these properties under generic fusions.

\medskip\noindent
We first describe the basic setting and set some global conventions.
Throughout $I$ is an index set, $L_\cap$ is a language, $(L_i)_{i \in I}$ is a family of languages, all with the same set of sorts, such that $L_i \cap L_j = L_\cap$ when $i \neq j$, $T_i$ is a (possibly incomplete) consistent $L_i$-theory for each $i$, and each $T_i$ has the same set $T_\cap$ of $L_\cap$-consequences.
In most examples each $T_i$ is complete, in this case our assumptions ensure that $T_i \cap T_j = T_\cap$ when $i \neq j$.
We declare $L_\cup := \bigcup_{i \in I} L_i$, $T_\cup := \bigcup_{i \in I} T_i$, and suppose that $T_\cup$ is consistent.
(When $T_\cap$ is complete this follows by Robinson joint consistency.)
Throughout $\sM_\cup$ is an $L_\cup$-structure and $\sM_\square$ is the $L_\square$-reduct of $\sM_\cup$ for $\square \in I \cup \{\cap\}$.

\medskip\noindent
%Three major themes in model theory are the study of ``random" or ``generic" structures, the study of model companions of theories of interest, and the classification of first order theories according to fundamental combinatorial properties.
In \cite{firstpaper} we began to study the interpolative fusion $T^*_\cup$ of $(T_i)_{i \in I}$ over $T_\cap$ (which exists under some conditions); this is the precise version of the generic ( or ``random'') fusion described earlier. 
When each $T_i$ is model complete $T^*_\cup$ is simply the model companion (if it exists) of $T_\cup$ and in general one can view $T^*_\cup$ as the model companion of $T_\cup$ \textit{relative} to $(T_i)_{i \in I}$.
(We will often reduce to the case when each $T_i$ is model complete by Morleyizing.)
We prefer to work with a language-independent definition of $T^*_\cup$ described in Section~\ref{section:first-paper}. Our first paper~\cite{firstpaper} was largely devoted to existence conditions for $T^*_\cup$.
In this paper we will generally assume that $T^*_\cup$ exists.
The reader of this paper will only need certain things from \cite{firstpaper} which are recalled in  Section~\ref{section:first-paper}. 

\medskip\noindent Our notion of fusion generalizes many ``generic" model-theoretic constructions.
Interesting theories are often (up to bi-interpretation) fusions of strictly simpler theories, e.g. the theory of the random graph is bi-interpretable with a fusion of two theories which are both bi-interpretable with the theory of equality and $\ACFA$ is bi-interpretable with a fusion of two theories which are both bi-interpretable with $\mathrm{ACF}$.
We survey many more examples in Section~\ref{section:examples}; these provid us with a testing ground for the preservation results we will discuss next.

% \medskip\noindent
% In the present paper we study preservation of model-theoretic properties with an emphasis on classification theoretic properties.
% We show that a number of results on specific theories lift to general fusion-preservation results.

%\medskip\noindent
%In this paper we study preservation of classification theoretic properties under our fusion.
%We show that many results of the form ``$T$ has $P$" for some concrete theory $T$ and tameness property $P$ are special cases of results of the form ``if each $T_i$ has $P$ (and possibly additional conditions are satisfied) then the fusion of $(T_i)_{i \in I}$ has $P$".
%In almost every case below the additional conditions concern the individual $T_i$ and the relationship of the individual $T_i$ to the base theory, and do not concern any relationship between the various $T_i$.

\medskip\noindent
%The main goal of this paper is to study preservation of classification theoretic properties.
Section~\ref{sec: combinatorialtameness} studies the preservation of classification theoretic properties.
Stability and $\NIP$ are only rarely preserved (see Section~\ref{sec:stablenip}).
This is natural, as the fusion typically introduces ``randomness". We do find sufficient conditions for their preservation which cover a number of interesting examples.
At present, $\NSOP_1$ is the only classification theoretic property which we know to be broadly preserved (see Section~\ref{sec:NSOP1}). 
An extra necessary hypothesis is a slight strengthening of the independence theorem, which we call the $T_\cap$-generic independence theorem. This hypothesis holds in all examples that we know of an $\NSOP_1$ theory with a stable reduct. Our main result for $\NSOP_1$ is Theorem~\ref{thm:preservationintro} (following from Theorem~\ref{thm:preservation}, Corollary~\ref{cor:stable-case}, and Corollary~\ref{cor:disforking-case}).
% \textcolor{blue}{satisfied if in the expansion amalgamating sets independently over models (i.e. applying the independence theorem) does not force their algebraic closures to be dependent in the sense of the reduct.
% This is always true in the stable case.
% When we prove that it works for stable theories over models, ask if it holds over general sets of parameters.
% If we get this we get that fusions of stable theories preserve wei.}

\begin{thm} \label{thm: mainintro}
\label{thm:preservationintro}
Assume $T^*_\cup$ exists. Suppose $T_\cap$ is stable, each $T_i$ is $\NSOP_1$, and Kim-independence in each $T_i$ satisfies the $T_\cap$-generic independence theorem. Then  $T_\cup^*$ is $\NSOP_1$.
It follows that:
\begin{enumerate}
\item If each $T_i$ is stable, then $T^*_\cup$ is $\NSOP_1$.
\item If $T_\cap$ is stable with disintegrated forking and each $T_i$ is $\NSOP_1$, then $T^*_\cup$ is $\NSOP_1$.
\end{enumerate}
\end{thm}

\noindent
We say a theory $T$ has \textbf{disintegrated forking} if $A \enind[f]_C B$ implies $A \enind[f]_C b$ for some element $b \in B$.
Some authors refer to this as ``trivial forking".
An $\NSOP_1$ theory is either simple or has the independence property.
Hence it follows from (1) above that if each $T_i$ is stable and $T^*_\cup$ is $\NIP$ then $T^*_\cup$ is stable.

\medskip\noindent
$\NSOP_1$ theories come equipped with a canonical independence relation, Kim-forking.
Under the assumptions of Theorem~\ref{thm:preservationintro}, we give an explicit description of Kim-forking in $T^*_\cup$ in terms of Kim-forking in the various $T_i$.
(In a simple theory Kim-forking agrees with forking.)

\medskip\noindent
Using results about preservation of $\aleph_0$-categoricity discussed in  Section~\ref{sec: Preservationcountablecat}, we show in Section~\ref{section:counterexample} that stability of $T_\cap$ in Theorem~\ref{thm: mainintro} is necessary. In particular, we construct two simple theories $T_1$ and $T_2$ whose fusion over a simple theory $T_\cap$ interprets the Henson graph, the best-known example of a theory which is $\NSOP$ and $\SOP_1$.
In this example $T_\cap$ is the theory of the random $3$-hypergraph.

\medskip\noindent
We will see that Theorem~\ref{thm:preservationintro} generalizes several known results: the model companion of the empty theory in an arbitrary language is $\NSOP_1$, the generic Skolemization of an $\NSOP_1$ theory is $\NSOP_1$, the theory of an algebraically closed field of positive characteristic with a generic additive subgroup is $\NSOP_1$, etc.

%\medskip \noindent
%The assumption of stability is more mild than it may appear as $T_\cap$ is stable is essentially all examples of interest.
%In most examples in the present paper $T_\cap$ is even $\aleph_0$-stable and $\aleph_0$-categorical.

\medskip\noindent
%We view $\NSOP_1$ as a generalization of simplicity which allows for more ``randomness".
%For example the theory of a generic $n$-ary relation is simple and the theory of a generic $n$-ary function is $\NSOP_1$ and not simple.
Simplicity is only preserved under more restrictive hypotheses (see Section~\ref{Sec: Simplicity}). However, some of the main examples of simple theories can be realized as fusions of stable theories satisfying these hypotheses. Theorem~\ref{thm:main2} is our main result on simple theories (following from Theorem~\ref{thm:simplepreservation}).
We say that $T_i$ is disintegrated relative to $T_\cap$ (or just \textbf{relatively disintegrated}) if for any $\sM_i \models T_i$ we have
$$ \acl_i(A) = \acl_\cap\left( \bigcup_{a \in A} \acl_i(a) \right) \quad \text{for all subsets  } A \text{  of  } M. $$

%\noindent
%If $T_\cap$ is the theory of equality then $T_i$ is relatively disintegrated if and only if $T_i$ is disintegrated.

\begin{thm}
\label{thm:main2}
Suppose $T^*_\cup$ exists, each $T_i$ is simple, $T_\cap$ is stable, and either:
\begin{enumerate}
\item every $T_i$ is disintegrated relative to $T_\cap$, or
\item there is $i^* \in I$ such that Kim-independence in $T_{i^*}$ satisfies the $T_\cap$-generic independence theorem, and if $i \neq i^*$ then algebraic closure and forking in $T_i$ agrees with algebraic closure and forking in $T_\cap$.
\end{enumerate}
Then $T_\cup$ is simple.
\end{thm}

\noindent
We give a number of examples in Section~6 which show that the strong assumptions of Theorem~\ref{thm:main2} are sharp (fusions are typically $\mathrm{TP}_2$).
We view $(1)$ as lift of the fact that $\ACFA$ is simple and $(2)$ as a lift of the fact that that the expansion of a simple theory by a generic unary predicate is simple.
%(Both of these results are due to Chatzidakas and Pillay~\cite{Cha-Pi}.)
Case $(2)$ generalizes a result of Tsuboi~\cite{Tsuboi}, see Fact~\ref{fact:tsuboi} below.
The assumptions in $(2)$ are very strong but they are satisfied in more examples then one might expect, in particular the random $n$-hypergraph and the generic tournament.

\medskip \noindent
The main tool required for Theorems~\ref{thm:preservationintro} and \ref{thm:main2} is Theorem~\ref{thm:alex} (following from Proposition~\ref{prop:relaclcomp}).
It is of independent interest.

\begin{thm}
\label{thm:alex}
Suppose $T_\cap$ is stable with weak elimination of imaginaries and $T^*_\cup$ exists.
Let $\sM_\cup \models T^*_\cup$ and $A$ be a subset of $\sM$.
Then
\begin{enumerate}
\item every $L_\cup$-formula $\psi(x)$ is $T_\cup^*$-equivalent to a finite disjunction of formulas
$$\exists y\, \bigwedge_{i\in J} \varphi_i(x,y)$$ where $J\subseteq I$ is finite, $\varphi_i(x,y)$ is an $L_i$-formula for all $i\in J$, and $\bigwedge_{i\in J} \varphi_i(x,y)$ is bounded in $y$,
\item a tuple $b$ is in $\acl_\cup(A)$ if and only if 
\[b \in \acl_{i_n}(\dots (\acl_{i_1}(A))\dots) \text{  for some  } i_1,\ldots,i_n \in I,\]
(i.e. $\acl_\cup(A)$ is the smallest superset of $A$ which is $\acl_i$-closed for all $i \in I$),
\item and if $A$ is $\acl_\cup$-closed then 
$$T_\cup^* \cup \bigcup_{i\in I} \tp_{L_i}(A) \models \tp_{L_\cup}(A).$$
\end{enumerate}
\end{thm}

\noindent
% The full version of Theorem~\ref{thm:alex} does not require that $T_\cap$ be stable, but rather assumes existence of an independence relation on $T_\cap$ with a certain relationship to each $T_i$.
% When $T_\cap$ is stable with weak elimination of imaginaries forking independence satisfies the required properties.
Theorem~\ref{thm:alex} may fail without weak elimination of imaginaries, see Section~\ref{section:blow-up}.
In applications of Theorem~\ref{thm:alex} we can usually avoid assuming weak elimination of imaginaries by passing to the $\eq$ of a structure.

\subsection{Conventions and notation}
All languages and theories are first-order.
Throughout, $I$, $(L_i)_{i \in I}$, $L_\cap$, $L_\cup$, $(T_i)_{i \in I}$, $T_\cap$, $T_\cup$, and $T^*_\cup$ are as in the introduction.
We let $\acl_i$ be algebraic closure in $T_i$ and $\acl_\cup$ be algebraic closure in $T_\cup$.
We let $\monster$ be the monster model.
All models other than the monster are small and all sets of parameters are small.

\medskip\noindent  
We work in a multi-sorted setting.
Let $S$ be the set of sorts of $L$.
Suppose $\sM$ is an 
$L$-structure. We use the corresponding capital letter $M$ to denote the $S$-indexed family $(M_s)_{s \in S}$ of underlying sets of the sorts of $\sM$.
By $A \subseteq M$, we mean $A =(A_s)_{s \in S}$ with $A_s \subseteq M_s$ for each $s \in S$.
If $A \subseteq M$, then a tuple of elements (possibly infinite) in $A$ is a tuple whose each component is in $A_s$ for some $s \in S$. If $x = (x_j)_{j\in J}$ is a tuple of variables (possibly infinite), we let $A^x = \prod_{j\in J} A_{s(x_j)} $ where $s(x_j)$ is the sort of the variable $x_j$.
If $\varphi(x,y)$ is an $L$-formula and $b \in M^y$, we let $\varphi(\sM, b)$ be the set defined in $\sM$ by the $L(b)$-formula $\varphi(x,b)$.
We call such $\varphi(\sM, b)$ a definable set in $\sM$ or an $\sM$-definable set. Hence, ``definable'' means ``definable, possibly with parameters''.
If we wish to exclude parameters, we write ``$\emptyset$-definable''.

\section{Preliminaries}\label{sec: preliminaries}

\subsection{Basic properties of interpolative fusions}
\label{section:first-paper}
We work with the following language-independent definition of $T^*_\cup$, from~\cite{firstpaper}.
We say that $\sM_\cup \models T_\cup$ is \textbf{interpolative}
if whenever $J \subseteq I$ is finite, $X_i$ is an $\sM_i$-definable set for all $i \in J$ and $\bigcap_{i \in J} X_i = \emptyset$, then for each $i \in J$ there is an $\sM_\cap$-definable set $Y_i$ such that $X_i \subseteq Y_i$ and $\bigcap_{i \in J} Y_i = \emptyset$.
If the interpolative $T_\cup$-models form an elementary class then we let $T^*_\cup$ be their theory and refer to $T^*_\cup$ as the \textbf{interpolative fusion} of $(T_i)_{i \in I}$ (over $T_\cap$). We also say ``$T^{*}_\cup$ exists'' if the class of interpolative $T_\cup$-models is elementary. In most of this paper we assume existence of $T^*_\cup$.

\begin{fact}[{\cite[Theorem 2.12]{firstpaper}}] \label{fact: first Theorem}
Suppose each $T_i$ is model-complete. 
Then $\sM_\cup \models T_\cup$ is interpolative if and only if it is existentially closed in the class of  $T_\cup$-models.  
Hence,  $T^*_\cup$ is precisely the model companion of $T_\cup$, if either of these exists.
\end{fact}

\begin{rem}[{\cite[Remark 2.5]{firstpaper}}] \label{rem: robust}
If we change languages in a way that does not change the class of definable sets (with parameters), then the class of interpolative $L_\cup$-structures is not affected. 
In particular:

\begin{enumerate}[leftmargin=*]
\item Any expansion of an interpolative structure by constants (to any of the languages involved) is interpolative.
\item Let $L^{\diamondsuit}_\square$ be an expansion by definitions of $L_\square$ for $\square \in I\cup \{ \cap\}$, $L^{\diamondsuit}_i \cap L^{\diamondsuit}_j = L^{\diamondsuit}_\cap$ for distinct $i$ and $j$ in $I$, and $L^{\diamondsuit}_\cup = \bigcup_{i\in I} L^\diamondsuit_i$ is the resulting expansion by definitions of $L_\cup$. Then any $L_\cup$-structure $\sM_\cup$ has a canonical expansion $\sM_\cup^\diamondsuit$ to an $L_\cup^\diamondsuit$-structure, and $\sM_\cup$ is interpolative  if and only if $\sM_\cup^{\diamondsuit}$ is interpolative.
\item An interpolative $\sM_\cup$-structure remains interpolative after each function symbol $f$ in each of the languages $L_\square$ for $\square\in I\cup \{\cup,\cap\}$ is replaced by a relation symbol $R_f$, interpreted as the graph of the interpretation of $f$ in $\sM_\cup$.
    \item Suppose $\sM_\cup$ is an $L_\cup$-structure.
    Moving to $\sM_\cap^\eq$ involves the introduction of new sorts and function symbols for quotients by $L_\cap$-definable equivalence relations on $M$.
    For all $\square\in I\cup \{\cup,\cap\}$, let $L_\square^{\cap-\eq}$ be the language expanding $L_\square$ produced by adding new symbols for $L_\square$-definable equivalence relations, and let $\sM_\square^{\cap-\eq}$ be the natural expansion of $\sM_\square$ to $L_\square^{\cap-\eq}$.
    Then $\sM_\cup$ is interpolative if and only if $\sM_\cup^{\cap-\eq}$ is interpolative. This follows from the fact that if $X_\square$ is an $\sM_\square^{\cap-\eq}$-definable set in one of the new sorts, corresponding to the quotient of $M^x$ by an $L_\cap$-definable equivalence relation, then the preimage of $X_\square$ under the quotient is $\sM_\square$-definable.
\end{enumerate} 
\end{rem}

\noindent  By Remark~\ref{rem: robust}(2), we are always free to Morleyize each theory $T_i$, thereby reducing to the case where each $T_i$ has quantifier elimination. Theories with quantifier elimination are model-complete, so Fact~\ref{fact: first Theorem} applies in this situation. However, we will not make a global assumption of quantifier elimination or model-completeness, as it is sometimes desirable to study a theory in its native language.

%\noindent If the class of interpolative models of $T_\cup$ is elementary with theory $T_\cup^*$, then we say that $T_\cup^*$ is the {\bf interpolative fusion} (of $(T_i)_{i \in I}$ over $T_\cap$). When $T_i$ is model complete for each $i \in I$, then it follows from Theorem~\cite{firstpaper} that $T^*_\cup$ is just the model companion of $T_\cup$ if either of these exists. At the right momment, we will Morleyize to reduce to the case where $T_i$ is model complete, but we will not do it now as it is sometimes desirable to study a theory in its native language.

%\noindent We also say that ``$T^{*}_\cup$ exists'' if the class of interpolative $T_\cup$-models is elementary; this is not automatic, as the definition of an interpolative structure is not first-order.
% In most of this paper we assume existence of $T^*_\cup$.
%In most of the examples existence of $T^*_\cup$ is either already known or a consequence of Fact~\ref{fact:minh}.

% \begin{rem}\label{rem:eq}
% \textcolor{blue}{describe how we can morleyize and move to eq}
% \end{rem}

%\begin{thm}\label{thm:3}
%Suppose $T_\cap$ is $\aleph_0$-stable and defines multiplicity.
%If each $T_i$ defines induced dimension, then $T^{*}_\cup$ exists.
%\end{thm}

%\subsubsection{Existence conditions and axioms}

\medskip \noindent We recall the following results from~\cite{firstpaper}. Fact~\ref{fact:extension} can be read as saying that $T_\cup^*$ is the model companion of $T_\cup$ ``relative" to $(T_i)_{i \in I}$.

\begin{fact}[{\cite[Theorem 2.7]{firstpaper}}]\label{fact:extension}
Assume $T_\cup^*$ exists.
\begin{enumerate}
    \item For any $\sM_\cup \models T_\cup$ there is $\sM_\cup \subseteq \sN_\cup\models T_\cup^*$ such that $\sM_i\elesub \sN_i$ for all $i \in I$. 
    \item If $\sM_\cup\subseteq \sN_\cup$ are both  $T_\cup^*$-models and $\sM_i\elesub \sN_i$ for all $i\in I$, then $\sM_\cup\elesub\sN_\cup$.
\end{enumerate}
\end{fact}

\noindent 
Fact~\ref{fact:consistency} is a variation of the Robinson joint consistency theorem.

\begin{fact}[{\cite[Corollary 2.3]{firstpaper}}]
\label{fact:consistency}
Suppose $p_\cap(x)$ 
is a complete $L_\cap$-type and  $p_i(x)$ is a complete $L_i$-type extending $p_\cap(x)$ for all $i \in I$.
Then $\bigcup_{i\in I} p_i(x)$ is a consistent (partial) $L_\cup$-type.
\end{fact}

\noindent Note that Fact~\ref{fact:consistency} merely tells us that $\bigcup_{i\in I} p_i(x)$ is consistent, not that it is consistent with a background $L_\cup$-theory (e.g., $T_\cup^*$ or the elementary diagram of a model). We now rectify this situation, under the assumption that $T_\cup^*$ exists.

\begin{prop}\label{prop:realizingtypes}
Assume $T_\cup^*$ exists and $\sM_\cup\models T_\cup^*$. Let $p_\cap(x)$ 
be a complete $L_\cap(M)$-type, and for all $i\in I$, let $p_i(x)$ be a complete $L_i(M)$-type such that $p_\cap(x)\subseteq p_i(x)$. Then $\bigcup_{i\in I} p_i(x)$ is realized in an elementary extension of $\sM_\cup$.
\end{prop}
\begin{proof}
By Fact~\ref{fact:consistency}, $\bigcup_{i\in I} p_i(x)$ is consistent. Suppose it is realized by $a$ in a model $\sN_\cup$. Since each $p_i$ is a complete $L_i(M)$-type, $\ediag(\sM_i)\subseteq p_i(x)$, so we may assume $\sM_\cup\subseteq \sN_\cup$ and $\sM_i\elesub \sN_i$ for all $i\in I$. In particular, $\sN_\cup\models T_\cup$, but we may not have $\sN_\cup\models T_\cup^*$.

By Fact~\ref{fact:extension}(1), there exists $\sM_\cup'\models T_\cup^*$ such that $\sN_\cup\subseteq \sM_\cup'$ and $\sN_i\elesub\sM_i'$ for all $i\in I$. In particular, $a$ satisfies $\bigcup_{i\in I} p_i(x)$ in $\sM_\cup'$. Also, $\sM_i\elesub \sM_i'$ for all $i\in I$, so by Fact~\ref{fact:extension}(2), $\sM_\cup\elesub \sM_\cup'$. 
\end{proof}

\noindent Proposition~\ref{prop:realizingtypes} is a useful tool for realizing types, but it can only be applied to those $L_\cup$-types which are entailed by a union of $L_i$-types. We identify complete $L_\cup$-types with this property in Section~\ref{sec:morleyization}. 

\subsection{An existence result}
\label{section:existence}
Much of \cite{firstpaper} is devoted to developing general conditions ensuring existence of $T^*_\cup$.
For the examples in the present paper, we will only need to use one result, Fact~\ref{fact:minh} below.
The special case of Fact~\ref{fact:minh} when $T_\cap$ is the theory of an infinite set is due to Winkler~\cite{Winkler}.

\begin{fact}
\label{fact:minh}
Suppose $T_\cap$ is complete, $\aleph_0$-stable, and $\aleph_0$-categorical.
Suppose one of the following holds
\begin{enumerate}
\item $T_i^{\eq}$ eliminates $\exists^\infty$ for all $i \in I$, or
\item $T_\cap$ weakly eliminates imaginaries and each $T_i$ eliminates $\exists^\infty$.
\end{enumerate}
Then $T^{*}_\cup$ exists.
\end{fact}

\noindent
In \cite{firstpaper} we obtain an explicit axiomization for $T^*_\cup$ under the assumptions of Fact~\ref{fact:minh}.
This axiomization is $\forall\exists$ when each $T_i$ is model complete.
We will not use this in the present paper.
If $T$ is intepretable in the theory of equality then $T$ $\aleph_0$-categorical and $\aleph_0$-stable, and $T^\eq$ eliminates $\exists^\infty$.
Corollary~\ref{cor:minh} is a useful special case of Fact~\ref{fact:minh}.

\begin{cor}
\label{cor:minh}
If $T_\cap$ is interpretable in the theory of equality and each $(T_i)^\eq$ eliminates $\exists^\infty$ then $T^*_\cup$ exists.
If each $T_i$ is interpretable in the theory of equality the $T^*_\cup$ exists.
\end{cor}

\subsection{Interpretations}
When dealing with examples, we will need to use some easy facts about interpretations.
 Let $T$ be an $L$-theory and $T'$ be an $L'$-theory.
 An \textbf{interpretation} of $T$ in $T'$, $F\colon T'\rightsquigarrow T$, consists of the following data: 
\begin{enumerate}
    \item For every sort $s$ in $L$, an $L'$-formula $\varphi_{s}(x_{s})$ and an $L'$-formula $E_{s}(x_{s},x^*_{s})$. 
    \item For every relation symbol $R$ in $L$ of type $(s_1,\dots,s_n)$, an $L'$-formula $\varphi_{R}(x_{s_1},\dots,x_{s_n})$.
    \item For every function symbol $f$ in $L$ of type $(s_1,\dots,s_n)\to s$, an $L'$-formula $\varphi_{f}(x_{s_1},\dots,x_{s_n},x_{s})$. 
\end{enumerate}
We then require that for every model $\sM'\models T'$, the formulas above define an $L$-structure $\sM\models T$ in the natural way. See~\cite[Section 5.3]{Hodges} for details.  For every sort $s$ in $L$, the underlying set $M_{s}$ of the $s$ sort in $\sM$ is the quotient of $\varphi_{s}(\sM')$ by the equivalence relation defined by $E_{s}$. We write $\pi_{s}$ for the surjective quotient map $\varphi_{s}(\sM')\to M_{s}$.
We sometimes denote $\sM$ by $F(\sM')$.

\medskip \noindent An interpretation $F\colon T'\rightsquigarrow T$ is an \textbf{existential  interpretation} if for each sort $s$ in $L$, the $L'$-formula $\varphi_{s}(x_{s})$ is $T'$-equivalent to an existential formula, and all other formulas involved in the interpretation and their negations (i.e., the formulas $E_{s}$, $\neg E_{s}$, $\varphi_{R}$, $\neg\varphi_{R}$, $\varphi_{f}$,  and $\neg\varphi_{f}$) are also $T'$-equivalent to existential formulas.
See \cite[Corollary 2.16]{firstpaper} for a proof of Fact~\ref{fact:existential}.

\medskip \noindent A {\bf bi-interpretation} $(F,G, \eta, \eta')$ between $T$ and $T'$ consists of an interpretation $F\colon T'\rightsquigarrow T$ and an interpretation  $G\colon T\rightsquigarrow T'$, together with $L$-formulas and $L'$-formulas defining for each $\sM \models T$ and each $\sN' \models T'$ isomorphisms $$\eta_{\sM}: \sM \to F(G(\sM)) \quad \text{and}  \quad \eta'_{\sN'}: \sN' \to G(F(\sN')).$$ The bi-interpretation is  \textbf{existential} if $F$ and $G$ are each existential interpretations, the formulas defining $\eta$ are $T$-equivalent to existential formulas, and the formulas defining $\eta'$ are $T'$-equivalent to existential formulas. 

\begin{fact}
\label{fact:existential}
Suppose $T$ and $T'$ $\forall\exists$-axiomatizable theories which are existentially bi-interpretable.
Then $T$ has a model companion if and only if $T'$ has a model companion. Moreover, the model companions are bi-interpretable, if they exist. 
\end{fact}

\noindent When $T^*$ is the model companion of a $\forall\exists$-axiomatizable theory $T$, we will sometimes wish to show that $T^*$ is bi-interpretable with an interpolative fusion $T^*_\cup$. We will do this by showing that $T$ is existentially bi-interpretable with a union of model-complete theories $T_\cup$ and then applying Facts~\ref{fact: first Theorem} and ~\ref{fact:existential}.

\subsection{Flat formulas and $\sK$-completeness}\label{sec:Kcomplete}

We recall the notion of flat formula from~\cite{firstpaper}. 

\medskip \noindent An \textbf{atomic flat formula} has the form $x = y$,  $R(x_1,\ldots,x_n)$, or $f(x_1,\ldots,x_n) = y$, where $R$ is an $n$-ary relation symbol and $f$ is an $n$-ary function symbol. Here  $x,y,x_1,\dots,x_n$ are single variables, which need not be distinct. A \textbf{flat literal} is an atomic flat formula or the negation of an atomic flat formula. A \textbf{flat formula} is a conjunction of finitely many flat literals. An \textbf{E$\flat$-formula} is a formula of the form $\exists y\, \varphi(x,y)$, where $\models \forall x\, \exists^{\leq 1} y\, \varphi(x,y)$ and $\varphi(x,y)$ is flat. Here $x$ and $y$ are finite tuples of variables, which may be empty. 

%\begin{rem}\label{rem:flat-closure}
%The class of E$\flat$-formulas is closed (up to equivalence) under finite conjunction: the conjunction of the E$\flat$-formulas $\exists y_1\, \varphi_1(x,y_2)$ and $\exists y_2\, \varphi_2(x,y_2)$ is equivalent to the E${\flat}$-formula \[\exists y_1y_2\, (\varphi_1(x,y_1)\land \varphi_2(x,y_2)).\]
%\end{rem}

\medskip\noindent The following result is an easy refinement of  Theorem 2.6.1 in~\cite{Hodges}. 
Note Hodges uses the term ``unnested'' instead of ``flat''.
%Fact~\ref{fact:eflat} is elementary.

\begin{fact}[{\cite[Lemma 2.9 and Corollary 2.10]{firstpaper}}] \label{fact:eflat}
Every atomic formula and every negated atomic formula is logically equivalent to an E$\flat$-formula. Every quantifier-free formula is logically equivalent to a finite disjunction of E$\flat$-formulas.
\end{fact}

\noindent The \textbf{flat diagram} $\fdiag(\sA)$ of an $L$-structure $\sA$ is the set of all flat literal $L(A)$-sentences true in $\sA$. 
The flat diagram is logically equivalent to the ordinary Robinson diagram, so we have the following: If $\sB$ is an $L(A)$-structure, then $\sB\models \fdiag(\sA)$ if and only if the obvious map $\sA\to \sB$ is an embedding~\cite[Lemma 1.4.3]{Hodges}.

\medskip \noindent 
We describe a slight generalization of standard results on model completeness.
Let $T$ be an $L$-theory, and let $\sK$ be a class of pairs $(\sA,\sM)$, where $\sM\models T$ and $\sA$ is a substructure of $\sM$. We say that $T$ is \textbf{$\sK$-complete} if for all $(\sA,\sM)\in \sK$, every embedding from $\sA$ to another $T$-model is partial  elementary. That is, if $f\colon \sA\to \sN$ is an embedding and $\sN\models T$, then $\sM\models \varphi(a)$ if and only if $\sN\models \varphi(f(a))$ for any formula $\varphi(x)$ and any $a\in A^x$.

\begin{rem}\label{rem:Kcomp}
The terminology $\sK$-complete comes from the following equivalent definition: $T$ is $\sK$-complete if and only if for all $(\sA,\sM)\in \sK$, \[T\cup \fdiag(\sA)\models \Th_{L(A)}(\sM),\] i.e., $T\cup \fdiag(\sA)$ is a complete $L(A)$-theory. The equivalence follows immediately from the fact that an $L(A)$-structure $\sN$ satisfies $\fdiag(\sA)$ if and only if the obvious map $\sA\to\sN$ is an embedding. 
\end{rem}

\noindent We say the class of $T$-models has the \textbf{$\sK$-amalgamation property} if whenever $(\sA,\sM)\in \sK$, $\sN \models T$, and $f\colon \sA\to \sN$ is an embedding, then there is an elementary extension $\sN\elesub \sN'$ and an elementary embedding $f'\colon \sM\to \sN'$ such that $f'|_A = f$, i.e., the following diagram commutes:
\[
\xymatrix{
\sM\ar[r]_{f'} & \sN'\\
\sA\ar[r]^f\ar[u]^{\subseteq} & \sN\ar[u]_{\elesub}
}
\]
If, in the situation above, we can choose $\sN'$ and $f'$ with the further condition that \[f'(M)\cap N = f'(A) = f(A),\] then the class of $T$-models has the \textbf{disjoint $\sK$-amalgamation property}.

\begin{prop}\label{prop:kcomp}
The theory  $T$ is $\sK$-complete if and only if the class of $T$-models has the $\sK$-amalgamation property. Further, if $T$ is $\sK$-complete, then $\sA$ is algebraically closed in $\sM$ for all $(\sA,\sM)\in \sK$ if and only if the class of $T$-models has the disjoint $\sK$-amalgamation property.
\end{prop}

\begin{proof}
We prove the first equivalence. Suppose $T$ is $\sK$-complete. The $\sK$-amalgamation property follows from~\cite[Theorem 6.4.1]{Hodges}. 

Conversely, suppose the class of $T$-models has the $\sK$-amalgamation property. 
If $\sM$ and $\sN$ are $T$-models, $\sA\subseteq \sM$ is in $\sK$, and $f\colon \sA\to \sN$ is an embedding, then there is an elementary extension $\sN\elesub \sN'$ and an elementary embedding $f'\colon \sM\to \sN'$ such that $f'|_A = f$. For any $L$-formula $\varphi(x)$ and  $a\in A^x$, $\sM\models \varphi(a)$ if and only if $\sN'\models \varphi(f'(a))$ if and only if $\sN\models \varphi(f(a))$. So $f$ is partial elementary. 
Thus $T$ is $\sK$-complete.

Now, assuming $T$ is $\sK$-complete, we prove the second equivalence. If every structure in $\sK$ is algebraically closed, then the class of $T$-models has the disjoint $\sK$-amalgamation property, by~\cite[Theorem 6.4.5]{Hodges}. 

Conversely, suppose the class of $T$-models has the disjoint $\sK$-amalgamation property. 
Assume towards a contradiction that $(\sA, \sM)\in \sK$ and $A$ is not algebraically closed in $\sM$. Then there is some $c\in M\setminus A$ such that $\tp(c/A)$ has exactly $k$ realizations $c_1,\dots,c_k$ in $M\setminus A$.
Taking $\sN = \sM$ and $f = \id_A$ in the disjoint $\sK$-amalgamation property, there is an elementary extension $\sM \elesub \sM'$ and an elementary embedding $f'\colon \sM\to \sM'$ which is the identity on $A$ and satisfies  $f'(M)\cap M = A$. 
Then $\tp(c/A)$ has $2k$ distinct realizations $c_1,\dots,c_k,f'(c_1),\dots,f'(c_k)$ in $\sM'$, contradiction.
\end{proof}

\noindent Suppose $T$ is $\sK$-complete. 
If $\sK$ is the class of all pairs $(\sM,\sM)$ such that $\sM\models T$, then $T$ is \textbf{model-complete}. 
We say $T$ is \textbf{substructure-complete} if $\sK$ is the class of all pairs $(\sA,\sM)$ such that $\sA$ is a substructure of $\sM$.
If $\cl$ is a closure operator on $T$-models and $\sK$ is the class of all pairs $(\sA,\sM)$ such that $\sA$ is a $\cl$-closed substructure of $\sM$, i.e., $\cl(A)= A$, then we say $T$ is \textbf{$\cl$-complete}. Similarly, we refer to the \textbf{(disjoint) $\cl$-amalgamation property}.

%\noindent We recall some classical facts about model-completeness and model companions.

%\begin{fact}[\cite{Hodges}, Theorem 6.5.9, Exercise 6.5.5]\label{fact:inductive}
%The following are equivalent: 
%\begin{enumerate}
%\item $T$ admits an %$\forall\exists$-axiomatization. 
%\item The class of $T$-models is closed under unions of chains.
%\item The class of $T$-models is closed under directed colimits (in the category of $L$-structures and embeddings).
%\end{enumerate}
%\end{fact}
%\noindent If  one of the equivalent conditions above are satisfied, we say that $T$ is \textbf{inductive}.

%\begin{fact}[\cite{Hodges},  Theorem 8.3.3]\label{fact:mcinductive}
%Every model-complete theory is inductive.
%\end{fact}

%\noindent An $L$-theory $T^*$ is a \textbf{model companion} of $T$ if $T^*$ is model-complete, every $T$-model embeds into a $T^*$-model, and every $T^*$-model embeds into a $T$-model. 

%\begin{fact}[\cite{Hodges}, Theorem 8.2.1, Theorem 8.3.6]\label{fact:modelcompanion}
%Suppose $T$ is inductive.
%Then:
%\begin{enumerate}
%\item Every $T$-model embeds into an existentially closed $T$-model.
%\item $T$ has a model companion if and only if the class of existentially closed $T$-models is elementary. 
%\item If $T$ has a model companion $T^*$, then $T^*$ is the theory of existentially closed $T$-models.
%\end{enumerate}
%\end{fact}

\medskip \noindent Model-completeness has a syntactic equivalent: every $L$-formula is $T$-equivalent to an existential (hence also a universal) formula ~\cite[Theorem 8.3.1(e)]{Hodges}. Substructure-completeness also has a syntactic equivalent: quantifier elimination. This follows from~\cite[Theorem 8.4.1]{Hodges} and Proposition~\ref{prop:kcomp} above.

\medskip \noindent Many of the theories we consider are $\acl$-complete.
Unfortunately, there does not seem to be a natural syntactic equivalent to $\acl$-completeness. For this reason, we introduce a slightly stronger notion, $\bcl$-completeness, which does have a syntactic equivalent.

\medskip \noindent An $L$-formula $\varphi(x,y)$ is \textbf{bounded in $y$} with bound $k$  (with respect to $T$) if \[T\models \forall x\, \exists^{\leq k} y\, \varphi(x,y).\] A formula $\exists y\, \psi(x,y)$ is \textbf{boundedly existential (b.e.)} (with respect to $T$) if $\psi(x,y)$ is quantifier-free and bounded in $y$. We allow $y$ to be the empty tuple of variables, so every quantifier-free formula is b.e.\ (with bound $k=1$, by convention). The E$\flat$-formulas introduced above are also b.e.\ with bound $k=1$ with respect to the empty theory.

\medskip\noindent Suppose $\sM\models T$ and $A\subseteq \sM$. The \textbf{boundedly existential algebraic closure} of $A$ in $\sM$, denoted $\bcl(A)$, is the set of all $b\in M$ such that $\sM\models \varphi(a,b)$ for some $a\in A^x$ and some $L$-formula $\varphi(x,y)$ which is b.e.\ and bounded in $y$. It follows that $\varphi(a,y)$ is algebraic, so $\bcl(A)\subseteq \acl(A)$. It can also be easily verified that $\bcl$ is a closure operator and if $A \subseteq \sM$ then $\langle A\rangle \subseteq \bcl(A)$; see Lemma~\ref{lem: bclisclosureoperator} for details.

\begin{rem}
Every model is $\acl$-closed, every $\acl$-closed set is $\bcl$-closed, and every $\bcl$-closed set is a substructure, therefore:
\[\text{QE $\Leftrightarrow$ substructure-complete}\Rightarrow \text{$\bcl$-complete}\Rightarrow \text{$\acl$-complete}\Rightarrow \text{model-complete}.\]
\end{rem}

\noindent Theorem~\ref{thm:bclacl} clarifies the relationship between $\acl$- and $\bcl$-completeness and provides the promised syntactic equivalent to $\bcl$-completeness. The proof is somewhat involved and may be of independent interest, so we delay it to Appendix~\ref{app:Kcomp}.

\begin{thm}\label{thm:bclacl}
The following are equivalent:
\begin{enumerate}
\item Every $L$-formula is $T$-equivalent to a finite disjunction of b.e.\ formulas.
\item $T$ is $\acl$-complete and $\acl = \bcl$ in $T$-models.
\item $T$ is $\bcl$-complete.
\end{enumerate}
\end{thm}

\subsection{Stationary and extendable independence relations}\label{sec:stationary}

\noindent In this section, $T$ is a complete $L$-theory, $L'$ is a  language extending $L$ such that $L$ and $L'$ have the same set of sorts, and $T'$ is a complete $L'$-theory extending $T$.
Let $\monster'$ be a monster model of $T'$ and $\monster$ be the $L$-reduct of $\monster'$, so $\monster$ is a monster model of $T$. We isolate some properties of forking independence in stable theories that weakly eliminate imaginaries; these properties will be necessary ingredients for the proofs of preservation of $\acl$- and $\bcl$-completeness in Section~\ref{sec:analysis}. 

\medskip \noindent Let $\ind$ be a ternary relation on small subsets of $\monster$. 
We consider the following properties that $\ind$ may satisfy. 
The first three are specific to $T$, while the fourth concerns the relationship between $T$ and $T'$. We let  $A$, $B$, and $C$ range over arbitrary small subsets of $\monster$.

\begin{enumerate}
\item \textbf{Invariance:} If $\sigma$ is an automorphism of $\monster$, then $A\ind_C B$ if and only if $\sigma(A)\ind_{\sigma(C)} \sigma(B)$.
\item \textbf{Algebraic independence:} If $A\ind_C B$, then $\acl(AC)\cap \acl(BC) = \acl(C).$
\item \textbf{Stationarity over algebraically closed sets:} If $C = \acl(C)$, $\tp_L(A/C) = \tp_L(A^*/C)$, $A\ind_C B$, and $A^*\ind_C B$, then $\tp_L(A/BC) = \tp_L(A^*/BC)$.
\item \textbf{Full existence over algebraically closed sets in $T'$:} If $C = \acl'(C)$, then there exists $A^*$ with $\tp_{L'}(A^*/C) = \tp_{L'}(A/C)$, and $A^*\ind_C B$ in $\monster$.
\end{enumerate}

\noindent We say $\ind$ is a \textbf{stationary independence relation} in $T$ if it satisfies invariance, algebraic independence, and stationarity over algebraically closed sets. 
%In particular, a stationary independence relation identifies, for every $L$-type $p(x)\in S_x(C)$ with $C = \acl(C)$ and every set $B$, a unique ``independent'' extension of $p(x)$ in $S_x(BC)$. 
We say $\ind$ is \textbf{extendable} (to $T'$) if it satisfies full existence  over algebraically closed sets in $T'$.

%When the $T_i$ are complete, we say $T_\cap$ admits 
 %a  {\bf stationary and extendable independence} if there is a stationary independence relation in $T_\cap$ which satisfies full existence over algebraically closed sets in each $T_i$; when the $T_i$ are incomplete, we mean that this property holds whenever we pass to completions of all the $T_i$ which extend a common completion of $T_\cap$. 

\medskip\noindent Our definition of a stationary independence relation differs from those used elsewhere, e.g., in~\cite{TZUrysohn}. 
Most natural stationary independence relations satisfy additional axioms (symmetry, monotonicity, etc.).
We only require the axioms above.

\medskip \noindent The main example of a stationary and extendable independence relation that the reader should keep in mind is forking independence $\eind[f]$ in a theory $T$ which is stable with weak elimination of imaginaries. We give a proof of Proposition~\ref{prop:stationarity} in Appendix~\ref{app:stationary}. Note that there are no hypotheses on $T'$.

\begin{prop}\label{prop:stationarity}
Suppose $T$ is stable with weak elimination of imaginaries. Then $\eind[f]$ is a stationary and independence relation in $T$ which is extendable to $T'$.
\end{prop}

\noindent The next example shows, there are also non-trivial examples of stationary independence relations in unstable theories, which may or may not be extendable.

\begin{example}\label{ex:rg}
Suppose $L$ contains a single binary relation $E$, and $T$ is the theory of the random graph (the \Fraisse limit of the class of finite graphs). Define:
\begin{align*}
A\eind[E]_C B &\iff A\cap B\subseteq C \text{ and } aEb\text{ for all $a\in A\setminus C$ and $b\in B\setminus C$}\\
A\eind[\not E]_C B &\iff A\cap B\subseteq C \text{ and } \lnot aEb\text{ for all $a\in A\setminus C$ and $b\in B\setminus C$.}
\end{align*}
Both $\eind[E]$ and $\eind[\not E]$ are stationary independence relations in $T$.

Now let $L' = \{E,P\}$, where $P$ is a unary predicate, and let $T'$ be the theory of the \Fraisse limit of the class of finite graphs with a predicate $P$ naming a clique. $T'$ extends $T$ and has quantifier elimination, and $\acl_{L'}(A) = A$ for all sets $A$.

Then $\eind[E]$ is extendable to $T'$. Indeed, for any $A$, $B$, and $C$, let $p(x) = \tp_{L'}(A/C)$, where $x = (x_a)_{a\in A}$ is a tuple of variables enumerating $A$. The type $$p(x)\cup \{x_a E b\mid a\in A\setminus C\text{ and }b\in B\setminus C\}$$ is consistent, and for any realization $A^*$ of this type, we have $A^*\eind[E]_C B$ in $\monster$.

On the other hand, let $a$ and $b$ be any two elements of $\monster'$ satisfying $P$. Then for any realization $a^*$ of $\tp_{L'}(a/\emptyset)$, we have $P(a^*)$, so $a^*Eb$, and $a^*\enind[\not E]_\emptyset b$ in $\monster$. So $\eind[\not E]$ is not extendable to $T'$.
\end{example}

\subsection{$\NSOP_1$ theories}
\label{sec:NSOP1basic}
%\noindent
%In this section we summarize what we need to know about the relevant classification theoretic properties.
%We omit the well-known definitions of stability and $\NIP$ (see, for example, \cite{Hodges}).
In this section, we summarize the necessary background on $\NSOP_1$ theories. $T$ is a complete $L$-theory, $\monster$ is a monster model of  $T$, and $\sM \prec \monster$ is a small submodel with underlying set $M$.

%\medskip\noindent
%Let $\varphi(x;y)$ be a formula.
%Then $\varphi(x;y)$ has \textbf{TP} (relative to $T$) if there are $k \in \nn$ and tuples $(b_\eta)_{ \eta \in \omega^{<\omega}}$ in $\monster$ such that $\{ \varphi(x;b_{\eta^\frown n}) \mid n \in \omega \}$ is $k$-inconsistent for all $\eta \in \omega^{<\omega}$ and $\{ \varphi(x;b_{\eta|_n} \mid n \in \omega \}$ is consistent for all $\eta \in \omega^\omega$.
%The theory $T$ is \textbf{simple} if no formula has $\mathrm{TP}$ relative to $T$.

\medskip\noindent
Let $\leq$ be the tree partial order on $2^{< \omega}$, and let $\nu^\frown \eta$ be the usual concatenation of $\nu,\eta \in 2^{<\omega}$.
The formula $\varphi(x;y)$ has \textbf{SOP$_1$} (relative to $T$) if there exist $(b_\eta)_{\eta\in 2^{<\omega}}$ in $\monsterset^y$ such that:
\begin{enumerate}
    \item For all $\nu,\eta \in 2^{< \omega}$, if $\nu^\frown 0 \leq \eta$, then $\{\varphi(x;b_{\eta}),\varphi(x;b_{\nu^\frown 1})\}$ is inconsistent.
    \item For all $\sigma \in 2^{\omega}$, $\{\varphi(x; a_{\sigma|_n})\mid n\in \omega\}$ is consistent.
\end{enumerate}
The theory $T$ is $\NSOP_1$ if no formula has $\SOP_1$ relative to $T$. An incomplete theory is $\NSOP_1$ if each of its completions is $\NSOP_1$.

%\medskip\noindent
%Finally $\varphi(x;y)$ has \textbf{TP}$_2$ (relative to $T$) if there is $k \in \omega$ and tuples $(b_{ij})_{i,j \in \omega}$ in $\monster$ such that
%\begin{enumerate}
%\item $\{ \varphi(x; b_{i,j}) \mid j \in \omega \}$ is $k$-inconsistent for all $i \in \omega$, and
%\item $\{ \varphi(x; b_{if(i)} \mid i \in \omega \}$ is consistent for any function $f : \omega \to \omega$.
%\end{enumerate}
%\noindent
%The theory $T$ has $\pmb{\mathrm{NTP}_2}$ if no formula has $\mathrm{TP}_2$ relative to T.

%\medskip\noindent
%An incomplete theory is simple ($\NSOP_1$, $\NTP_2$) if each of its completions is simple ($\NSOP_1$, $\NTP_2$).

%\medskip\noindent
%The first claim of Fact~\ref{fact:nsop1-dicho} is Corollary 8.5 of~\cite{Kim}, but it also follows immediately from the well-known facts that a non-simple theory has TP$_1$ or TP$_2$~\cite{ShelahCT}, any $\NSOP_1$ theory is $\NSOP_2$~\cite{DS}, and NTP$_1$ is equivalent to $\NSOP_2$~\cite{KimKim}.
%The second claim follows from Shelah's theorem that stability is equivalent to the conjunction of $\NIP$ and the failure of the strict order property~\cite{ShelahCT} together with the fact that every $\NSOP_1$ theory does not have the strict order property.

%\begin{fact}
%\label{fact:nsop1-dicho}
%If $T$ is $\NSOP_1$ then $T$ is either simple or $\mathrm{TP}_2$.
%So if $T$ is $\NSOP_1$ and $\NIP$ then $T$ is stable.
%\end{fact}

\medskip\noindent
This ``negative" definition of $\NSOP_1$ is due to Dzamonja and Shelah~\cite{DS}. We will find it more convenient to work with a ``positive" characterization of $\NSOP_1$, due to Chernikov, Ramsey, and Kaplan, in terms of the relation of Kim-independence.
%Instead we will use Kim-independence, defined below.

%\medskip \noindent We recall the definition of Kim-independence, due to Ramsey, and review some foundational results, most of which are due to Kaplan and Ramsey.
 
\medskip \noindent A global type $q \in S_y(\monster)$ is \textbf{$M$-invariant} if for any formula $\psi(y,z)$ and any $c,c'\in \monsterset^z$ with $\tp(c/M)=\tp(c'/M)$, we have $\psi(y,c)\in q$ if and only if $\psi(y,c')\in q$.
A sequence $(b_i)_{k \in \omega}$ is a \textbf{Morley sequence} for $q$ over $M$ if $b_k$ realizes the restriction of $q$ to $Mb_0\dots b_{k-1}$ for all $i$.
Suppose $\varphi(x,b)$ is a formula with $b\in \monsterset^y$ and $q\in S_y(\monster)$ is a global $M$-invariant type extending $\tp(b/M)$. 
Then $\varphi(x,b)$ \textbf{$q$-divides over $M$} if $\{\varphi(x,b_i)\mid i\in \omega\}$ is inconsistent for some (equivalently any)  Morley sequence $(b_i)_{i\in \omega}$ for $q$ over $M$.

\medskip\noindent A formula $\varphi(x,b)$ \textbf{Kim-divides} over $M$ if it $q$-divides over $M$ for some global $M$-invariant type $q$ extending $\tp(b/M)$, and $\varphi(x,b)$ \textbf{Kim-forks} over $M$ if it implies a disjunction of formulas which Kim-divide over $M$.
We say that $A$ is \textbf{Kim-independent} from $B$ over $M$ when no formula in $\tp(A/MB)$ Kim-forks over $M$.
We write $A\eind[K]_M B$ when $A$ is Kim-independent from $B$ over $M$.

\medskip \noindent These definitions are made over a model $\sM$, rather than over an arbitrary set of parameters $C$, because a type over $C$ may fail to extend to any global $C$-invariant type. 

\medskip \noindent We can now state the axiomatic characterization of $\NSOP_1$ and Kim-independence. An earlier version of this criterion appeared in~\cite{CR}.

\begin{fact}[\cite{Kim} Theorem 9.1]\label{fact:char}
Suppose $\ind$ satisfies the following for all $A,A',B,B'$ and all models $\sM,\sM'$:
\begin{enumerate}
\item \textbf{Invariance}: If $A\ind_M B$ and $MAB \equiv M'A'B'$, then $A'\ind_{M'} B'$.
\item \textbf{Existence}: $A\ind_M M$
\item \textbf{Monotonicity}: If $A\ind_M B$ and $A'\subseteq A$ and $B'\subseteq B$, then $A'\ind_M B'$.
\item \textbf{Symmetry}: If $A\ind_M B$, then $B\ind_M A$.
\item \textbf{The independence theorem}: If $A\equiv_M A'$, $A\ind_M B$, $A'\ind_M C$, and $B\ind_M C$, then there exists $A''$ such that $A'' \equiv_{MB} A$, $A''\equiv_{MC} A'$, and $A''\ind_M BC$.
\item \textbf{Strong finite character}: If $A\nind_M B$, then there is a formula $\varphi(x,b,m)\in \tp(A/MB)$ such that for any $c$ such that $\monster\models \varphi(c,b,m)$, we have $c\nind_M b$.
\end{enumerate}
Then $T$ is $\NSOP_1$. Suppose $\ind$ additionally satisfies: 
\begin{enumerate}\setcounter{enumi}{6}
\item \textbf{Witnessing}: If $A\nind_M B$, then there is a formula $\varphi(x,b,m)\in \tp(A/MB)$ which Kim-forks over $M$.
\end{enumerate}
Then $\ind_M = \eind[K]_M$ for all $\sM$.
\end{fact}

\begin{rem}
If $T$ is $\NSOP_1$, then Kim-independence satisfies all of the properties in Fact~\ref{fact:char}. 
The only nontrivial properties are symmetry (\cite{Kim} Theorem 5.16) and the independence theorem (\cite{Kim} Theorem 6.5).
\end{rem}

%\noindent Let $\ind$ be a ternary relation on subsets of $\monster$. We say that $A_1$, $A_2$, $A_3$ is an \textbf{$\ind$-independent triple over $B$} if $A_1\ind_B A_2A_3$, $A_2\ind_B A_1A_3$, and $A_3\ind_B A_1A_2$. In contrast to the situation  in simple theories, the  conclusion $A''\eind[K]_M BC$ in the conclusion of the independence theorem for $\NSOP_1$ theories does not imply that $A''$, $B$, $C$ is an $\eind[K]$-independent triple over $M$. However, this can be arranged. 

%\begin{fact}[Symmetric  independence theorem, \cite{KRExp} Theorem 2.13]
%Assume $T$ is $\NSOP_1$. If  $A\ind_M B$, $A'\ind_M C$, $B\ind_M C$, and $A\equiv_M A'$, then there exists $A''$ such that $A'' \equiv_{MB} A$, $A''\equiv_{MC} A'$, and $A''$, $B$, $C$ is an $\ind[K]$-independent triple over $M$.
%\end{fact}

\noindent Fact~\ref{fact:kim} is a version of Kim's lemma for $\NSOP_1$ theories.
Kim's lemma was originally proven for forking in simple theories~\cite{Kim1998}. 

\begin{fact}[\cite{Kim} Theorem 3.15]\label{fact:kim}
Assume $T$ is $\NSOP_1$. If a formula $\varphi(x,b)$  $q$-divides over $M$ for \emph{some} global $M$-invariant type $q$ extending $\tp(b/M)$, then $\varphi(x,b)$ $q$-divides for \emph{every} global $M$-invariant type $q$ extending $\tp(b/M)$.
\end{fact}

\noindent A consequence of Fact~\ref{fact:kim} is that Kim-forking and Kim-dividing agree in $\NSOP_1$ theories.

\begin{fact}[\cite{Kim} Proposition 3.19]
\label{fact:kimforking}
Assume $T$ is $\NSOP_1$. A formula $\varphi(x,b)$ Kim-forks over $M$ if and only  if it Kim-divides over $M$.
\end{fact}

\noindent Just like forking independence in arbitrary theories, Kim-independence in $\NSOP_1$ theories is blind to algebraic closures.

\begin{fact}[\cite{Kim} Corollary 5.17]\label{fact:kimacl}
Assume $T$ is $\NSOP_1$.
Then for all sets $A$ and $B$ and all  models $\sM$, $ A\eind[K]_M B$ if and only if $\acl(MA) \eind[K]_M \acl(MB).$
\end{fact}

\noindent Every simple theory is $\NSOP_1$, and in a simple theory, Kim-independence $\eind[K]$ agrees with forking independence $\eind[f]$. Fact~\ref{fact:kimsimple} below, which follows from Propositions 8.4 and 8.8 of~\cite{Kim}, characterizes the simple theories among $\NSOP_1$ theories.
%This will be crucial below as we our results on preservation of simplicity are special cases of our result on preservation of $\NSOP_1$.
We say  Kim-independence satisfies \textbf{base monotonicity over models} if $a\eind[K]_M Nb$ implies $a\eind[K]_N b$ for all $\sM \preceq \sN$.

\begin{fact}\label{fact:kimsimple}
Assume $T$ is $\NSOP_1$.
Then the following are equivalent:
\begin{enumerate}
\item $T$ is simple,
\item $\eind[f]_M = \eind[K]_M$ for all models $\sM$,
%\item $T$ is NTP$_2$,
\item Kim-independence satisfies base monotonicity over models.
\end{enumerate}
\end{fact} 

\subsection{Independence and reducts}\label{sec:indreduct}
In this section, let $L\subseteq L'$ be languages, $T$ a complete $L$-theory, $T'$ a complete $L'$-theory extending $T$, $\monster'$ a monster model of $T'$,  and $\monster$ its reduct to  $L$. We consider the relationship between notions of independence in  $\monster'$ and $\monster$. 

\medskip \noindent 
It is not clear from the definition that Kim-independence is preserved under reducts, since the property of being an $M$-invariant type is not preserved under reducts in general.
However, when $T'$ is $\NSOP_1$, Fact~\ref{fact:kim} shows that Kim-dividing is always witnessed by $q$-dividing for a global type $q$ which is finitely satisfiable in $M$, and this property is preserved under reducts.
This gives us Lemma~\ref{lem:kimreduct}. 

\begin{lem}
\label{lem:kimreduct}
If $T'$ is $\NSOP_1$, then:
\begin{enumerate}
\item $T$ is $\NSOP_1$.
\item Let $\sM\prec\monster'$, and let $\varphi(x,b)$ be an $L$-formula. Then $\varphi(x,b)$ Kim-divides over $M$ in $\monster$ if and only if it Kim-divides over $M$ in $\monster'$.
\item Kim-independence is preserved by reducts: if $A\eind[K]_M B$ in $\monster'$, then also $A\eind[K]_M B$ in $\monster$.
\end{enumerate}
\end{lem}
\begin{proof}
For (1), the fact that $\NSOP_1$ is preserved by reducts is clear from the definition: any formula with SOP$_1$ relative to $T$ also has SOP$_1$ relative to $T'$.

For (2), fix a global $L'$-type $q'$ extending $\tp_{L'}(b/M)$, which is finitely satisfiable in $M$ (hence $M$-invariant). Let $(b_i)_{i\in \omega}$ be a Morley sequence for $q'$ over $M$. Let $q$ be the restriction of $q'$ to $L$. Then $q$ is also finitely satisfiable in $M$ (hence $M$-invariant) and extends $\tp_{L}(b/M)$, and $(b_i)_{i\in \omega}$ is a Morley sequence for $q$ over $M$. By Fact~\ref{fact:kim} and (1), $\varphi(x,b)$ Kim-divides over $M$ in $\monster$ if and only if $\{\varphi(x,b_i)\}_{i\in \omega}$ is inconsistent if and only if $\varphi(x,b)$ Kim-divides over $M$ in $\monster'$.

For (3), suppose $A\eind[K]_M B$ in $\monster'$. Then no formula in $\tp_{L'}(A/MB)$ Kim-divides over $M$ in $\monster'$, so, by (2), no formula in $\tp_{L}(A/MB)$ Kim-divides over $M$ in $\monster$. By Fact~\ref{fact:kimforking}, $A\eind[K]_M B$ in $\monster$.
\end{proof}

\noindent We usually work with a reduct $\monster$ which is stable, or at least simple. In this situation, it is useful to consider an independence relation induced on $\monster'$ by forking independence $\eind[f]$ in $\monster$ (which is defined over arbitrary sets, not just models). Define the relation $\eind[r]$, \textbf{independence in the reduct}: $$A\eind[r]_C B \quad \text{if and only if} \quad \acl'(AC)\eind[f]_{\acl'(C)} \acl'(BC)\text{ in }\monster.$$
 
\noindent Note that in the definition of $\eind[r]$, $\acl'$ is the algebraic closure in $\monster'$. If $L$ is the language of equality and $T$ is the theory of an infinite set then $\eind[r] = \eind[a]$ in $\monster'$, where $\eind[a]$ is \textbf{algebraic independence}: $$A\eind[a]_C B \quad \text{if and only if} \quad \acl(AC)\cap \acl(BC) = \acl(C).$$

\noindent Strengthened versions of extension and the independence theorem, adding additional instances of algebraic independence to the conclusion, were established for Kim-independence in $\NSOP_1$ theories in~\cite{KRExp}.
Theorem~\ref{thm:reasext} and Theorem~\ref{thm:reasind} are generalizations of these results, with $\eind[a]$ replaced by the relation $\eind[r]$ induced by the reduct $T$. The statements are non-trivial because $\eind[r]$ does not satisfy base monotonicity in general. 

\begin{thm}[Reasonable extension]\label{thm:reasext}
Suppose $T'$ is $\NSOP_1$ and $T$ is simple with stable forking and geometric elimination of imaginaries. 
For all $A\eind[K]_M B$ and for all $C$, there exists $A'$ such that $\tp_{L'}(A'/MB) = \tp_{L'}(A/MB)$, $A'\eind[K]_M BC$, and $A'\eind[r]_{MB} C$. 
\end{thm}

\begin{thm}[Reasonable independence theorem]\label{thm:reasind}
Suppose $T'$ is $\NSOP_1$ and $T$ is simple with stable forking and geometric elimination of imaginaries. If $A\eind[K]_M B$, $A'\eind[K]_M C$, $B\eind[K]_M C$, and $\tp_{L'}(A/M) = \tp_{L'}(A'/M)$, then there exists $A''$ such that $\tp_{L'}(A''/MB) = \tp_{L'}(A/MB)$, $\tp_{L'}(A''/MC) = \tp_{L'}(A'/MC)$, $A''\eind[K]_M BC$, and further $A''\eind[r]_{MC} B$, $A''\eind[r]_{MB} C$, and $B\eind[r]_{MA''} C$. 
\end{thm}

\noindent Theorem~\ref{thm:reasext} and Theorem~\ref{thm:reasind} follow from Example~\ref{ex:forking}, Theorem~\ref{thm:reasext*}, and Theorem~\ref{thm:reasind*} in  Appendix~\ref{app:reasonable}. The hypotheses on $T$ come from Lemma~\ref{lem:forkingreduct}.

\medskip \noindent When we prove the independence theorem for interpolative fusions in Section~\ref{sec:NSOP1}, we need an additional strengthening of the independence theorem for $\eind[K]$, which unfortunately we do not know how to prove in general.

\medskip \noindent Assume $T'$ is $\NSOP_1$ and $T$ is stable. We say that $\ind[K]$ in $T'$ satisfies the \textbf{$T$-generic independence theorem} if it satisfies the following strengthening of the independence theorem: Let $\sM'\preceq \monster'$ be a model with underlying set $M$, and let $A$, $A'$, $B$, and $C$ be $\acl'$-closed sets, each of which contains $M$. Suppose $\tp_{L'}(A/M) = \tp_{L'}(A'/M)$, $A\eind[K]_M B$, $A'\eind[K]_M C$, and $B\eind[K]_M C$ in $\monster'$. Then there exists $A''$ such that $\tp_{L'}(A''/B) = \tp_{L'}(A/B)$, $\tp_{L'}(A''/C) = \tp_{L'}(A'/C)$, and $A''\eind[K]_M BC$ in $\monster'$, and further, in the stable reduct $\monster$:
\begin{align*}
\acl'(A''B)\eind[f]_{A''B} \acl'(A''C)\acl'(BC)\\
\acl'(A''C)\eind[f]_{A''C} \acl'(A''B)\acl'(BC)\\
\acl'(BC)\eind[f]_{BC} \acl'(A''B)\acl'(A''C).
\end{align*}

\begin{question}\label{q:genericind}
Assume $T'$ is $\NSOP_1$ and $T$ is stable. Does $\eind[K]$ in $T'$ always satisfy the $T$-generic independence theorem? What if $T'$ is simple?
\end{question}

\noindent We do not know the answer to Question~\ref{q:genericind} in general. But we will now observe that there are many situations in which we obtain a positive answer, including when $T'$ is stable.

\begin{prop}\label{prop:stablethreebody}
If $T'$ is stable, then $\eind[K] = \eind[f]$ in $T'$ satisfies the $T$-generic independence theorem.
\end{prop}
\begin{proof} 
Let $\sM'\elesub \monster'$ be model with underlying set $M$. By Fact~\ref{fact:kimsimple}, $\eind[K]_M = \ind[f]_M$ in $\monster'$. Let $A$, $A'$, $B$, and $C$ be $\acl'$-closed sets, each of which contains $M$. Suppose $\tp_{L'}(A/M) = \tp_{L'}(A'/M)$, $A\eind[f]_M B$, $A'\eind[f]_M C$, and $B\eind[f]_M C$ in $\monster'$. Applying the independence theorem, we obtain $A''$ such that $\tp_{L'}(A''/B) = \tp_{L'}(A/B)$, $\tp_{L'}(A''/C) = \tp_{L'}(A'/C)$, and $A''\eind[f]_M BC$ in $\monster'$. 

Then $A''\eind[f]_M \acl'(BC)$ in $\monster'$. It follows that $\tp_{L'}(A''/\acl'(BC))$ is finitely satisfiable in $M$, since $\sM'$ is a model and $T'$ is stable. To prove that $$\acl'(BC)\eind[f]_{BC} \acl'(A''B)\acl'(A''C) \text{ in $\monster$},$$ it suffices by symmetry of $\ind[f]$ to show that $\tp_{L}(\acl'(A''B)\acl'(A''C)/\acl'(BC))$ is finitely satisfiable in $BC$. 

Suppose $\monster'\models \varphi(d_1,d_2,e)$, where $\varphi$ is an $L$-formula, $d_1$ is a tuple from $\acl'(A''B)$, $d_2$ is a tuple from $\acl'(A''C)$, and $e$ is a tuple from $\acl'(BC)$. Let $\psi_1(w_1,a,b)$ be an $L'$-formula isolating $\tp_{L'}(d_1/A''B)$, with $a$ from $A''$ and $b$ from $B$, and let $\psi_2(w_2,a',c)$ be an $L'$-formula isolating $\tp_{L'}(d_2/A''C)$, with $a'$ from $A''$ and $c$ from $C$. We may assume that any instances of $\psi_1$ and $\psi_2$ are algebraic. Then we have: $$\monster'\models \exists w_1\exists w_2\, (\psi_1(w_1,a,b)\land \psi_2(w_2,a',c)\land \varphi(w_1,w_2,e)).$$
Since $\tp_{L'}(A''/\acl'(BC))$ is finitely satisfiable in $M$, there exist $m$ and $m'$ in $M$ such that:
$$\monster'\models \exists w_1\exists w_2\, (\psi_1(w_1,m,b)\land \psi_2(w_2,m',c)\land \varphi(w_1,w_2,e)).$$
Let $d_1'$ and $d_2'$ be witnesses to the existential quantifiers. Then $d_1'\in \acl'(MB) = B$, since $B$ is $\acl'$-closed and contains $M$. Similarly, $d_2'\in C$. This shows that $\varphi(w_1,w_2,e)$ is satisfiable in $BC$ and establishes the claim.

We have proven one of the desired independencies. By basic properties of forking independence in stable theories (symmetry, base monotonicity, and transitivity), $A''\eind[f]_M BC$ implies  $B\eind[f]_M A''C$ and $C\eind[f]_M A''B$ in $\monster'$, and the other two independencies follow by the same argument.
\end{proof}

\noindent Note that the proof of Proposition~\ref{prop:stablethreebody} only works when $M$ is a model, since it uses the equivalence of non-forking and finite satisfiability over models in stable theories. This raises the following question in pure stability theory.

\begin{question}\label{q:genericind2}
Assume $T'$ and $T$ are both stable with elimination of imaginaries. Does $\eind[f]$ in $T'$ always satisfy the analogue of the $T$-generic independence theorem where we replace the model $M$ with an arbitrary $\acl'$-closed set?
\end{question}

\noindent We say that $\acl'$ is \textbf{disintegrated relative to $\acl$} if for any sets $A$ and $B$, we have $$\acl'(AB) = \acl(\acl'(A)\acl'(B)).$$
Equivalently, for any set $A$, $$ \acl'(A) = \acl\left( \bigcup_{a \in A} \acl'(a) \right). $$
To simplify terminology, we also say that $T'$ is \textbf{relatively disintegrated}. 

\begin{prop}\label{prop:reldisthreebody}
Assume $T'$ is $\NSOP_1$ and relatively disintegrated, and $T$ is stable. Then $\eind[K]$ in $T'$ satisfies the $T$-generic independence theorem.
\end{prop}
\begin{proof}
In this case, the additional independencies in $\monster$ are trivial.
For example, since $A''$ and $B$ are $\acl'$-closed sets, we have $\acl'(A''B) = \acl(\acl'(A'')\acl'(B)) = \acl(A''B)$. And it is always true that: $$\acl(A''B)\eind[f]_{A''B}\acl'(A''C)\acl'(BC)$$ e.g., by Fact~\ref{fact:kimacl} applied in the stable theory $T$.
\end{proof}

\noindent We say that $T$ has \textbf{disintegrated forking} if $A \enind[f]_C BB'$ implies $A \enind[f]_C B$ or $A\enind[f]_C B'$. Equivalently, if $A\enind[f]_C B$, then $A\enind[f]_C b$ for some singleton $b\in B$.  Some authors refer to this as ``trivial forking".

\medskip  \noindent The reasonable independence theorem (Theorem~\ref{thm:reasind} above) says that in any $\NSOP_1$ theory $T'$, we can witness the independence theorem in such a way that the sets $\acl'(A''B)$, $\acl'(A''C)$, and $\acl'(BC)$ are pairwise independent in the stable reduct $T$. If $T$ has disintegrated forking, then these pairwise independencies are sufficient to satisfy the $T$-generic independence theorem. 

\begin{prop}
\label{prop:disforkthreebody}
Assume $T'$ is $\NSOP_1$ and $T$ is stable with elimination of imaginaries and disintegrated forking. Then $\eind[K]$ in $T'$ satisfies the $T$-generic independence theorem.
\end{prop}

\begin{proof}
Let $\sM'\elesub \monster'$ be model with underlying set $M$. Let $A$, $A'$, $B$, and $C$ be $\acl'$-closed sets, each of which contains $M$. Suppose $\tp_{L'}(A/M) = \tp_{L'}(A'/M)$, $A\eind[f]_M B$, $A'\eind[K]_M C$, and $B\eind[K]_M C$ in $\monster'$. By Theorem~\ref{thm:reasind}, we obtain $A''$ such that $\tp_{L'}(A''/MB) = \tp_{L'}(A/MB)$, $\tp_{L'}(A''/MC) = \tp_{L'}(A'/MC)$, $A''\eind[K]_M BC$ in $\monster'$, and additionally in $\monster$:
\begin{align*}
\acl'(A''C)&\eind[f]_{C}\acl'(BC)\\
\acl'(A''B)&\eind[f]_{B}\acl'(BC)\\
\acl'(A''B)&\eind[f]_{A''}\acl'(A''C).
\end{align*}
By base monotonicity for $\eind[f]$ in $\monster$, we have:
\begin{align*}
\acl'(A''B)&\eind[f]_{A''B}\acl'(BC)\\
\acl'(A''B)&\eind[f]_{A''B}\acl'(A''C)
\end{align*}
and since $T$ has disintegrated forking,:
\begin{align*}
\acl'(A''B)&\eind[f]_{A''B}\acl'(A''C)\acl'(BC).
\end{align*}
The other two independencies follow by the same argument.
\end{proof}

\section{Logical tameness}\label{sec:analysis} 
\noindent Throughout this section, we fix the languages $L_\square$ and theories $T_\square$ for $\square\in I\cup \{\cup,\cap\}$, and we assume the interpolative fusion $T_\cup^*$ exists. 

\medskip\noindent We seek to understand when logical tameness properties (completeness, model completeness, quantifier elimination, etc.) of the theories $T_i$ are preserved in passing to the interpolative fusion $T_\cup^*$. We have already seen a close connection between interpolative fusions and model-completeness (Fact~\ref{fact: first Theorem}), which we reformulate as a first preservation result in Section~\ref{sec:presmc}. 

\medskip \noindent In order to understand definable sets and types, we often want something stronger than model-completeness. So Sections~\ref{sec:aclbcl} is devoted to preservation of $\acl$- and $\bcl$-completeness (defined in Section~\ref{sec:Kcomplete}). This requires stronger hypotheses, which hold, for example, when $T_\cap$ is stable with weak elimination of imaginaries, and with no additional assumptions on the theories $T_i$. Under the same hypotheses, we also obtain a description of algebraic closusure in $T_\cup^*$. Section~\ref{sec:qe} is about  preservation of
substructure-completeness (equivalently, quantifier elimination) under even stronger hypothesis.
Model-, $\acl$-, $\bcl$-, and substructure-completeness are all instances of the notion of $\sK$-completeness discussed in Section~\ref{sec:Kcomplete}. 

\medskip \noindent In Section~\ref{sec:morleyization}, we Morleyize the theories $T_i$ in order to translate the preservation results of Sections~\ref{sec:presmc} through~\ref{sec:qe} to results about relative syntactic complexity of definable sets for general interpolative fusions (without $\sK$-completeness assumptions). Finally, in Section~\ref{sec:preservationcompleteness}, we deduce a sufficient condition for completeness of the interpolative fusion.

\subsection{Model-completeness}\label{sec:presmc} Flat formulas (defined in Section~\ref{sec:Kcomplete}) help explain the relevance of $\sK$-completeness of $T_\cup^*$ for understanding types, and are useful in expressing syntactic consequences of $\sK$-completeness of $T_\cup^*$.

\begin{rem}\label{rem:flat} Any flat literal $L_\cup$-formula is an $L_i$-formula for some $i\in I$.
Thus:
\begin{enumerate}[leftmargin=*]
\item If $\varphi(x)$ is a flat $L_\cup$-formula, then there is some finite $J\subseteq I$ and a flat $L_i$-formula $\varphi_i(x)$ for all $i\in J$ such that $\varphi(x)$ is logically equivalent to $\bigwedge_{i\in J}\varphi_i(x)$.
\item If $A_\cup$ is an $L_\cup$-structure, then $\fdiag(A_\cup) = \bigcup_{i\in I} \fdiag(\sA_i)$, where $\fdiag(\sA_i)$ is an $L_i(A)$-theory.
    \item By Remark~\ref{rem:Kcomp}, if $T_\cup^*$ is $\sK$-complete, then for any pair $(\sA_\cup,\sM_\cup)\in \sK$, \[ T_\cup^*\cup \bigcup_{i\in I} \fdiag(\sA_i)\models \Th_{L_\cup(A)}(\sM_\cup).\]
As a consequence, when $(\sA_\cup,\sM_\cup)\in \sK$, the complete $L_\cup$-type of $A$ is determined by the quantifier-free $L_i$-types of $A$ for all $i\in I$.
\end{enumerate}
\end{rem}

\noindent We now restate Fact~\ref{fact: first Theorem} as preservation of model-completeness. If $T_\cup^*$ is model-complete, then every $L_\cup$-formula is equivalent to an existential $L_\cup$-formula.  We give a slight refinement, using the flat formulas defined above.

\begin{thm}\label{thm:mcpres}
Suppose each $T_i$ is model-complete. Then $T_\cup^*$ is model-complete, and every $L_\cup$-formula $\psi(x)$ is $T^*_\cup$-equivalent to a finite disjunction of formulas of the form \[\exists y\, \bigwedge_{i\in J} \varphi_i(x,y),\]
where $J\subseteq I$ is finite and each $\varphi_i(x,y)$ is a flat $L_i$-formula. 
\end{thm}
\begin{proof}
Model-completeness follows immediately from Fact~\ref{fact: first Theorem}. Since $T_\cup^*$ is model-complete, $\psi(x)$ is $T_\cup^*$-equivalent to an existential $L_\cup$-formula $\exists z\, \varphi(x,z)$. By Fact~\ref{fact:eflat}, $\varphi(x,z)$ is equivalent to a finite disjunction of E$\flat$-formulas. 
Now distribute $\exists z$ over the disjunction and apply Remark~\ref{rem:flat}$(1)$. 
\end{proof}

\subsection{$\acl$-completeness and $\bcl$-completeness}\label{sec:aclbcl}
Our next goal is to find sufficient conditions for preservation of $\acl$ and $\bcl$-completeness as introduced in Section~\ref{sec:Kcomplete}. It turns out that stationary and extendable independence relations, as defined in Section~\ref{sec:stationary}, play an important role. 

\medskip \noindent We say \textbf{$T_\cap$ admits a stationary and extendable independence relation} if whenever $\widehat{T}_\cap$ is a completion of $T_\cap$, and  $\widehat{T}_i$ extending $\widehat{T}_\cap$ is a completion of $T_i$ for each $i$, there is a stationary independence relation $\eind[\cap]$ in $\widehat{T_\cap}$ which is extendable to $\widehat{T}_i$ for each $i$.

\medskip \noindent In general, this is a hypothesis on the relationship between $T_\cap$ and each $T_i$, not on $T_\cap$ alone. However, by Proposition~\ref{prop:stationarity}, it is always satisfied by forking independence in $T_\cap$ when $T_\cap$ is stable with weak elimination of imaginaries, with no additional assumptions on the theories $T_i$. This holds, for instance, when $T_\cap$ is the theory of an infinite set or the theory of algebraically closed fields.

\medskip \noindent From the proof, we also obtain a characterization of $\acl$ in $T_\cup^*$.  Let $A$ be a subset of a $T^{*}_\cup$-model $\sM_\cup$. Given $\square\in I\cup \{\cup,\cap\}$, let $\acl_\square(A)$ be the algebraic closure of $A$ in the reduct $\sM_\square$.
The \textbf{combined closure}, $\ccl(A)$, is the smallest set containing $A$ which is $\acl_i$-closed for each $i\in I$. 
More concretely, $b\in \ccl(A)$ if and only if
\[b\in \acl_{i_n}(\dots (\acl_{i_1}(A))\dots) \text{  for some  } i_1,\ldots,i_n \in I.\]

%In the general case, elimination of imaginaries for $T_\cap$ is easily arranged (see Remark~\ref{rem: robust}(4)).

\medskip\noindent We can now state and prove the main result of this section.

\begin{thm}\label{thm:dap}
Assume $T_\cap$ admits a stationary and extendable independence relation $\eind[\cap]$. If each $T_i$ is $\acl$-complete, then $T_\cup^*$ is $\acl$-complete and $\acl_\cup = \ccl$. 
\end{thm}

\begin{proof}
Each $T_i$ is model-complete, so $T_\cup^*$ is model-complete by Theorem~\ref{thm:mcpres}. 
In order to apply Proposition~\ref{prop:kcomp}, we will show that the class of $T_\cup^*$-models has the disjoint $\ccl$-amalgamation property.

So suppose $\sA_\cup$ is a $\ccl$-closed substructure of a $T^{*}_\cup$-model $\sM_\cup$, $\sN_\cup\models T_\cup^*$,  and $f \colon \sA_\cup\to \sN_\cup$ is an embedding. Let $\monster_\cup$ be a monster model of $\Th_{L_\cup}(\sN_\cup)$ (this is a completion of $T_\cup^*$), so $\sN_\cup\elesub \monster_\cup$. Let $p_\square(x) = \tp_{L_\square}(M/A)$ for each $\square \in I\cup \{\cap\}$, where $x$ is an infinite tuple of variables enumerating $M$. By $\acl$-completeness of $T_i$, $f\colon \sA_i\to \sN_i$ is partial elementary for all $i\in I$, so $f\colon \sA_\cap\to \sN_\cap$ is also partial elementary, and we can replace the parameters from $A$ in $p_\square(x)$ by their images under $f$, obtaining a consistent type $p_\square'(x)$ over $A' = f(A)\subseteq N$ for all $\square\in I\cup \{\cap\}$. 

Fix $i\in I$. Since $A$ is algebraically closed in $\sM_i$ and $f\colon \sA_i\to \monster_i$ is partial elementary, $A'$ is algebraically closed in $\monster_i$. Since $\eind[\cap]$ satisfies full existence over algebraically closed sets in $T_i$, there is a realization $M'_i$ of $p_i'(x)$ in $\monster_i$ such that $M_i'\eind[\cap]_{A'} N$ in $\monster_\cap$. Let $q_i(x) = \tp_{L_i}(M'_i/N)$. 

For all $i,j\in I$, $\tp_{L_\cap}(M_i'/A') = \tp_{L_\cap}(M_j'/A') = p'_\cap(x)$, so since $\eind[\cap]$ satisfies  stationarity over algebraically closed sets, $\tp_{L_\cap}(M_i'/N) = \tp_{L_\cap}(M_j'/N)$. Let $q_\cap(x)$ be this common type, so $q_\cap(x)\subseteq q_i(x)$ for all $i$. By Proposition~\ref{prop:realizingtypes}, $\bigcup_{i\in I} q_i(x)$ is realized by a set $M'$ in a model $\sN_\cup'$ such that $\sN_\cup\elesub \sN_\cup'\prec \monster_\cup$.

Let $f'\colon \sM_\cup\to \sN'_\cup$ be the map induced by the common enumeration of $M$ and $M'$ by the variables $x$. Then $f'$ is an $L_i$-embedding for all $i\in I$, so it is an $L_\cup$-embedding. Since $T_\cup^*$ is model-complete, $f'$ is an elementary embedding. If $a\in A$, then $a$ is enumerated by a variable $x_a$ from $x$, and the formula $x_a = a$ is in $p_\cap(x)$. Then the formula $x_a = f(a)$ is in $p'_\cap(x)$, so $f'(a) = f(a)$. This establishes the ccl-amalgamation property.
\[
\xymatrix{
\sM_\cup\ar[r]_{f'} & \sN_\cup'\\
\sA_\cup\ar[r]^f\ar[u]^{\subseteq} & \sN_\cup\ar[u]_{\elesub}
}
\]

Additionally, we have $M'\eind[\cap]_{A'} N$ in $\monster_\cap$, so since $\eind[\cap]$ satisfies algebraic independence, $M'\cap N = A'$, and hence $f'(M)\cap N = f(A)$. This establishes the disjoint ccl-amalgamation property.

By Proposition~\ref{prop:kcomp}, $T_\cup^*$ is $\ccl$-complete and every $\ccl$-closed substructure is $\acl_\cup$-closed. It follows that for any set $B\subseteq \sM\models T$, $\acl_\cup(B)\subseteq \ccl(B)$. 

For the converse, it suffices to show $\acl_\cup(B)$ is $\acl_i$-closed for all $i\in I$. 
Indeed, 
\[
\acl_i(\acl_\cup(B))\subseteq \acl_\cup(\acl_\cup(B)) = \acl_\cup(B).
\]

So $\acl_\cup = \ccl$, and hence $T_\cup^*$ is $\acl$-complete. 
\end{proof}

\noindent Under the same hypotheses, we also obtain preservation of $\bcl$-completeness and a syntactic consequence for $L_\cup$-formulas.

\begin{cor}
\label{cor:bclpres}
Assume $T_\cap$ admits a stationary and extendable independence relation.
Suppose each $T_i$ is $\bcl$-complete.
Then $T_\cup^*$ is $\bcl$-complete and every $L_\cup$-formula is $T^{*}_\cup$-equivalent to a finite disjunction of b.e.\ formulas of the form \[\exists y\, \bigwedge_{i\in J} \varphi_i(x,y),\] where $J\subseteq I$ is finite and $\varphi_i(x,y)$ is a flat $L_i$-formula for all $i\in J$. 
\end{cor}

\begin{proof}
By Theorem~\ref{thm:bclacl}, $T_i$ is  $\acl$-complete and $\bcl_i = \acl_i$ for all $i  \in I$.
We have $\bcl_\cup(A) \subseteq \acl_\cup(A)$ for any subset $A$ of a $T_\cup$-model.
But also, for all $i\in I$, 
\begin{align*}
\acl_i(\bcl_\cup(A)) &= \bcl_i(\bcl_\cup(A))\\
&\subseteq \bcl_\cup(\bcl_\cup(A))\\
&= \bcl_\cup(A).
\end{align*}
So $\bcl_\cup(A)$ is $\acl_i$-closed for all $i\in I$, hence
\[\ccl(A) \subseteq \bcl_\cup(A)\subseteq \acl_\cup(A).\]

Theorem~\ref{thm:dap} implies $T_\cup^*$ is $\acl$-complete and $\ccl(A) = \acl_\cup(A)$, so the containments above are equalities: \[\ccl(A) = \bcl_\cup(A)= \acl_\cup(A).\] Since $T_\cup^*$ is $\acl$-complete and $\acl_\cup = \bcl_\cup$, $T_\cup^*$ is $\bcl$-complete.

It remains to prove the characterization of $L_\cup$-formulas.
By Theorem~\ref{thm:bclacl}, $\bcl$-completeness implies that every $L_\cup$-formula is $T_\cup^*$-equivalent to a finite disjunction of b.e.\ formulas. Let $\exists y\, \psi(x,y)$ be a b.e.\ formula appearing in the disjunction. By Fact~\ref{fact:eflat}, the quantifier-free formula $\psi(x,y)$ is equivalent to a finite disjunction of E$\flat$-formulas $\bigvee_{k=1}^m \exists z_k\, \theta_k(x,y,z_k)$. Distributing the quantifier $\exists y$ over the disjunction, we find that $\exists y\exists z_k\, \theta_k(x,y,z_k)$ is a b.e. formula. Applying Remark~\ref{rem:flat}(1) to the flat formula $\theta_k(x,y,z_k)$ yields the result.
\end{proof}

\subsection{Quantifier elimination}\label{sec:qe} 
Recall that quantifier elimination is equivalent to substructure-completeness. This follows from~\cite[Theorem 8.4.1]{Hodges} and Proposition~\ref{prop:kcomp}. In contrast to model-completeness, $\acl$-completeness, and $\bcl$-completeness, we cannot obtain preservation of substructure-completeness without tight control on algebraic closure in the theories $T_i$.

\medskip\noindent Theorem~\ref{thm:qe}  below is motivated by some comments in the introduction of \cite{Moosa-Scanlon} on the failure of quantifier elimination in $\ACFA$. For subsets $A\subseteq B \subseteq \sM$, we denote by $\Aut(B/A)$ the group of all partial elementary bijections $f\colon B\to B$ which fix $A$ pointwise.

\begin{thm}\label{thm:qe}
Assume $T_\cap$ admits a stationary and extendable independence relation. Suppose each $T_i$ has quantifier elimination, and
$$ \acl_i(A) = \acl_\cap(A) \quad \text{and} \quad \Aut_{L_\cap}(\acl_\cap(A)/A) = \Aut_{L_i}(\acl_\cap(A)/A) $$
for all $L_\cup$-substructures $A$ of $T_\cup^*$-models and all $i \in I$. 
Then $T_\cup^*$ has quantifier elimination.
\end{thm}

\begin{proof}
For all $i\in I$, $T_i$ is substructure-complete, and hence $\acl$-complete. So by Theorem~\ref{thm:dap}, $T_\cup^*$ is $\acl$-complete and $\acl_\cup = \ccl$. We will show $T_\cup^*$ is substructure complete. 

Suppose $\sA_\cup$ is an $L_\cup$-substructure of a $T^{*}_\cup$-model $\sM_\cup$, $\sN_\cup\models T^{*}_\cup$, and $f\colon \sA_\cup\to \sN_\cup$ is an embedding. 
As each $T_i$ is substructure-complete, $f$ is partial elementary $\sA_i\to \sN_i$, so $f$ extends to a partial elementary map $g_i\colon \acl_i(A) \to \sN_i$. By our hypothesis, $\acl_i(A) = \acl_\cap(A)$, so we have a family of maps $(g_i)_{i\in I}$ defined on $\acl_\cap(A)$ and extending $f$.

Fix $j\in I$. For all $i\in I$, $(g_i^{-1}\circ g_j)\in \Aut_{L_\cap}(\acl_\cap(A)/A)$, so in fact it is an $L_i$-automorphism of $\acl_\cap(A)$ by our assumption on the automorphism groups. It follows that $g_j = g_i\circ (g_i^{-1}\circ g_j)$ is an $L_i$-embedding $\acl_\cap(A)\to \sN_i$. Since $i$ was arbitrary, $g_j$ is an $L_\cup$-embedding. Now since $\acl_\cap(A)$ is $\acl_i$-closed for all $i\in I$, it is $\ccl$-closed, and hence $\acl_\cup$-closed. So by $\acl$-completeness of $T_\cup^*$, $g_j\colon \acl_\cap(A) \to \sN_\cup$ is partial elementary, and hence so is $g_j|_A = f$.
\end{proof}

\noindent We prefer hypothesis which can be checked language-by-language, i.e., which refer only to properties of $T_i$, $T_\cap$, and the relationship between $T_i$ and $T_\cap$, rather than how $T_i$ and $T_j$ relate when $i\neq j$, or how $T_i$ relates to $T_\cup$. The hypotheses of Theorem~\ref{thm:qe} are not strictly language-by-language, because they refer to an arbitrary $L_\cup$-substructure $A$. However, there are several natural strengthenings of these hypotheses which are language-by-language. One is to simply assume the hypotheses of Theorem~\ref{thm:qe} for all sets $A$. Simpler language-by-language criteria are given in the following corollaries.  

\begin{cor}\label{cor:qe}
Assume $T_\cap$ admits a stationary and extendable independence relation.
Suppose each $T_i$ admits quantifier elimination. If either of the following conditions hold for all sets $A$, then $T_\cup^*$ has quantifier elimination:
\begin{enumerate}
\item $\acl_i(A) = \langle A\rangle_{L_i}$ for all $i\in I$.
\item $\acl_i(A) = \dcl_\cap(A)$ for all $i\in I$.
\end{enumerate}
\end{cor}

\begin{proof}
We apply Theorem~\ref{thm:qe}, so assume $A = \langle A \rangle_{L_\cup}$. 
\begin{enumerate}
\item We have $A\subseteq \dcl_\cap(A)\subseteq \acl_\cap(A) \subseteq \acl_i(A) = \langle A \rangle_{L_i} = A$.
\item We have $\dcl_\cap(A) \subseteq \acl_\cap(A)\subseteq \acl_i(A) = \dcl_\cap(A)$. 
\end{enumerate}
In either case, $\acl_i(A) = \acl_\cap(A) = \dcl_\cap(A)$. It follows that the group $\Aut_{L_\cap}(\acl_\cap(A)/A) = \Aut_{L_\cap}(\dcl_\cap(A)/A)$ is already trivial, since every partial elementary map which fixes $A$ pointwise also fixes $\dcl_\cap(A)$ pointwise, so the subgroup $\Aut_{L_i}(\acl_\cap(A)/A)$ is also trivial.
\end{proof}

\begin{cor}\label{cor:qe2}
Assume $T_\cap$ admits a stationary and extendable independence relation. Suppose each $T_i$ admits quantifier elimination and a universal axiomatization.
Then $T_\cup^*$ has quantifier elimination.
\end{cor} 
\begin{proof}
Every $L_i$-substructure of a model of $T_i$ is an elementary substructure, and hence $\acl_i$-closed, so we can apply Corollary~\ref{cor:qe}(1).
\end{proof}

\subsection{Language-independent consequences}\label{sec:morleyization}

The results of Sections~\ref{sec:presmc} through \ref{sec:qe} can be lifted to the general case (when we have no $\sK$-completeness hypotheses on the $T_i$) via Morleyization.
This allows us to understand  $L_\cup$-definable sets and certain complete $L_\cup$-types relative to $L_i$-definable sets and complete $L_i$-types.

\medskip\noindent To set notation: For each $i$, Morleyization gives a definitional expansion $L_i^\diamondsuit$ of $L_i$ and an extension $T_i^\diamondsuit$ of $T_i$ by axioms defining the new symbols in $L_i^\diamondsuit$. We assume that the new symbols in $L_i^\diamondsuit$ and $L_j^\diamondsuit$ are distinct for $i\neq j$, so that $L_i^\diamondsuit \cap L_j^\diamondsuit = L_\cap$. It follows that each $T_i^\diamondsuit$ has the same set of $L_\cap$ consequences, namely $T_\cap$. We let $L^\diamondsuit_\cup = \bigcup_{i\in I} L_i^\diamondsuit$ and $T_\cup^{\diamondsuit} = \bigcup_{i\in I} T_i^\diamondsuit$. Then every $T_\cup$-model $\sM_\cup$ has a canonical expansion to a $T_\cup^{\diamondsuit}$-model $\sM_\cup^\diamondsuit$, and by Remark~\ref{rem: robust}, $\sM_\cup$ is interpolative if and only if $\sM_\cup^\diamondsuit$ is interpolative. 

\medskip\noindent The first result applies to any interpolative fusion $T_\cup^*$.

\begin{prop}\label{prop:relmc}\quad
\begin{enumerate}[leftmargin=*]
    \item Every $L_\cup$-formula $\psi(x)$ is $T_\cup^*$-equivalent to a finite disjunction of formulas of the form
$$ \exists y \bigwedge_{i \in J} \varphi_i(x,y) $$
where $J \subseteq I$ is finite and $\varphi_i(x,y)$ is an $L_i$-formula for all $i \in J$.
    \item If $\sM_\cup$ is a $T^*_\cup$-model, then \[T_\cup^*\cup \bigcup_{i\in I} \ediag(\sM_i)\models \ediag(\sM_\cup).\]
\end{enumerate}
\end{prop}
\begin{proof}
For (1), each Morleyized theory $T_i^\diamondsuit$ has quantifier elimination, hence is model-complete, so we can apply Theorem~\ref{thm:mcpres} to the interpolative fusion  $(T_\cup^\diamondsuit)^*$. This says that $\psi(x)$ is $(T_\cup^\diamondsuit)^*$-equivalent to a finite disjunction of formulas of the form $\exists y\, \bigwedge_{i\in J} \varphi_i(x,y)$, where each $\varphi_i(x,y)$ is a flat $L_i^\diamondsuit$-formula. But since $L_i^\diamondsuit$ is a definitional expansion of $L_i$, each formula $\varphi_i(x,y)$ can be translated back to an $L_i$-formula. 

Now (2) is just a restatement of Fact~\ref{fact:extension}(2). But we will give another proof, to illustrate the Morleyization method, which will be used repeatedly below. By Theorem~\ref{thm:mcpres},  $(T_\cup^\diamondsuit)^*$ is model-complete, so by Remark~\ref{rem:flat}(3) we have \[(T_\cup^*)^{\diamondsuit}\cup \bigcup_{i\in I} \fdiag(\sM_i^{\diamondsuit}) \models \ediag(\sM_\cup^{\diamondsuit}).\] But since $L_i^{\diamondsuit}$ is a definitional expansion of $L_i$,  $\fdiag(\sM_i^{\diamondsuit})$ is completely determined by $\ediag(\sM_i)$, and the result follows.
\end{proof}

%\noindent We note that Proposition~\ref{prop:relmc}(2) is simply a restatement of Theorem~\ref{thm:relativemc}(3), which we think of as ``relative model-completeness''. 

\noindent We will now establish a sequence of variants on Proposition~\ref{prop:relmc}, with stronger hypotheses and stronger conclusions, but with essentially the same proof.

\begin{prop}
\label{prop:relaclcomp}
Assume $T_\cap$ admits a stationary and extendable independence relation. Then:
\begin{enumerate}
\item In models of $T_\cup^*$, $\acl_\cup = \ccl$. 
\item Every $L_\cup$-formula $\psi(x)$ is $T_\cup^*$-equivalent to a finite disjunction of formulas of the form $$\exists y\, \bigwedge_{i\in J} \varphi_i(x,y)$$ where $J\subseteq I$ is finite, $\varphi_i(x,y)$ is an $L_i$-formula for all $i\in J$, and $\bigwedge_{i\in J} \varphi_i(x,y)$ is bounded in $y$ (see Appendix~\ref{app:Kcomp} for the definition).
\item If $A$ is an $\acl_\cup$-closed subset of a $T_\cup^*$-model $M_\cup$, then $$T_\cup^*\cup \bigcup_{i\in I} \Th_{L_i(A)}(\sM_i)\models \Th_{L_\cup(A)}(\sM_\cup).$$
\end{enumerate}
\end{prop}

\begin{proof}
Each Morleyized theory $T_i^{\diamondsuit}$ has quantifier elimination, hence is $\acl$-complete and $\bcl$-complete. Further, Morleyization does not affect the stationary and extendable independence relation on $T_\cap$. Then (1) follows from Theorem~\ref{thm:dap}, observing that $\acl_\cup$ and $\ccl$ are not altered by definitional expansions. (2) and (3) follow just as in Proposition~\ref{prop:relmc}, using the syntactic result of  Corollary~\ref{cor:bclpres} for (2).
\end{proof}

\begin{prop}\label{prop:relqe1}
Assume $T_\cap$ admits a stationary and extendable independence relation. Suppose further that  
$$ \acl_i(A) = \acl_\cap(A) \quad\text{and}\quad \Aut_{L_\cap}(\acl_\cap(A)/A) = \Aut_{L_i}(\acl_\cap(A)/A)$$ for all $L_\cup$-substructures $\sA_\cup$ of $T_\cup^*$-models and all $i \in I$. Then:
\begin{enumerate}
\item Every formula $\psi(x)$ is $T_\cup^*$-equivalent to a finite disjunction of formulas
$$\exists y\, \bigwedge_{i\in J} \varphi_i(x,y)$$ where $J\subseteq I$ is finite, $\varphi_i(x,y)$ is an $L_i$-formula for all $i\in J$, and $\bigwedge_{i\in J} \varphi_i(x,y)$ is bounded in $y$ with bound $1$.
\item If $\sA_\cup$ is an $L_\cup$-substructure of a $T_\cup^*$-model $\sM_\cup$, then 
$$T_\cup^* \cup \bigcup_{i\in I} \Th_{L_i(A)}(\sM_i) \models \Th_{L_\cup(A)}(\sM_\cup).$$
\end{enumerate}
\end{prop}
\begin{proof}
Morleyization does not affect our hypotheses on $\acl_i$, $\acl_\cap$, and the stationary and extendable independence relation on $T_\cap$. So by Theorem~\ref{thm:qe},  $(T_\cup^\diamondsuit)^*$ has quantifier elimination. This gives us (2) as in the proof of Proposition~\ref{prop:relmc}. 

For (1), $\psi(x)$ is $(T_\cup^\diamondsuit)^*$-equivalent to a quantifier-free $L_\cup^{\diamondsuit}$-formula $\psi^{\diamondsuit}(x)$. We cannot translate $\psi^{\diamondsuit}(x)$ back to a Boolean combination of $L_i$-formulas, since a single atomic subformula of $\psi^{\diamondsuit}(x)$ may involve function and relation symbols from distinct languages. 

However, by Fact~\ref{fact:eflat}, $\psi^{\diamondsuit}(x)$ is equivalent to a finite disjunction of E$\flat$-formulas $\exists y\, \theta^{\diamondsuit}(x,y)$. Applying Remark~\ref{rem:flat}(1) to each flat formula $\theta^{\diamondsuit}(x,y)$, we obtain a finite disjunction of formulas of the form $$\exists y\, \bigwedge_{i\in J} \varphi^{\diamondsuit}_i(x,y)$$ where $J\subseteq I$ is finite, $\varphi_i^{\diamondsuit}(x,y)$ is a flat $L_i^{\diamondsuit}$-formula for all $i\in J$, and $\bigwedge_{i\in J} \varphi_i(x,y)$ is bounded in $y$ with bound $1$. Replacing each $L_i^{\diamondsuit}$-formula  $\varphi_i^{\diamondsuit}$ with an equivalent $L_i$-formula, we obtain the desired result.
\end{proof}

\begin{rem}\label{rem:qe1}
As in Corollary~\ref{cor:qe}(1), we can replace the hypotheses of Proposition~\ref{prop:relqe1} with:  $T_\cap$ admits a stationary and extendable independence relation, and for all sets $A$ and all $i\in I$, $\acl_i(A) = \langle A\rangle_{L_i}$. The assumption $\acl_i(A) = \dcl_\cap(A)$ gives us something stronger, see Remark~\ref{rem:qe2} below. 
\end{rem} 

\noindent With a slightly stronger hypothesis, we can get true relative quantifier elimination down to $L_i$-formulas in $T_\cup^*$.

\begin{prop}
\label{prop:relqe2}
Assume $T_\cap$ admits a stationary and extendable independence relation. Suppose further that 
$$\acl_i(A) = \acl_\cap(A) \quad\text{and}\quad \Aut_{L_\cap}(\acl_\cap(A)/A) = \Aut_{L_i}(\acl_\cap(A)/A)$$ for all sets $A$ and all $i \in I$. Then:
\begin{enumerate}
    \item Every formula is $T_\cup^*$-equivalent to a Boolean combination of $L_i$-formulas.
    \item For any subset $A$ of a $T_\cup^*$-model $\sM_\cup$,  
$$T_\cup^* \cup \bigcup_{i\in I} \Th_{L_i(A)}(\sM_i) \models \Th_{L_\cup(A)}(\sM_\cup).$$
\end{enumerate}
\end{prop}

\begin{proof}
We first move to a relational language by replacing all function symbols by their graphs. Then we proceed just as in the proof of Proposition~\ref{prop:relqe1}, noting that when $L_\cup^\diamondsuit$ is relational, every subset of a $(T_\cup^{\diamondsuit})^*$-model is a substructure, and each quantifier-free $L_\cup^\diamondsuit$-formula is already a Boolean combination of $L_i^\diamondsuit$-formulas.
\end{proof}

\begin{rem}\label{rem:qe2}
Once again, as in Corollary~\ref{cor:qe}(2), we can replace the hypotheses of Proposition~\ref{prop:relqe2} with: $T_\cap$ admits a stationary and extendable independence relation, and $\acl_i(A) = \dcl_\cap(A)$ for all sets $A$ and all $i\in I$. The assumption $\acl_i(A) = \langle A\rangle_{L_i}$ does not suffice for this, because this condition is lost when passing to a relational language. 
\end{rem}

%\noindent It follows from Proposition~\ref{prop:relaclcomp} that if $T_\cap$ admits a stationary and extendable independence relation, then the completions of $T_\cup^*$ are determined by the $L_i$-types of $\acl_\cup(\emptyset)$ for all $i$. 

\subsection{Completeness} \label{sec:preservationcompleteness}
We view Theorem~\ref{thm:dap} and its avatar Proposition~\ref{prop:relaclcomp} as the main results of this section, since they give us useful tools for understanding $L_\cup$-definable sets and complete $L_\cup$-types, while only requiring mild hypotheses (in particular, they apply whenever $T_\cap$ is stable with weak elimination of imaginaries). For example, we obtain from Proposition~\ref{prop:relaclcomp} the following criterion for completeness of $T_\cup^*$.

\begin{cor}
\label{cor:completeness}
Assume $T_\cap$ admits a stationary and extendable independence relation.
Suppose each $T_i$ is complete and $\emptyset$ is $\acl_i$-closed for all $i \in I$.
Then $T^*_\cup$ is complete.
\end{cor}
\begin{proof}
Since $\emptyset$ is $\acl_i$-closed for all $i\in I$, it is $\ccl$-closed, and hence $\acl_\cup$-closed by Proposition~\ref{prop:relaclcomp}(1). So for any model $\sM_\cup\models T_\cup^*$, by Proposition~\ref{prop:relaclcomp}(3), \[T_\cup^*\cup \bigcup_{i\in I} \Th(\sM_i)\models \Th(\sM_\cup).\]
Since each $T_i$ is complete, $\bigcup_{i\in I} \Th(\sM_i)= T_\cup\subseteq T_\cup^*$, so $T_\cup^*\models \Th(\sM_\cup)$. 
\end{proof}

\noindent In general, when $T_\cap$ admits a stationary and extendable independence relation and each $T_i$ is complete, a completion of $T_\cup^*$ is determined by the $L_\cup$-isomorphism type of $\acl_\cup(\emptyset) = \ccl(\emptyset)$ in any model. For example, this is what happens in $\ACFA$.

%Corollary~\ref{cor:completeness} follows from [ccl] and Proposition~\ref{prop:relqe1}. The assumption that $\emptyset$ is $\acl_i$-closed for all $i$ is necessary. Suppose $I = \{1,2\}$, $T_1$ is the theory of an infinite set equipped with a distinguished element $c$ and $T_2$ is the theory of a set equipped with a unary relation $P$ defining an infinite and co-infinite set. It is easy to see that $T^*_\cup$ has exactly two completions, depending on whether $P(c)$ or $\neg P(c)$ holds.

\medskip \noindent We conclude with two counterexamples indicating the sharpness of Theorem~\ref{thm:dap} and Proposition~\ref{prop:relaclcomp}. 
In the first example, $T_\cap$ is unstable with elimination of imaginaries, and $\eind[f]$ fails to be stationary. In the second example, $T_\cap$ is stable but fails weak elimination of imaginaries, and $\eind[f]$ is stationary but fails to be extendable. In both examples, $T_\cap$ does not admit any stationary and extendable independence relation, and we do not even get the result of Corollary~\ref{cor:completeness}.

\begin{example}
Let $L_\cap = \{\leq\}$ where $\leq$ is a binary relation symbol, and let $L_i$ be the expansion of $L_\cap$ by a unary predicate $P_i$ for $i \in \{1,2\}$. 
Let $T_\cap = \text{DLO}$, and for $i \in \{1,2\}$, let $T_i$ be the theory of a dense linear order equipped with a downwards closed set with an upper bound, but no least upper bound, defined by $P_i$. Note that $T_i$ is complete and $\emptyset$ is $\acl_i$-closed for $i\in \{1,2\}$. A model $\sM_\cup\models T_\cup$ is interpolative if and only if $P_1(\sM_\cup)\neq P_2(\sM_\cup)$, so $T_\cup^*$ exists. But a $T_\cup^*$-model either has $P_1(\sM_\cup) \subsetneq P_2(\sM_\cup)$ or $P_2(\sM_\cup) \subsetneq P_1(\sM_\cup)$, so $T_\cup^*$ is not complete.
\end{example}

\begin{example}\label{ex:eqrel}
Let $L_\cap = \{E\}$ where $E$ is a binary relation symbol, and let $L_i$ be the expansion of $L_\cap$ by a unary predicate $P_i$ for $i \in \{1,2\}$.
Let $T_\cap$ be the theory of an equivalence relation with infinitely many infinite classes.
For $i\in \{1,2\}$, let $T_i$ be the theory of a $T_\cap$-model with a distinguished equivalence class, defined by $P_i$.
Again, $T_i$ is complete and $\emptyset$ is $\acl_i$-closed for $i\in \{1,2\}$.
Every model of $T_\cup$ is interpolative, so $T_\cup^* = T_\cup$. But a $T_\cup^*$-model $\sM_\cup$ either has $P_1(\sM_\cup) = P_2(\sM_\cup)$ or $P_1(\sM_\cup)\neq P_2(\sM_\cup)$, so $T^{*}_\cup$ is not complete.
\end{example}

\section{Combinatorial tameness} \label{sec: combinatorialtameness}

\noindent Throughout this section, we fix the languages $L_\square$ and theories $T_\square$ for $\square\in I\cup \{\cup,\cap\}$, and we assume the interpolative fusion $T_\cup^*$ exists. 

\subsection{Stability and $\NIP$}\label{sec:stablenip}

Under very strong hypotheses, we can prove that stability and NIP are preserved by interpolative fusions.

\begin{prop}
\label{prop:stabnip}
Assume the hypotheses of Proposition~\ref{prop:relqe2}. If each $T_i$ is stable, then $T_\cup^*$ is stable. If each $T_i$ is $\NIP$, then $T_\cup^*$ is $\NIP$. 
\end{prop}

\begin{proof}
This follows from Proposition~\ref{prop:relqe2}(1), since any Boolean combination of stable formulas is stable, and any Boolean combination of $\NIP$ formulas is $\NIP$. 
\end{proof}

\noindent We can also use Proposition~\ref{prop:relqe2}(2) to count types.

\begin{prop}\label{prop:stableinkappa}
Assume the hypotheses of Proposition~\ref{prop:relqe2}, and let $\kappa$ be an infinite cardinal such that $\kappa^{|I|} = \kappa$. If each $T_i$ is stable in $\kappa$, then $T_\cup^*$ is stable in $\kappa$. 
\end{prop}
\begin{proof}
Let $A$ be a subset of a model $\sM_\cup\models T_\cup^*$ such that $|A|\leq \kappa$. We would like to understand the size of $S^x_{L_\cup}(A)$, the space of $L_\cup$-types in the finite variable context $x$ over $A$. By Proposition~\ref{prop:relqe2}(2), a type in $S^x_{L_\cup}(A)$ is completely determined by its restrictions to $L_i$-types in $S^x_{L_i}(A)$ for all $i\in I$. Since $T_i$ is stable in $\kappa$, we have $|S^x_{L_i}(A)|\leq \kappa$, and $|S^x_{L_\cup}(A)| \leq \prod_{i\in I} |S^x_{L_i}(A)| \leq \kappa^{|I|} = \kappa$. So $T_\cup^*$ is stable in $\kappa$.
\end{proof}

\noindent Since $\kappa^{|I|} = \kappa$ for all infinite $\kappa$ when $I$ is finite, we obtain the following corollary. 

\begin{cor}
Assume the hypotheses of Proposition~\ref{prop:relqe2}, and further assume $|I|$ is finite. If each $T_i$ is $\aleph_0$-stable, then $T_\cup^*$ is $\aleph_0$-stable. If each $T_i$ is superstable, then $T_\cup^*$ is superstable.
\end{cor}

\medskip\noindent We do not expect to obtain preservation of stability or NIP without strong restrictions on $\acl$, as in the hypotheses of Proposition~\ref{prop:relqe2}. The proofs of Propositions~\ref{prop:stabnip} and ~\ref{prop:stableinkappa} do not apply to other classification-theoretic properties such as simplicity, $\NSOP_1$, and NTP$_2$, as these properties are not characterized by counting types, and formulas with these properties are not closed under Boolean combinations in general. Nevertheless, in the subsequent sections we obtain preservation results for simplicity and $\NSOP_1$ under more general hypotheses.

\subsection{$\NSOP_1$}\label{sec:NSOP1}
%We now turn to preservation of simplicity and $\NSOP_1$. 
Toward proving that $T_\cup^*$ is $\NSOP_1$, we define an independence relation $\ind$ on subsets of any monster model $\monster_\cup\models T_\cup^*$. Then we seek to apply Fact~\ref{fact:char} to show that $T_\cup^*$ is $\NSOP_1$ and $\ind$ is Kim-independence. 

%\subsection{Simplicity and $\NSOP_1$: the relatively disintegrated case}

%In this section, we isolate a condition, relatively disintegrated $\acl$, which makes preservation simplicity and $\NSOP_1$ easier to establish. We handle the general case in the next section.

\medskip \noindent Assume each $T_i$ is $\NSOP_1$, and let $\monster_\cup\models T_\cup^*$ be a monster model (the choice of $\monster_\cup$ amounts to the choice of a completion of $T_\cup^*$). For all $A,B\subseteq \monsterset$ and $\sM_\cup\preceq \monster_\cup$, we define:
$$A\ind_M B \quad \text{if and only if} \quad \acl_\cup(MA) \eind[K]_M \acl_\cup(MB) \text{ in } \monster_i  \text{ for all }  i \in I.$$
This definition is motivated by the following considerations. If $\Th(\monster_\cup)$ is $\NSOP_1$, then $A\eind[K]_M B$ implies $\acl_\cup(MA)\eind[K]_M \acl_\cup(MB)$ in $\monster_\cup$ by Fact~\ref{fact:kimacl}. Then by Lemma~\ref{lem:kimreduct}, we have $\acl_\cup(MA)\eind[K]_M \acl_\cup(MB)$ in $\monster_i$ for all $i$. Conversely, it is reasonable to hope that Kim-forking between $\acl_\cup(MA)$ and $\acl_\cup(MB)$ in some $\monster_i$ is the only source of Kim-forking between $A$ and $B$ in $\monster_\cup$.

\medskip \noindent In fact, with the exception of the independence theorem over models, all of the properties required by Fact~\ref{fact:char} follow easily for $\ind$.

\begin{prop}\label{prop:properties}
Assume each $T_i$ is $\NSOP_1$. Then $\ind$ satisfies invariance, existence, monotonicity, symmetry, and strong finite character. If $\Th(\monster_\cup)$ is $\NSOP_1$, then $\ind$ also satisfies witnessing.
\end{prop}
\begin{proof}
\emph{Invariance, existence, monotonicity, symmetry:} Clear from the definition, using the corresponding properties of Kim-independence in each $\monster_i$. 

\emph{Strong finite character:} Suppose $A\nind_M B$. Then for some $i\in I$, we have $\acl_\cup(MA)\enind[K]_M \acl_\cup(MB)$ in $\monster_i$. So there is some $a'\in \acl_\cup(MA)$ and $b'\in \acl_\cup(MB)$ such that $a'\enind[K]_M b'$ in $\monster_i$. Let $\varphi(x',b',m)$ be an $L_i$-formula in $\tp_{L_i}(a'/Mb')$ which Kim-forks over $M$ in $\monster_i$, let $\psi(x',a,m)$ be an $L_\cup$-formula isolating the algebraic type $\tp_{L_\cup}(a'/MA)$, and let $\theta(y',b,m)$ be an $L_\cup$-formula isolating the algebraic type  $\tp_{L_\cup}(b'/MB)$. Note that by replacing $\psi$ with $\psi(x',a,m)\land (\exists^{\leq k} x'\, \psi(x',a,m))$ for some $k$, we may assume $\psi(x',c,m)$ has only finitely many realizations for any $c$.

We claim that the following formula $\chi(x,b,m)$ witnesses strong finite character: $$\exists x'\,\exists y'\,\left[\varphi(x',y',m)\land \psi(x',x,m)\land \theta(y',b,m)\right].$$

Certainly we have $\chi(x,b,m)\in \tp_{L_\cup}(A/MB)$. Suppose we are given $c$ such that $\monster_\cup\models \chi(c,b,m)$. Then picking witnesses $c'$ and $b''$ for the existential quantifiers, we have that $c'\in \acl_\cup(Mc)$ (since $\monster_\cup\models \psi(c',c,m)$) and $b''\in \acl_\cup(Mb)$ (since $\monster_\cup\models \theta(b'',b,m)$). Further, $b''\equiv_{MB} b'$, so $\varphi(x',b'',m)$ Kim-forks over $M$ in $\monster_i$. Since $\monster_\cup\models \varphi(c',b'',m)$, we have $c'\enind[K]_M b''$ in $\monster_i$, so $c\nind_M b$.

\emph{Witnessing:} For this property, we assume $\Th(\monster_\cup)$ is $\NSOP_1$. Suppose again $A\nind_M B$. We use the same notation as in the proof of strong finite character, and we seek to show that $\chi(x,b,m)$ Kim-forks over $M$ in $\monster_\cup$. 

If not, then by compactness we can find a complete $L_\cup$-type $p(x)$ over $Mb$ which contains $\chi(x,b,m)$ but does not Kim-fork over $M$. Let $e$ realize this type. Then we have $e\eind[K]_M b$ in $\monster_\cup$, so by Fact~\ref{fact:kimacl}, $\acl_\cup(Me)\eind[K]_M \acl_\cup(Mb)$ in $\monster_\cup$. But since $\monster_\cup\models \chi(e,b,m)$, there is some $e'\in \acl_\cup(Me)$ and some $b''\in \acl_\cup(Mb)$ such that $\monster_\cup\models \varphi(e',b'',m)$. This is a contradiction, since by Lemma~\ref{lem:kimreduct} and the fact that $\tp_{L_\cup}(b''/M) = \tp_{L_\cup}(b'/M)$, $\varphi(x',b'',m)$ Kim-forks over $M$ in $\monster_\cup$. 
\end{proof}

\noindent It remains to show that $\ind$ satisfies the independence theorem, which could also be called ``independent $3$-amalgamation over models''. In Theorem~\ref{thm:dap} above, we used stationarity of an independence relation in $T_\cap$ to establish disjoint $2$-amalgamation over algebraically closed sets. For the independence theorem, we need to appeal to a strengthening of stationarity, namely $3$-uniqueness, which holds over models in stable theories. 

%\noindent To prove the independence theorem, we need to address some subtle issues around algebraic closure in the $T_i$ and forking in the stable reduct $T_\cap$.

\medskip\noindent Assume $T$ is stable. Let $B,A_1,A_2,A_3\subseteq \monsterset$ be sets. We say $A_1$, $A_2$, $A_3$ is an \textbf{independent triple over $B$} if $A_2\eind[f]_B A_1$ and $A_3 \eind[f]_B A_1A_2$. By basic properties of forking independence in stable theories (symmetry, base monotonicity, and transitivity), it follows that whenever $\{i,j,k\} = \{1,2,3\}$:
$$A_i\eind[f]_B A_jA_k \quad \text{and}\quad A_i\eind[f]_{BA_j}A_k.$$

\medskip \noindent Now assume the stable theory $T$ has elimination of imaginaries and $B$ is $\acl$-closed. Suppose $A_1$, $A_2$, $A_3$ is an independent triple over $B$. For each $i\in \{1,2,3\}$, let $a_i$ be a tuple enumerating $A_i$. The three types $\tp(a_i/B)$ for $i\in \{1,2,3\}$ uniquely determine $\tp(a_1a_2a_3/B)$ in the following sense. If $A_1'$, $A_2'$, $A_3'$ is another independent triple over $B$, with each $A_i'$ enumerated by a tuple $a_i'$, and if $\tp(a_i/B) = \tp(a_i'/B)$ for all $i\in \{1,2,3\}$, then by stationarity of $\eind[f]$ over $B$, $\tp(a_1a_2a_3/B) = \tp(a_1'a_2'a_3'/B)$. 

\medskip\noindent The analogous statement may not hold if we consider the types of the sets $\acl(A_iA_j)$. Fix again an independent triple $A_1$, $A_2$, $A_3$ over and $\acl$-closed set $B$, and assume that for each $i\in \{1,2,3\}$, $A_i$ is algebraically closed and $B\subseteq A_i$.  For all $i\in \{1,2,3\}$, let $a_i$ be a tuple enumerating $A_i$, and for all $1\leq i<j\leq 3$, let $a_{ij}$ be a tuple extending $a_i$ and $a_j$ and enumerating $\acl(A_iA_j)$. Then we say $T$ has \textbf{$3$-uniqueness over $B$} if for any other independent triple $A_1'$, $A_2'$, $A_3'$ over $B$, with each $A_i$ enumerated by a tuple $a'_i$ and each $\acl(A_iA_j)$ enumerated by a tuple $a'_{ij}$ extending $a_i'$ and $a_j'$, if $\tp(a_i/B) = \tp(a_i'/B)$ for all $i\in \{1,2,3\}$ and $\tp(a_{ij}/B) = \tp(a'_{ij}/B)$ for all $1\leq i < j \leq 3$, then $\tp(a_{12}a_{13}a_{23}/B) = \tp(a_{12}'a_{13}'a_{23}'/B)$. %We say $T$ has \textbf{$3$-uniqueness} if it has $3$-uniqueness over arbitrary algebraically closed sets.

%\medskip \noindent Hrushovski~\cite{HGroupoids} showed that a stable theory has $3$-uniqueness if and only if it eliminates \emph{generalized imaginaries}.
%Generalized imaginaries correspond to definable groupoids.
%Ordinary amalgamation over algebraically closed sets in the sense of~\cite[Proposition 2.23]{firstpaper} requires weak elimination of imaginaries in $T_\cap$.
%It is therefore natural that independent $3$-amalgamation in $\widehat{T}$ (the independence theorem, the main component in showing $\widehat{T}$ is $\NSOP_1$) requires elimination of generalized imaginaries in $T_\cap$.

\begin{fact}[{\cite[Proposition 1.6(2)]{DePiroKimMillar}}]\label{fact:3uniqueness}
Every stable theory with elimination of imaginaries has $3$-uniqueness over models. 
\end{fact}

\noindent In its general form, $3$-uniqueness was introduced by Hrushovski in~\cite{HGroupoids}. The  reference~\cite{DePiroKimMillar} cited for Fact~\ref{fact:3uniqueness} actually shows that stable theories satisfy a stronger property, called $n$-complete amalgamation, over models. In the special case of $3$-uniqueness, the result follows easily from~\cite[Lemma 4.2]{HGroupoids} together with the fact that in a stable theory, any type which does not fork over a model $M$ is finitely satisfiable in $M$.

\medskip  \noindent We are now ready to proceed with the main theorem. See Section~\ref{sec:indreduct} for the definition of the $T_\cap$-generic independence theorem. 

\begin{thm}\label{thm:preservation}
Assume $T_\cap$ is stable, each $T_i$ is $\NSOP_1$, and $\eind[K]$ satisfies the $T_\cap$-generic independence theorem in each $T_i$. Then $T^*_\cup$ is $\NSOP_1$ and $\ind = \eind[K]$.
\end{thm}

\begin{proof}
To make our notation more compact, in this proof we write $[X]_\square$ for $\acl_\square(X)$ when $\square\in I\cup \{\cap,\cup\}$. By Remark~\ref{rem: robust}, we may assume that $T_\cap$ eliminates imaginaries.

Fix a monster model $\monster_\cup\models T_\cup^*$ (equivalently, a completion of $T_\cup^*$). We would like to show that $\Th(\monster_\cup)$ is $\NSOP_1$. By Proposition~\ref{prop:properties}, $\ind$ satisfies invariance, existence, monotonicity, symmetry, and strong finite character. If we show that $\ind$ satisfies the independence theorem, then by Fact~\ref{fact:char}, $\Th(\monster_\cup)$ is $\NSOP_1$. Proposition~\ref{prop:properties} then tells us that $\ind$ satisfies witnessing, so $\ind_M = \eind[K]_M$ for all models $\sM_\cup\elesub \monster_\cup$. 
Toward the independence theorem, suppose we are given $A,A',B,C\subseteq \monsterset$ and $\sM_\cup\elesub\monster_\cup$ such that: $$\tp_{L_\cup}(A/M) = \tp_{L_\cup}(A'/M),\quad A\ind_M B,\quad A'\ind_M C,\quad \text{and}\quad B\ind_M C.$$ By adding elements to $A$, $A'$, $B$, and $C$, we may assume $A = [MA]_\cup$, $A' = [MA']_\cup$, $B = [MB]_\cup$, and $C = [MC]_\cup$. Then by definition of $\ind$, we have, for all $i\in I$: $$\tp_{L_i}(A/M) = \tp_{L_i}(A'/M), \quad A\eind[K]_M B,\quad A'\eind[K]_M C,\quad \text{and}\quad B\eind[K]_M  C \quad \text{in }\monster_i.$$

Let us fix some notation. Since $\tp_{L_i}(A/M) = \tp_{L_i}(A'/M)$, there is a partial elementary bijection $A \to A'$ which fixes $M$ pointwise. Let $x_A$ be a tuple of variables simultaneously enumerating $A$ and $A'$ according to this bijection. Let $x_B$ and $x_C$ be tuples enumerating $B$ and $C$, respectively. For all $\square\in I\cup \{\cap,\cup\}$, we define:
\begin{align*}
p_{\square}^{AB}(x_{AB}) &= \tp_{L_\square}([AB]_\cap/M)\\ 
p^{AC}_\square(x_{AC}) &= \tp_{L_\square}([A'_iC]_\cap/M)\\
p^{BC}_\square(x_{BC}) &= \tp_{L_\square}([BC]_\cap/M)
\end{align*} where $x_{AB}$ is a tuple of variables extending $x_A$ and $x_B$ and enumerating $[AB]_\cap$, and similarly for $x_{AC}$ and $x_{BC}$. We additionally define: \begin{align*}
q_\square^{AB}(x_{AB},y_{AB}) &= \tp_{L_\square}([AB]_\cup/M)\\
q_\square^{AC}(x_{AC},y_{AC}) &= \tp_{L_\square}([A'C]_\cup/M)\\
q_\square^{BC}(x_{BC},y_{BC}) &= \tp_{L_\square}([BC]_\cup/M)
\end{align*}
where $y_{AB}$ is a tuple of variables enumerating $[AB]_\cup\setminus [AB]_\cap$, and similarly for $y_{AC}$ and $y_{BC}$. 

Now $T_\cap$ is stable with elimination of imaginaries, so by Proposition~\ref{prop:relaclcomp}(3), $\bigcup_{i\in I} q_i^{AB}$ axiomatizes $q_\cup^{AB}$ relative to $T_\cup^*$. Similarly, $\bigcup_{i\in I} q_i^{AC}$ axiomatizes $q_\cup^{AC}$ and $\bigcup_{i\in I} q_i^{BC}$ axiomatizes $q_\cup^{BC}$.

For the moment, fix $i\in I$.
Since $\eind[K]$ in $T_i$ satisfies the $T_\cap$-generic independence theorem, there exists $A''_i$ such that $\tp_{L_i}(A''_i/MB) = \tp_{L_i}(A/MB)$, $\tp_{L_i}(A''_i/MC) = \tp_{L_i}(A'/MC)$, $A''_i\eind[K]_M BC$ in $\monster_i$, and further, in $\monster_\cap$:
\begin{align}
[A''_iB]_i&\eind[f]_{A''_iB} [A''_iC]_i[BC]_i\\
[A''_iC]_i&\eind[f]_{A''_iC} [A''_iB]_i[BC]_i\\
[BC]_i&\eind[f]_{BC} [A''_iB]_i[A''_iC]_i.
\end{align}

Since Kim-independence is preserved under reducts (Lemma~\ref{lem:kimreduct}) and agrees with forking independence in a stable theory (Fact~\ref{fact:kimsimple}), we have $B\eind[f]_M C$ and $A''_i\eind[f]_M BC$ in $\monster_\cap$, so $A''_i$, $B$, $C$ is an independent triple over $M$ in $\monster_\cap$.

Let $E_i = A''_iBC$. By base monotonicity from (1), (2), and (3),  $[A''_iB]_i$, $[A''_iC]_i$, $[BC]_i$ is an independent triple over $E_i$ in $\monster_\cap$. In particular: 
\begin{align}
[BC]_i\eind[f]_{E_i}[A''_iB]_i[A''_iC]_i.
\end{align}

Our goal is to extend the sets $[A''_iB]_i$, $[A''_iC]_i$, $[BC]_i$ to realizations of the types $q_i^{AB}$, $q_i^{AC}$, and $q_i^{BC}$, in such a way that these realizations also form an independent triple over $E_i$ in $\monster_\cap$

By Fact~\ref{fact:kimacl}, $A_i''\eind[K]_M BC$ implies $A_i''\eind[K]_M  [BC]_i$ in $\monster_i$.
By reasonable extension (Theorem~\ref{thm:reasext}), we can find a realization $A'''_i$ of $\tp(A''_i/[BC]_i)$ such that: $$ A'''_i\eind[K]_M [BC]_\cup \quad\text{and}\quad A'''_i\eind[r]_{[BC]_i}[BC]_\cup \quad \text{in }\monster_i.$$ Let $\sigma$ be an automorphism of $\monster_i$ which fixes $[BC]_i$ pointwise and moves $A'''_i$ to $A''_i$, and let $D^{BC}_i = \sigma([BC]_\cup)$. Then $D^{BC}_i$ realizes $q_i^{BC} = \tp_{L_i}([BC]_\cup/M)$, and we have in $\monster_i$:
\begin{align}
A''_i&\eind[K]_M D^{BC}_i \\
A''_i&\eind[r]_{[BC]_i} D^{BC}_i.
\end{align}

By definition of $\eind[r]$, since $[A''_iB]_i$ and $[A''_iC]_i$ are both subsets of $[A''_i [BC]_i]_i$, we have in $\monster_\cap$: 
\begin{align}
D^{BC}_i &\eind[f]_{[BC]_i} [A''_iB]_i[A''_iC]_i \quad \text{by  symmetry and monotonicity, from (6)}\\
D^{BC}_i &\eind[f]_{E_i[BC]_i} [A''_iB]_i[A''_iC]_i \quad \text{by base monotonicity, from (7)}\\
D^{BC}_i &\eind[f]_{E_i} [A''_iB]_i[A''_iC]_i \quad \text{by transitivity, from (4) and (8).}
\end{align} 

Thus $[A''_iB]_i$, $[A''_iC]_i$, $D_i^{BC}$ is an independent triple over $E_i$ in $\monster_\cap$. In particular:
\begin{align}
%D^{BC}_i&\eind[f]_{E_i}[A''_iC]_i\\
[A''_iB]_i &\eind[f]_{E_i}[A''_iC]_iD^{BC}_i.
\end{align}

Since $\tp_{L_i}(A''_iB/M) = \tp_{L_i}(AB/M)$, we can extend $[A''_iB]_i$ to a realization $D^{AB}_i$ of $q_i^{AB} = \tp_{L_i}([AB]_\cup/M)$. Further, since $\eind[f]$ in $\monster_\cap$ is extendable to $\monster_i$ (Proposition~\ref{prop:stationarity}), we may assume that in $\monster_\cap$:
\begin{align}
D^{AB}_i &\eind[f]_{[A''_iB]_i} [A''_iC]_iD_i^{BC}.
\end{align}

Repeating the argument above, we have in $\monster_\cap$:
\begin{align}
D^{AB}_i &\eind[f]_{E_i[A''_iB]_i}[A''_iC]_iD^{BC}_i \quad \text{by base monotonicity, from (11)}\\
D^{AB}_i &\eind[f]_{E_i} [A''_iC]_iD^{BC}_i \quad \text{by transitivity, from (10) and (12).}
\end{align}

Thus $[A''_iC]_i$, $D_i^{AB}$, $D_i^{BC}$ is an independent triple over $E_i$ in $\monster_\cap$. In particular:
\begin{align}
[A''_iC]_i &\eind[f]_{E_i}D^{AB}_iD^{BC}_i.
\end{align}

Similarly, since $\tp_{L_i}(A''_iC/M) = \tp_{L_i}(A'C/M)$, we can extend $[A''_iC]_i$ to a realization $D^{AC}_i$ of $q_i^{AC} = \tp_{L_i}([A'C]_\cup/M)$ such that in  $\monster_\cap$: 
\begin{align}
D^{AC}_i &\eind[f]_{[A''_iC]_i} D^{AB}_iD_i^{BC}.
\end{align}

Repeating the argument one more time, we have in $\monster_\cap$:
\begin{align}
D^{AC}_i &\eind[f]_{E_i[A''_iC]_i}D^{AB}_iD^{BC}_i \quad \text{by base monotonicity, from (15)}\\
D^{AC}_i &\eind[f]_{E_i} D^{AB}_iD^{BC}_i \quad \text{by transitivity, from (14) and (16).}
\end{align}

Thus $D_i^{AB}$, $D_i^{AC}$, $D_i^{BC}$ is an independent triple over $E_i$ in $\monster_\cap$. 

With all the pieces in place, we set $$q_i^{ABC}(x_{AB},y_{AB},x_{AC},y_{AC},x_{BC},y_{BC}) = \tp_{L_i}(D^{AB}_i,D^{AC}_i,D^{BC}_i/M).$$

We now claim that for all $i,j\in I$, the restrictions of $q_i^{ABC}$ and $q_j^{ABC}$ to $L_\cap$ are equal: $$\tp_{L_\cap}(D^{AB}_i,D^{AC}_i,D^{BC}_i/M) = \tp_{L_\cap}(D^{AB}_j,D^{AC}_j,D^{BC}_j/M).$$ 

As noted above, $A''_i$, $B$, $C$ and $A''_j$, $B$, $C$ are independent triples over $M$ in $\monster_\cap$. We also have:
\begin{align*}
\tp_{L_\cap}([A''_iB]_\cap/M) &= \tp_{L_\cap}([A''_jB]_\cap/M) = p_\cap^{AB}\\ \tp_{L_\cap}([A''_iC]_\cap/M) &= \tp_{L_\cap}([A''_jC]_\cap/M) = p_\cap^{AC}\\ \tp_{L_\cap}([BC]_\cap/M) &= p_\cap^{BC}.
\end{align*}
So by $3$-uniqueness over $M$ (Fact~\ref{fact:3uniqueness}), $$\tp_{L_{\cap}}([A_i''B]_\cap,[A_i''C]_\cap,[BC]_\cap/M) = \tp_{L_{\cap}}([A_j''B]_\cap,[A_j''C]_\cap,[BC]_\cap/M).$$ It follows that there exists a partial $L_\cap$-elementary bijection $\tau\colon [E_i]_\cap \to [E_j]_\cap$ extending the identity on $[BC]_\cap$ and the elementary bijections $[A''_iB]_\cap \to [A''_jB]_\cap$ and  $[A''_iC]_\cap \to [A''_jC]_\cap$ given by the enumerations of these sets by the variables $x_{AB}$ and $x_{AC}$.

Extending $\tau$ to an automorphism of $\monster_\cap$, we may identify $[E_i]_\cap$ with $[E_j]_\cap$ and call this set just $[E]_\cap$. This also has the effect of identifying $A_i''$ with $A_j''$ and $E_i$ with $E_j$, and similarly we call these sets just $A''$ and $E$, respectively. 

By (1) and (11) above, in $\monster_\cap$: $$[A''B]_i\eind[f]_{A''B} E\quad\text{and}\quad D_i^{AB}\eind[f]_{[A''B]_i} E$$ so by transitivity and closing under $\acl_\cap$, and applying the same argument to $D_j^{AB}$:  $$D_i^{AB}\eind[f]_{[A''B]_\cap} [E]_\cap\quad\text{and}\quad D_j^{AB}\eind[f]_{[A''B]_\cap} [E]_\cap.$$
We have $\tp_{L_\cap}(D_i^{AB}/[A''B]_\cap) = \tp_{L_\cap}(D_j^{AB}/[A''B]_\cap)$, since both agree with $q_\cap^{AB}$, so by stationarity, $\tp_{L_\cap}(D_i^{AB}/[E]_\cap) = \tp_{L_\cap}(D_j^{AB}/[E]_\cap)$.

The same argument, using (2), (3), (7), and (15), shows $\tp_{L_\cap}(D_i^{AC}/[E]_\cap) = \tp_{L_\cap}(D_j^{AC}/[E]_\cap)$ and $\tp_{L_\cap}(D_i^{BC}/[E]_\cap) = \tp_{L_\cap}(D_j^{BC}/[E]_\cap)$.

Since forking independence over a set agrees with forking independence over the algebraic closure of that set, $D_i^{AB}$, $D_i^{AC}, D_i^{BC}$ and $D_j^{AB}$, $D_j^{AC}, D_j^{BC}$ are both independent triples over $[E]_\cap$. So by stationarity: $$\tp_{L_\cap}(D^{AB}_i,D^{AC}_i,D^{BC}_i/[E]_\cap) = \tp_{L_\cap}(D^{AB}_j,D^{AC}_j,D^{BC}_j/[E]_\cap).$$
In particular, the restrictions of these types to $M$ agree, which establishes the claim. It follows that that there is a complete $L_\cap$-type $q_\cap^{ABC}$ over $M$ such that $q_\cap^{ABC} \subseteq q_i^{ABC}$ for all $i\in I$. 

By Proposition~\ref{prop:realizingtypes}, $\bigcup_{i\in I}q_i^{ABC}$ is consistent and realized in $\monster_\cup$ by sets $(D_{AB},D_{AC},D_{BC})$. Let $A^*$, $B^*$, and $C^*$ be the subsets enumerated by $x_A$, $x_B$, and $x_C$, respectively.

Since $D_{BC}$ satisfies $q_i^{BC}$ for all $i\in I$, $D_{BC}$ satisfies $q_\cup^{BC} = \tp_{L_\cup}(\acl_\cup(BC)/M)$. So after applying an automorphism of $\monster_\cup$ which fixes $M$ pointwise, we may assume $D_{BC} = \acl_\cup(BC)$, and in particular $B^* = B$ and $C^* = C$. Similarly, $D_{AB}$ and $D_{AC}$ satisfy $q_i^{AB}$ and $q_i^{AC}$ for all $i\in I$, so they satisfy $q_\cup^{AB}$ and $q_\cup^{AC}$, and hence $\tp_{L_\cup}(A^*B/M) = \tp_{L_\cup}(AB/M)$ and $\tp_{L_\cup}(A^*C/M) = \tp_{L_\cup}(A'C/M)$. 

For all $i\in I$,  $\tp_{L_i}(A^*D_{BC}/M) = \tp_{L_i}(A''_iD^{BC}_i/M)\subseteq    q_i^{ABC}$, so $A^*\eind[K]_M \acl_\cup(BC)$ in $\monster_i$. Since $\acl_\cup(MA^*) = A^*$, we have $A^*\ind_M BC$, as desired. 
\end{proof}

\noindent We will now draw some immediate corollaries, using the sufficient conditions for the $T_\cap$-generic independence theorem derived in Section~\ref{sec:indreduct}. 

\begin{cor}
\label{cor:stable-case} Assume $T_\cap$ is stable and for each $i$, $T_i$ is  stable or $T_i$ is $\NSOP_1$ and relatively disintegrated. Then $T_\cup^*$ is $\NSOP_1$ and $\ind =  \eind[K]$. 
\end{cor}

\begin{cor}
\label{cor:disforking-case} Assume $T_\cap$ is stable with disintegrated forking. If each $T_i$ is $\NSOP_1$,  then $T_\cup^*$ is $\NSOP_1$ and $\ind =  \eind[K]$. 
\end{cor}

\noindent
In many of the examples in Section~\ref{section:examples}, $T_\cap$ is interpretable in the theory of an  infinite set.
If $T$ has disintegrated forking and finite $U$-rank, then any theory interpretable in $T$ has disintegrated forking~\cite[\hspace{-.1cm}\mathsection 5]{Goode-trivial}.
So any theory interpretable in the theory of an infinite set has disintegrated forking, and Corollary~\ref{cor:triv-forking} follows from Corollary~\ref{cor:disforking-case}.

\begin{cor}
\label{cor:triv-forking}
Assume $T_\cap$ is interpretable in the theory of an infinite set.
If each $T_i$ is $\NSOP_1$, then $T^*_\cup$ is $\NSOP_1$ and $\ind =  \eind[K]$. 
\end{cor}

\noindent Finally, we observe that Corollary~\ref{cor:stable-case} implies that an unstable $\NIP$ theory cannot be decomposed as a fusion of stable theories.

\begin{cor}
\label{cor:nip}
Assume each $T_i$ is stable. If $T^*_\cup$ is $\NIP$, then $T^*_\cup$ is stable. 
\end{cor}
\begin{proof}
By Corollary~\ref{cor:stable-case}, $T_\cup^*$ is $\NSOP_1$, and hence does not have the strict order property. Every $\NIP$ theory without the strict order property is stable.  
\end{proof}

\subsection{Simplicity} \label{Sec: Simplicity}
Having obtained sufficient conditions for the preservation of $\NSOP_1$, we can now use Fact~\ref{fact:kimsimple} to improve this to preservation of simplicity, under stronger hypotheses. 

\medskip \noindent We do this in two ways. In Theorem~\ref{thm:rel-dis-simple}, we assume that each $T_i$ is simple and relatively disintegrated (see Section~\ref{sec:indreduct} for the definition). In Theorem~\ref{thm:simplepreservation}, we relax the hypotheses on a single $T_{i^*}$, only requiring the $T_\cap$-generic independence theorem for $\ind[f]$ in $T_{i^*}$; but in this case we have to assume that all the other $T_i$ ($i\neq i^*$) fail to add new algebraicity or forking to $T_\cap$. These two theorems generalize Propositions 6.3.13 and 6.3.15 in~\cite{wagner}, which concern the special case when $T_\cap$ is the theory of an infinite set. The second is inspired by a theorem of Tsuboi from~\cite{Tsuboi}, see Fact~\ref{fact:tsuboi} below.

\begin{thm}
\label{thm:rel-dis-simple}
Assume $T_\cap$ is stable. 
If each $T_i$ is simple and relatively disintegrated, then $T_\cup^*$ is simple. 
\end{thm}

\begin{proof}
By Corollary~\ref{cor:stable-case}, $T_\cup^*$ is $\NSOP_1$, and $\ind = \eind[K]$. By Fact~\ref{fact:kimsimple}, it suffices to show that $\eind[K]$ satisfies base monotonicity over models.

So fix $M\prec N \prec \monster_\cup$, $M\subseteq A$, and $N\subseteq B$. Assume that $A = \acl_\cup(A)$, $B = \acl_\cup(B)$, and $A\eind[K]_M B$. It suffices to show that $A\eind[K]_N B$. Since each $T_i$ is simple, $\eind[f]$ satisfies base monotonicity in $\monster_i$. So we have: 
\begin{align*}
A \eind[K]_M B \text{ in }\monster_\cup &\Rightarrow A \eind[f]_M B \text{ in }\monster_i\text{ for all }i\in I \\
&\Rightarrow A \eind[f]_N B \text{ in }\monster_i\text{ for all }i\in I\\
&\Rightarrow \acl_i(NA) \eind[f]_N B \text{ in }\monster_i\text{ for all }i\in I
\end{align*}
If we can improve this to $$\acl_\cup(NA) \eind[f]_N B \text{ in }\monster_i\text{ for all }i\in I,$$ then it follows from the characterization of $\eind[K]$ that $A\eind[K]_N B$ in $\monster_\cup$, as desired. 

Since each $\acl_i$ is disintegrated relative to $\acl_\cap$, we have $$\acl_i(NA) = \acl_\cap(\acl_i(N)\acl_i(A)) = \acl_\cap(NA),$$ since $N$ and $A$ are $\acl_\cup$-closed. Thus $\acl_\cap(NA)$ is $\acl_i$-closed for all $i\in I$, and hence $\acl_\cup$-closed, since  $\acl_\cup = \ccl$ by Proposition~\ref{prop:relaclcomp}. So $$\acl_\cup(NA) = \acl_\cap(NA) = \acl_i(NA),$$  and we have already proven what we wanted. 
\end{proof}

\noindent The next result is inspired by the following theorem of Tsuboi.

\begin{fact}[\cite{Tsuboi}, \cite{wagner} Proposition 6.3.15]
\label{fact:tsuboi}
Suppose that $I = \{1,2\}$, $L_1 \cap L_2 = \emptyset$, and $T_1$ and $T_2$ are simple and eliminate $\exists^\infty$. Then $T_\cup^*$ exists.
If $\acl_1$ is trivial and $T_1$ has $U$-rank one, then $T^*_\cup$ is simple.
\end{fact}

\noindent
A theory $T$ has trivial algebraic closure and $U$-rank one if and only if algebraic closure and forking in $T$ agrees with algebraic closure and forking in the theory of an infinite set.
So Theorem~\ref{thm:simplepreservation} generalizes Fact~\ref{fact:tsuboi}.

\begin{thm}
\label{thm:simplepreservation}
Assume $T_\cap$ is stable and each $T_i$ is simple. Fix $i^*\in I$ and assume that:
\begin{enumerate}
    \item $\eind[f]$ in  $T_{i^*}$ satisfies the $T_\cap$-generic independence theorem. 
    \item For all $i\neq i^*$, $\acl_i = \acl_\cap$.
    \item For all $i\neq i^*$ and all sets $A$, $B$, $C$, we have $A\eind[f]_C B$ in $\monster_i$ if and only if $A\eind[f]_C B$ in $\monster_\cap$.
\end{enumerate}
Then $T_\cup^*$ is simple. 
\end{thm}
\begin{proof}
For all $i\neq i^*$, $T_i$ is relatively disintegrated, and hence $\eind[f]$ in $T_i$ satisfies the $T_\cap$-generic independence theorem by Proposition~\ref{prop:reldisthreebody}. So 
by Theorem~\ref{thm:preservation}, $T_\cup^*$ is $\NSOP_1$, and $\ind = \eind[K]$.
By Fact~\ref{fact:kimsimple}, it suffices to show that $\eind[K]$ satisfies base monotonicity over models.

We begin just as in the proof of Theorem~\ref{thm:rel-dis-simple}, fixing $M\prec N \prec \monster_\cup$, $M\subseteq A$, and $N\subseteq B$ such that $A = \acl_\cup(A)$, $B = \acl_\cup(B)$, and $A\eind[K]_M B$. Then we can show that for all $i\in I$, $$\acl_i(NA) \eind[f]_N B \text{ in }\monster_i,$$ and we are done if we can improve this to $$\acl_\cup(NA) \eind[f]_N B \text{ in }\monster_i.$$

Unlike the relatively disintegrated case, we may not have $\acl_i(NA) = \acl_\cup(NA)$ for all $i$. But by our hypothesis and Proposition~\ref{prop:relaclcomp}, $\acl_\cup = \acl_{i^*}$, so $$\acl_\cup(NA) \eind[f]_N B \text{ in }\monster_{i^*}.$$
Taking the reduct (Lemma~\ref{lem:kimreduct}), we also have $$\acl_\cup(NA) \eind[f]_N B \text{ in }\monster_{\cap}.$$ And since for all $i\neq i^*$, forking in $\monster_i$ agrees with forking in $\monster_\cap$, we have $$\acl_\cup(NA) \eind[f]_N B \text{ in $\monster_{i}$ for all $i\neq i^*$}.$$ This completes the proof.
\end{proof}

\section{$\aleph_0$-categoricity}\label{section:categoricity}
\noindent
We do not assume that $T_\cup^*$ exists in this section.

\subsection{Existence and Preservation} \label{sec: Preservationcountablecat}

Applying the preservation results from Section~\ref{sec:analysis}, we show if each $T_i$ is $\aleph_0$-categorical and certain extra hypotheses hold, then $T^{*}_\cup$ exists and is $\aleph_0$-categorical.
This section closely follows work of Pillay and Tsuboi~\cite{PillayTsuboi}.
Our principle innovation is to note that their assumptions are satisfied by a number of examples, see Section~\ref{section:examples}.

\begin{prop}\label{prop:ccelementary}
Assume $T_\cap$ admits a stationary and extendable independence relation.
Assume also that all languages have only finitely many sorts.
Suppose that each $T_i$ is $\aleph_0$-categorical and that there is some $i^*\in I$ such that $\acl_i(A) = \acl_\cap(A)$ for all $i\neq i^*$.
Then $T^*_\cup$ exists.
\end{prop}

\begin{proof}
A $T_\cup$-model $\sM_\cup$ has the \textbf{joint consistency property} if for every finite $B\subseteq M$ such that $B = \acl_{i^*}(B)$ and every family $(p_i(x))_{i\in J}$ such that $J$ is a finite subset of $I$, if $p_i(x)$ is a complete $L_i$-type over $B$ for all $i\in J$, and the $p_i$ have a common restriction $p_\cap(x)$ to $L_\cap$, then $\bigcup_{i\in I} p_i(x)$ is realized in $\sM_\cup$.

Note that the joint consistency property is elementary. Indeed, by $\aleph_0$-categoricity, there is an $L_{i^*}$-formula $\psi(y)$ expressing the property that the set $B$ enumerated by a tuple $b$ is $\acl_{i^*}$-closed. Since $B$ is finite, every complete $L_i$-type $p_i(x)$ over $B$ is isolated by a single formula. And the property that the $L_i$-formula $\varphi_i(x,b)$ isolates a complete $L_i$-type over $B$ whose restriction to $L_\cap$ is isolated by the $L_\cap$-formula $\varphi_\cap(x,b)$ is definable by a formula $\theta_{\varphi_i,\varphi_\cap}(b)$. So the class of $T_\cup$-models with the joint consistency property is axiomatized by $T_\cup$ together with sentences of the form \[\forall y\, \left[ \left(\psi(y)\land \bigwedge_{i\in J} \theta_{\varphi_i,\varphi_\cap}(y)\right) \rightarrow \exists x\, \bigwedge_{i\in J} \varphi_i(x,y)\right].\]

It remains to show that a structure $\sM_\cup$ is interpolative if and only if it has the joint consistency property. So suppose $\sM_\cup$ is interpolative, let $B$ and $(p_i(x))_{i\in J}$ be as in the definition of the joint consistency property, and suppose for contradiction that $\bigcup_{i\in J} p_i(x)$ is not realized in $\sM_\cup$. Note that since $B$ is $\acl_{i^*}$-closed, it is also $\acl_i$-closed for all $i\neq i^*$, since $\acl_i(B) = \acl_\cap(B)\subseteq \acl_{i^*}(B) = B$. 

Each $p_i(x)$ is isolated by a single $L_i$-formula $\varphi_i(x,b)$, and \[\sM_\cup \models \lnot \exists x\, \bigwedge_{i\in J} \varphi_i(x,b).\] It follows that the $\varphi_i$ are separated by a family of $L_\cap$-formulas $(\psi^i(x,c_i))_{i\in J}$. Let $C = B\cup \{c_i\mid i\in J\}$. By full existence for $\ind$ in $T_i$, since $B$ is $\acl_i$-closed, $p_i(x)$ has an extension to a type $q_i(x)$ over $C$ such that for any realization $a_i$ of $q_i(x)$, $a\ind_B C$. By stationarity, the types $q_i(x)$ have a common restriction $q_\cap$ to $L_\cap$. 
Now for all $i\in J$, since $\varphi_i(x,b)\in p_i(x)$, $\psi^i(x,c_i)\in q_i(x)$, and hence $\psi^i(x,c_i)\in q_\cap(x)$. This is a contradiction, since $\{\psi^i(x,c_i)\mid i\in J\}$ is inconsistent. 

Conversely, suppose $\sM_\cup$ has the joint consistency property. Let $(\varphi_i(x,a_i))_{i\in J}$ be a family of formulas which are not separated. Let $B = \acl_{i^*}((a_i)_{i\in J})$. Since $T_{i^*}$ is $\aleph_0$-categorical and $J$ is finite, $B$ is finite. For each $i\in J$, there is an $L_\cap$-formula $\psi^i(x,b)$ such that $\sM_\cup\models \psi^i(a,b)$ if and only if $\tp_{L_\cap}(a/B)$ is consistent with $\varphi_i(x,a_i)$ (we may take $\psi^i(x,b)$ to be the disjunction of formulas isolating each of the finitely many such types). Since the formulas $\psi^i(x,b)$ do not separate the formulas $\varphi_i(x,a_i)$, there must be some element $a\in M^x$ satisfying $\bigwedge_{i\in J} \psi^i(x,b)$. Then $p_\cap(x) = \tp_{L_\cap}(a/B)$ is consistent with each $\varphi_i(x,a_i)$, so $p_\cap(x)\cup \{\varphi_i(x,a_i)\}$ can be extended to a complete $L_i$-type $p_i(x)$ over $B$. By the joint consistency property, there is some element in $M^x$ realizing $\bigcup_{i\in J}p_i(x)$, and in particular satisfying $\bigwedge_{i\in J} \varphi_i(x,a_i)$. 
\end{proof}

\noindent
Theorem~\ref{thm:ccpres} follows by a type-counting argument as in Proposition~\ref{prop:stableinkappa}.

\begin{thm}
\label{thm:ccpres}
Assume the hypotheses of Proposition~\ref{prop:ccelementary}, and let $T_\cup^*$ be the interpolative fusion.
Assume additionally that $I$ is finite.
Then every completion of $T_\cup^*$ is $\aleph_0$-categorical.
\end{thm}

\begin{proof}
Let $\widehat{T}$ be a completion of $T_\cup^*$. It suffices to show that for any finite tuple of variables $x$, there are only finitely many $L_\cup$-types over the empty set in the variables $x$ relative to $\widehat{T}$. Since $\acl_\cup = \acl_{i^*}$ is uniformly locally finite, there is an upper bound $m$ on the size of $\acl_\cup(a)$ for any tuple $a\in M^x$ when $M\models \widehat{T}$. 

By Proposition~\ref{prop:relaclcomp}, $\tp_{L_\cup}(\acl_\cup(a))$ is determined by $\bigcup_{i\in I} \tp_{L_i}(\acl_\cup(a))$. So the number of possible $L_\cup$-types of $a$ is bounded above by the product over all $i$ of the number of $L_i$-types of $m$-tuples relative to $T_i$. This is finite, since $I$ is finite and each $T_i$ is $\aleph_0$-categorical.
\end{proof}

\noindent Corollary~\ref{cor:ccpres} follows from Corollary~\ref{cor:completeness}.

\begin{cor}
\label{cor:ccpres}
Assume the hypotheses of Proposition~\ref{prop:ccelementary}.
Suppose that $I$ is finite and $\emptyset$ is $\acl_i$-closed for all $i \in I$.
Then $T^*_\cup$ is complete  and $\aleph_0$-categorical.
\end{cor}

\noindent
The following result of Pillay and Tsuboi is a special case of Theorem~\ref{thm:ccpres}.

\begin{cor}[{\cite[Corollary 5]{PillayTsuboi}}] \label{Cor: PillayTsuboi}
Assume $T_\cap$ is stable with weak elimination of imaginaries. Let $I = \{1,2\}$, suppose $T_1$ and $T_2$ are $\aleph_0$-categorical single-sorted theories, and suppose $\acl_1(A) = \acl_\cap(A)$ for all $A\subseteq \monsterset_1$. Then $T_\cup$ admits an $\aleph_0$-categorical completion.
\end{cor}

% \noindent
% We describe an example which shows that the assumptions of the Corollary~\ref{cor:ccpres} are sharp.
% Let $L_\cap$ be the two sorted empty language, $T_\cap$ be the theory of two disjoint infinite sets, $I = \{1,2\}$, and for each $i \in I$, $L_i = \{f_i\}$ for a unary function symbol $f_i$ from the first sort to the second, and $T_i$ be the $L_i$-theory such that $(M,N;f_i) \models T_i$ if and only if $f : M \to N$ is a bijection.
% Note that each $T_i$ is interpretable in the theory of an infinite set, and is hence $\aleph_0$-categorical, $\aleph_0$-stable, and eliminates $\exists^\infty$.
% It is the easy to see that the empty set is algebraically closed in $T_i$.
% By Corollary~\ref{cor:minh} $T^*_\cup$ exists and by Corollary~\ref{cor:completeness} $T^*_\cup$ is complete.
% Given $(M,N;f_1,f_2) \models T^*_\cup$ and $a \in M$ we let we definable $a_0 = a$, $a_{2n + 1} = f_1(a_{2n})$ and $a_{2n + 2} = f^{-1}_2(a_{2n + 1})$ for all $n$.
% It is easy to see that there is $(M,N;f_1,f_2) \models T^*_\cup$ and $a \in M$ such that $a_n \ne a_m$ when $n \ne m$.
% So $\acl_\cup(a)$ is not finite, hence $T^*_\cup$ is not $\aleph_0$-categorical.
% \textcolor{blue}{this is just the theory of a set with a generic bijection}

\subsection{A counterexample}
\label{section:counterexample}
We now construct an example where the interpolative fusion of two
$\mathrm{NSOP}_1$ theories with trivial $\acl$ over a unstable reduct $T_\cap$ exists but is not $\mathrm{NSOP}_1$. This demonstrates that the stability of $T_\cap$ in Theorem~\ref{thm:preservation} is necessary.

\medskip\noindent
Let $\sK$ and $\sK'$ be \Fraisse classes in the languages $L\subseteq L'$, respectively. We say $\sK'$ is a \textbf{\Fraisse expansion} of $\sK$ if:
\begin{enumerate}
\item $\sK = \{\sA|_L\mid \sA\in \sK'\}$
\item For every $\sA\in \sK$, every one-point extension $\sA\subseteq \sB$ with $\sB\in \sK$, and every expansion $\sA'\in \sK'$ with $\sA'|_L = \sA$, there exists $\sB'\in \sK'$ such that $\sA'$ is an $L'$-substructure of $\sB'$ and $\sB'|_L = \sB$.
\end{enumerate}
See, for example, \cite[Theorem 2.7]{Namalg} for a proof of Fact~\ref{fact:fraisse-expansion}.

\begin{fact}
\label{fact:fraisse-expansion}
Let $\sK'$ and $\sK$ be as above.
Then $\sK'$ is a \Fraisse expansion of $\sK$ if and only if the \Fraisse limit of $\sK'$ is an expansion of the \Fraisse limit of $\sK$.
\end{fact}

\noindent Recall that a $3$-hypergraph is a set $V$ with a symmetric ternary relation $R$ on $V$ such that $R(a,b,c)$ implies that $a,b,c$ are distinct. Finite $3$-hypergraphs form a  \Fraisse class, the \Fraisse limit of which is known as the random $3$-hypergraph. It is well-known that the theory of random $3$-hypergraph is simple but unstable.
We leave the proof of the following Lemma to the reader.

\begin{lem}\label{lem:Fraisseexample}
Let $L =\{R\}$ with  $R$ a ternary relation symbol, and view the \Fraisse class $\sK$  of finite $3$-hypergraphs as a class of $L$-structures. Set $L_1 = \{R,E_1\}$, and $L_2 = \{R,E_2\}$, where $E_1$ and $E_2$ are binary relation symbols. Let $\sK_1$ be the class of all finite $L_1$-structures such that  $R$ is a $3$-hypergraph relation, $E_1$ is a graph  relation (symmetric and anti-reflexive), and
\[\forall xyz\, ((E_1(x,y)\land E_1(y,z)\land E_1(z,x))\rightarrow R(x,y,z)).\]
Let $\sK_2$ be the class of all finite $L_2$-structures such that $R$ is a $3$-hypergraph relation, $E_2$ is a graph  relation, and
\[\forall xyz\, ((E_2(x,y)\land E_2(y,z)\land E_2(z,x))\rightarrow \lnot R(x,y,z)).\]
Then each $\sK_i$ is a \Fraisse class with disjoint amalgamation and is a \Fraisse expansion of $\sK$.
\end{lem}

\noindent
Recall that the class of finite triangle-free graphs is a \Fraisse class, the \Fraisse limit of which is known as the Henson graph. The theory of the Henson graph is $\SOP_3$ and $\NSOP_4$; see for example~\cite{gabe-henson}. In particular, it is $\SOP_1$. We now construct the promised example.

\begin{prop}
There are simple $\aleph_0$-categorical  theories $T_\cap$, $T_1$, $T_2$, each with trivial $\acl$, such that $T^*_\cup$ exists, is complete and $\aleph_0$-categorical, and has $\SOP_3$. In particular, the theory of the Henson graph is interpretable in $T^*_\cup$.
\end{prop}

\begin{proof}
Let $L_\cap = \{R\}$, $L_1 = \{R,E_1\}$, and $L_2 = \{R,E_2\}$, where $R$ is a ternary relation symbol and $E_1$ and $E_2$ are binary relation symbols. 
Let $\sK$, $\sK_1$, and $\sK_2$ be as in Lemma~\ref{lem:Fraisseexample}, let $T_\cap$ be the complete theory of the \Fraisse limit of $\sK$ (the random $3$-hypergraph), and let $T_i$ be the complete theory of the \Fraisse limit of $\sK_i$ for $i = 1$ and $2$. It follows from Lemma~\ref{lem:Fraisseexample} that $T_1$ and $T_2$ have a common set of $L_\cap$-consequences, namely $T_\cap$. 

To show that $T^*_\cup$ exists and is complete and $\aleph_0$-categorical, we will apply Proposition~\ref{prop:ccelementary} and Corollary~\ref{cor:ccpres}.
The theories $T_\cap$, $T_1$, and $T_2$ are \Fraisse limits in finite relational languages and hence are $\aleph_0$-categorical with quantifier elimination. Since $\sK$, $\sK_1$, and $\sK_2$ have disjoint amalgamation, $\acl_1(A) = \acl_2(A) = \acl_\cap(A) = A$ for all sets $A$. 
So it is enough to construct a stationary and extendable independence relation on $T_\cap$.

If $A,B,C$ are subsets of a $T_\cap$-model $(V;R)$ define $A \eind[\otimes]_C B$ if and only if $ABC$ is a free amalgam of $AC$ and $BC$ over $C$, i.e. $AC\cap BC = C$, and if $a,b,c \in ABC$ and $R(a,b,c)$ then either $a,b,c \in AC$ or $a,b,c \in BC$. This is a stationary independence relation on $T_\cap$ (by quantifier elimination), and extendibility to $T_1$ and $T_2$ follows from the fact that $\sK_1$ and $\sK_2$ admit free amalgamation. So $T^*_\cup$ exists by Proposition~\ref{prop:ccelementary} and is complete and $\aleph_0$-categorical by Corollary~\ref{cor:ccpres}.

%Since $T_1$ and $T_2$ have quantifier elimination, they are model complete, so $T^*_\cup$ is the model companion of $T_\cup$. It is also not difficult to show existence and $\aleph_0$-categoricity of the model companion directly; it is again a \Fraisse limit.

Let $(V;R,E_1,E_2)$ be a countable model of $T^*_\cup$.
Define a binary relation $E$ on $V$ by $E(x,y)$ if and only if $E_1(x, y)\land E_2(x, y)$.
It is immediate that $(V;E)$ is an $\aleph_0$-categorical graph. It is triangle-free because if $a,b,c$ is an $E$-triangle, then it is both an $E_1$-triangle and an $E_2$-triangle, and hence $R(a,b,c)$ and $\lnot R(a,b,c)$, contradiction. We show that $(V;E)$ is isomorphic to the Henson graph by verifying that it satisfies the usual extension axioms.
Suppose that $A,B$ are disjoint finite subsets of $V$ and that there are no $E$-edges between elements of $A$.
We need to find $c \in V$ such that $c$ is connected to every element of $A$ and not connected to any element of $B$.

We first show that we can add a new element satisfying these properties. Suppose that $c \notin V$, and let $W = V \cup \{c\}$. We extend $E_1$ and $E_2$ to $W$ by setting $E_1(c,a)$ and $E_2(c,a)$ (and symmetrically, $E_1(a,c)$ and $E_2(a,c)$) for all $a \in A$. We do not add any other instances of $E_1$ or $E_2$. We extend $R$ to $W$ by setting $R(c,a,a')$ (and its symmetrical instances) when $a,a'\in A$ and $c$, $a$, $a'$ is an $E_1$-triangle. We do not add any other instances of $R$. 

For $i\in \{1,2,\cap\}$, let $p_i$ be the complete quantifier-free $L_i$-type of $c$ over $V$. The reduct of $W$ to $L_i$ satisfies the universal theory of $T_i$, so it embeds in a model of $T_i$. Thus $p_i$ is consistent with $T_i$, and by quantifier elimination it axiomatizes a complete $L_i$-type over $V$. Now both $p_1$ and $p_2$ extend $p_\cap$, so by Proposition~\ref{prop:realizingtypes}, $p_1\cup p_2$ is realized by an element $c'$ in an elementary extension of $V$. Since $T_\cup^*$ is $\aleph_0$-categorical and $AB$ is finite, $\tp_{L_\cup}(c'/AB)$ is realized in $V$. By construction, we have $E(c',a)$ for  all $a\in A$ and $\lnot E(c',b)$ for all $b\in B$. This completes the proof. 
\end{proof}
% \medskip\noindent
% We show that $T^*_\cup$ exists by the countable categoricity result.
% We use the free amalgamation independence relation, where $A,B$ are independent over $C$ if $A \cup B \cup C$ is the free amalgam, i.e. any relation in $A \cup B$ is in $A$ or in $B$, this is clearly stationary.
% here we use that the Fraisse classes have free amalgamation.
% Extendibility basically follows because you can take the free amalgam.
% Once $T^*_\cup$ is produced you get an $\aleph_0$-categorical triangle free graph.
% Use existential closedness to get the extension axioms.
\noindent We end this section with a question:
\begin{question}
\label{ques:simple}
% Is there a classification-theoretic upper bound on a fusion of simple theories?
% For example, is there $n$ such that a fusion of simple theories is always $\NSOP_n$?
Must an interpolative fusion of simple theories be $\NSOP$?
\end{question}

% \noindent
% Theories with free amalgamation relations have to be $\NSOP_4$.

\section{Examples}
\label{section:examples}
\noindent
We describe a number of motivating examples which illustrate applications and sharpness of our theorems; see \cite{firstpaper} for other examples. 
The first two examples are among the original motivating examples of unstable simple theories: the random (hyper-)graph and $\ACFA$.
The remaining examples are mainly various kinds of generic constructions which preserve $\NSOP_1$.

%\textcolor{red}{Why are these two corollary here}
%\textcolor{blue}{they are useful in handling the examples.
%previously i would say things like ``$T_\cap$ is interpretable in the theory of equality so $T_\cap$ is $\aleph_0$-categorical and $\aleph_0$-stable, so we can apply Fact~\ref{fact:minh}\ldots".
%now i say ``$T_\cap$ is interpretable in the theory of equality so by Corollary~\ref{cor:minh}\ldots"
%Feel free to move them somewhere else.}

\subsection{Random hypergraphs and relations}
\label{section:hypergraph}
Fix $n \geq 2$.
Let $L = \{E\}$ where $E$ is an $n$-ary relation symbol.
An $n$-hypergraph is an $L$-structure $(V;E)$ such that $E$ is symmetric and $E(a_1,\ldots,a_n)$ implies $a_i \neq a_j$ for all $i \neq j$.
The \textbf{random $n$-hypergraph} is the \Fraisse limit of the class of finite $n$-hypergraphs.
(Strictly speaking, the definition given here is for the generic $n$-hypergraphs in the sense of \Fraisse.
However, the theory of finite $n$-hypergraphs satisfies a $0-1$ law and the associated almost sure theory is the theory of the random $n$-hypergraph.)
Let $T$ be the $L$-theory of infinite $n$-hypergraphs.
It is well-known that the $L$-theory of the random $n$-hypergraph is the model companion $T^*$ of $T$, so the random $n$-hypergraph is also generic in the sense of Robinson.

\begin{prop}
\label{prop:random-graph}
There are $T_\cap,T_1,T_2$ such that the theory of the random hypergraph is bi-interpretable $T^*_\cup$, $T_1$ and $T_2$ are both interpretable in the theory of equality, $\acl_2$ agrees with $\acl_\cap$, and forking in $T_2$ agrees with forking in $T_\cap$.
\end{prop}

\noindent The proof of Proposition~\ref{prop:random-graph},  in fact, shows that $T$ is existentially bi-interpretable with $T_\cup$. 
As $T_1$ and $T_2$ are interpretable in the theory of equality existence of $T^*$ follows from Corollary~\ref{cor:minh} and Fact~\ref{fact:minh}.

\medskip \noindent
By combining Theorem~\ref{thm:simplepreservation} and Proposition~\ref{prop:random-graph} we recover the well known fact that the random $n$-hypergraph is simple for any $n$.
%This is sharp as the random $n$-hypergraph is unstable. 
This shows that the conclusion of Theorem~\ref{thm:simplepreservation} is sharp as the random $n$-hypergraph is unstable.
By combining Corollary~\ref{cor:ccpres} and Proposition~\ref{prop:random-graph} we recover $\aleph_0$-categoricity the random $n$-hypergraph.

\begin{proof}
Given a set $V$ let $\Delta_V$ be the set of $(v_1,\ldots,v_n) \in V^n$ such that $v_i \neq v_j$ for all $i \neq j$ and $\sim_V$ be the equivalence relation on $\Delta_V$ given by $(v_1,\ldots,v_n) \sim_V (v'_1,\ldots,v'_n)$ if and only if $\{v_1,\ldots,v_n\} = \{v'_1,\ldots,v'_n\}$.
%\textcolor{red}{I moved this paragraph inside the proof as I see no reason it should be outside.}
%\textcolor{blue}{i had it outside because i used the same notation in the next result, and possibly somewhere else.
%i added a sentence to \ref{prop:hypergraph-dis} to account for this.}

Let $L_\cap$ be the two-sorted empty language.
Let $T_\cap$ be such that $(V,S) \models T_\cap$ when $V,S$ are both infinite.
Let $L_1$ be the expansion of $L_\cap$ by an $(n+1)$-ary relation symbol $\pi$ on $V^n \times S$.
Let $T_1$ be the $L_1$-theory such that $(V,S;\pi) \models T_1$ if $\pi$ is the graph of a surjection $\Delta_V \to S$ such that $\pi(a) = \pi(b)$ if and only if $a \sim b$.
Let $L_2$ be the expansion of $L_\cap$ by a unary predicate $P$ of sort $S$.
Let $T_2$ be such that $(V,S;P) \models T_2$ if $(V,S) \models T_\cap$ and $P$ is infinite and co-infinite.
It is easy to see that $T_1$ and $T_2$ are interpretable in the theory of equality, $\acl_2$ agrees with $\acl_\cap$, and forking in $T_2$ agrees with forking in $T_\cap$.

We observe $T$ is existentially bi-interpretable with $T_\cup$ and then apply Fact~\ref{fact:existential}.
Suppose that $(V,S;\pi,P) \models T_\cup$.
Let $E \subseteq V^m$ be the preimage of $P$ under $\pi$.
Then $(V;E)$ is an infinite $n$-hypergraph.
Suppose that $(V;E)$ is an infinite $n$-hypergraph.
Let $\pi_V$ be the quotient map $\Delta_V \to \Delta_V/\sim_V$ and $P_V$ be the image of $E$ under $\pi_V$.
Then $(V,\Delta_V/\sim_V;\pi_V,P_V)$ is a $T_\cup$-model.
These observations may be formalized to construct an existential bi-interpretation between $T$ and $T_\cup$.
\end{proof}

\noindent
We now describe a second realization of the theory of random $n$-hypergraph as an interpolative fusion.
Our second realization shows that the conditions for preservation of stability and NIP provided in Proposition~\ref{prop:stabnip} are sharp.
Both $T_1$ and $T_2$ are bi-interpretable with the theory of an infinite set, $\acl_1$ and $\acl_2$ agree both agree with $\acl_\cap$.
However, $\acl_\cap$ does not agree with $\dcl_\cap$.

\begin{prop}
\label{prop:hypergraph-dis}
There are $T_\cap$, $T_1$, and $T_2$ such that $T_1$ and $T_2$ are both bi-interpretable with the theory of equality, $\acl_1$ and $\acl_2$ both agree with $\acl_\cap$, and the theory of the random $n$-hypergraph is bi-interpretable with $T^*_\cup$.
\end{prop}

\begin{proof}
We let $\Delta_V$ and $\sim_V$ be as in the proof of Proposition~\ref{prop:random-graph}.
Let $L_\cap$ be the two sorted language containing a single ternay relation $D$ and let $T_\cap$ be the $L_\cap$-theory such that $(V,S;D) \models T_\cap$ if $D$ is a subset of $\Delta_V \times D$ satisfying
\begin{enumerate}
\item if $a,a' \in \Delta_V, s \in S$ and $a \sim_V a'$ then $D(a,s)$ implies $D(a',s)$,
\item for every $a \in \Delta_V$ there are exactly two $s \in S$ such that $D(a,s)$.
\end{enumerate}
For each $i \in \{1,2\}$ let $L_i$ be the expansion of $L_\cap$ by a binary relation $g_i$ and $T_i$ be the $L_i$-theory such that $(V,S;D,g_i) \models T_i$ if $(V,S;D) \models T_\cap$ and $g_i$ is the graph of a function $g_i : \Delta_V \to S$ such that $D(a,g_i(a))$ for all $\Delta_V$ and if $a,a' \in \Delta_V$ and $a \sim_V a'$ then $g_1(a) = g_2(a')$.

Let $V$ be a set.
We let $Q$ be the quotient of $\Delta_V$ by $\sim_V$ and identify the $\sim_V$-class of $(v_1,\ldots,v_n)$ with $\{v_1,\ldots,v_n\}$.
Let $c_1,c_2 \in V$ be distinct, $S = (\{c_1\} \times Q) \cup (\{c_2\} \times Q)$, $\pi$ be the projection $S \to Q$, and $D$ be the set of $(v_1,\ldots,v_n,a)$ in $\Delta_V \times S$ such that $\pi(a) = \{v_1,\ldots,v_n\}$.
Let $g_1: \Delta_V \to S$ be given by declaring $g_1(v_1,\ldots,v_n) = (c_1,\{v_1,\ldots,v_n\})$.
Then $(V,S;D,g_1) \models T_1$.
It follows that $T_1$ and $T_2$ are both bi-interpretable with the theory of equality and that $\acl_1$ and $\acl_2$ both agree with $\acl_\cap$.

By Fact~\ref{fact:existential} it suffices to show that $T_\cup$ is existentially bi-interpretable with $T$.
Suppose $(V;E)$ is an $n$-hypergraph.
Fix distinct $c_1,c_2 \in V$ and let $S$ and $D$ be defined as in the proceeding paragraph.
Given $(v_1,\ldots,v_n) \in \Delta_V$ we declare
\begin{align*}
g_1(v_1,\ldots,v_n) &= (c_1,\{v_1,\ldots,v_n\}),\\
g_2(v_1,\ldots,v_n) &= (c_1,\{v_1,\ldots,v_n\}) \quad \text{if} \quad E(v_1,\ldots,v_n), \text{  and}\\
g_2(v_1,\ldots,v_n) &= (c_2,\{v_1,\ldots,v_n\}) \quad \text{if} \quad \neg E(v_1,\ldots,v_n).
\end{align*}
Suppose $(V,S;D,g_1,g_2)$ is a model of $T_\cup$.
Given $v_1,\ldots,v_n \in V$ we declare $E(v_1,\ldots,v_n)$ if and only if $(v_1,\ldots,v_n) \in \Delta_V$ and $g_1(v_1,\ldots,v_n) = g_2(v_1,\ldots,v_n)$.
Then $(V;E)$ is an $n$-hypergraph.
These observations are easily formalized to obtain an existential bi-interpretation between $T_\cup$ and $T$.
\end{proof}

%\subsubsection{Random relations}
%\label{section:generic-relation}
\noindent The asymmetric version of the argument in Proposition~\ref{prop:random-graph} gives a construction of the theory of a generic $n$-ary relation.
Fix $n \geq 2$ and suppose that $L$ contains a single $n$-ary relation symbol $R$.
The collection of finite $L$-structures form a \Fraisse class, the \Fraisse limit of this class is the \textbf{generic $n$-ary relation}.
The theory $T^*_R$ of the generic $n$-ary relation is the model companion of the empty $L$-theory.
When $n = 2$, $T^*_R$ is the theory of the generic directed graph.
Simplicity of the generic $n$-ary relation follows from Theorem~\ref{thm:simplepreservation} and Proposition~\ref{prop:generic-relation}.
We recover $\aleph_0$-categoricity of the generic $n$-ary relation by combining Corollary~\ref{cor:ccpres} and Proposition~\ref{prop:generic-relation}.

\begin{prop}
\label{prop:generic-relation}
There are $T_\cap$, $T_1$, and $T_2$ such that $T^*_R$ is bi-interpretable $T^*_\cup$, $T_1$ and $T_2$ are both interpretable in the theory of equality, $\acl_2$ agrees with $\acl_\cap$, and forking in $T_2$ agrees with forking in $T_\cap$.
\end{prop}

\begin{proof}
Let $L_\cap$ be the two-sorted empty language.
Let $T_\cap$ be such that $(V,S) \models T_\cap$ when $V,S$ are both infinite.
Let $L_1$ be the expansion of $T_\cap$ by an $n$-ary function $\pi : V^n \to S$.
Let $T_1$ be such that $(V,S;\pi) \models T_1$ when $\pi$ is a bijection $V^n \to S$.
Let $L_2$ be the expansion of $L_\cap$ by a unary relation of sort $S$.
Let $T_2$ be such that $(V,S;P) \models T_2$ if $P$ is infinite and co-infinite.
Observe that $T_1$ and $T_2$ are both interpretable in the theory of equality, $\acl_2$ agrees with $\acl_\cap$, and forking in $T_2$ agrees with forking in $T_\cap$.

By Fact~\ref{fact:existential} it suffices to show that $T_\emptyset$ is existentially bi-interpretable wuth $T_\cup$.
Suppose that $(V,S;\pi,P) \models T_\cup$.
Let $E$ be the pre-image of $P$ under $\pi$.
Then $(V;E)$ is an $n$-ary relation.
Suppose $(V;E)$ is an $n$-ary relation.
Let $\iota$ be the identity map $V^n \to V^n$.
Then $(V,V^n;\iota,E) \models T_\cup$.
These observations may be formalized to construct an existential bi-interpretation between $T$ and $T_\cup$.
\end{proof}

\noindent
We can also realize the theory of a generic relation as a relatively disintegrated fusion.
Proposition~\ref{prop:relation-dis} follows by a straightforward asymmetric version of the proof of Proposition~\ref{prop:hypergraph-dis}.
We leave the details to the reader.

\begin{prop}
\label{prop:relation-dis}
There are $T_\cap,T_1,T_2$ such that $T^*_R$ is bi-interpretable with $T^*_\cup$, $T_1$ and $T_2$ are both bi-interpretable with the theory of equality, and $\acl_1$ and $\acl_2$ both agree with $\acl_\cap$.
\end{prop}

\subsection{Generic automorphisms}
\label{section:auto}
Arguably the most important unstable simple theory is the theory $\ACFA$ of existentially closed difference fields.
It turns out that $\ACFA$ is bi-interpretable with a fusion of two theories, each of which is bi-interpretable with $\mathrm{ACF}$.
This is a special case of a more general fact about theories of generic automorphisms.
We work with the theory of a structure equipped with a generic family of (non-commuting) automorphisms.
It is no more difficult to handle this via our approach than the expansion by a single generic automorphism.

\medskip \noindent
Let $J$ be an index set.
Suppose that $T$ is a model complete $L$-theory with infinite models.
Let $L_{J-\text{Aut}}$ be the extension of $L$ by a family $(\sigma_j)_{j \in J}$ of unary function symbols and $T_{J-\text{Aut}}$ be an $L_{J-\text{Aut}}$-theory such that $(\sM;(\sigma_j)_{j \in J}) \models T_{J-\text{Aut}}$ if and only if $\sM \models T$ and each $\sigma_j$ is an $L$-automorphism of $\sM$.
We let $T^*_{J-\text{Aut}}$ be the model companion of $T_{J-\text{Aut}}$ if it exists.
We drop the $J$ when $|J| = 1$.
%It is shown in \cite{Kikyo-Shelah} that $T^*_{\text{Aut}}$ does not exist when $T$ is $\SOP$.
%Baldwin and Shelah~\cite{Bald-Shelah} give necessary and sufficient conditions for existence of $T^*_{\text{Aut}}$ for stable $T$.
It is conjectured that $T^*_{\text{Aut}}$ does not exist when $T$ is unstable.
The case of Fact~\ref{fact:aut-simple} when $|J| = 1$ is due to Chatzidakas and Pillay.

\begin{fact}
\label{fact:aut-simple}
If $T$ is stable and $T^*_{J-\mathrm{Aut}}$ exists then $T^*_{J-\mathrm{Aut}}$ is simple.
\end{fact}

\noindent
This is sharp in the sense that  $T^*_{J-\mathrm{Aut}}$ is almost never stable.
Suppose that $|J| = 1$, $L$ is the empty language, and $T$ is the theory of an infinite set.
Suppose $(M;\sigma) \models T^*_{\text{Aut}}$.
It is easy to see that $\phi(x,y) := [\sigma(x) = y]$ is unstable. Fact~\ref{fact:aut-simple} follows from Theorem~\ref{thm:rel-dis-simple} and Proposition~\ref{prop:aut}. 

\begin{prop}
\label{prop:aut}
Let $I = J \cup \{0\}$.
Suppose that $T^*_{J-\mathrm{Aut}}$ exists.
Then there are $T_\cap$ and $(T_i)_{i \in I}$ such that $T^*_{J-\mathrm{Aut}}$ is bi-interpretable with $T^*_\cup$, $T_\cap$ is mutually interpretable with $T$, each $T_i$ is bi-interpretable with $T$ and  disintegrated relative to $T_\cap$.
\end{prop}

\noindent
As usual, we only sketch the proof.

\begin{proof}
We let $T_\cap$ be the theory of two disjoint $T$-models with no additional structure, i.e. $(\sM,\sN) \models T_\cap$ when $\sM,\sN \models T$.
So $T_\cap$ is mutually existentially interpretable, but not bi-intepretable, with $T$.
For each $i \in I$ let $T_i$ be the theory of $(\sM,\sN;\tau_i)$ where $(\sM,\sN) \models T_\cap$ and $\tau_i$ is an $L$-isomorphism $\sM \to \sN$.
If $(\sM,\sN; (\sigma_i)_{i \in I})$ is a $T_\cup$-model then $(\sM; (\tau^{-1}_j \circ \tau_0)_{j \in J} ) \models T_{J-\text{Aut}}$ and if $(\sM; (\sigma_i)_{i \in J} )$ is a $T_{J-\mathrm{Aut}}$-model then $(\sM,\sM;\mathrm{id},(\sigma_i)_{i \in J} ) \models T_\cup$.
It easily follows that $T_\cup$ is existentially bi-interpretable with $T_{J-\text{Aut}}$.
The first claim now follows by Fact~\ref{fact:existential}.

Fix $i \in I$.
It remains to show that $T_i$ is disintegrated relative to $T_\cap$.
Suppose that $(\sM,\sN; \sigma_i) \models T_i$ and $A \subseteq M, B \subseteq N$.
Observe that $\acl_\cap(A,B) = (\acl_T(A), \acl_T(B))$ and that $\acl_i(A,B) = (\acl_T(A\tau^{-1}_i(B)), \acl_T(\tau_i(A)B))$.
So for all $A,A' \subseteq M$ and $B,B' \subseteq N$ we have $ \acl_i(AA',BB') = \acl_\cap( \acl_i(A,B) \acl_i(A',B').$
\end{proof}

% \noindent
% It is easy to see that if $T$ weakly eliminates imaginaries then $T_1$ and $T_2$ weakly eliminate imaginaries.
% Note that if $T_\cup$ weakly eliminates imaginaries, and $T_\square$ eliminates imaginaries for $i \in I \cup\{\cap\}$, then $T_\cup$ eliminates imaginaries because it codes finite sets.
% Corollary~\ref{cor:aut-wei} follows from Theorem~\ref{thm:ei},  Proposition~\ref{prop:aut}, and Fact~\ref{fact:wei-interpret}
% %Corollary~\ref{cor:aut-wei} follows.
% %\textcolor{blue}{add ref to interpration lemmma}
% \textcolor{blue}{this depends on the unwritten wei result}

% \begin{cor}
% \label{cor:aut-wei}
% Suppose that $T$ is stable, $T$ has $3$-uniqueness, and $T^*_{J-\mathrm{Aut}}$ exists.
% Then $T^*_{J-\mathrm{Aut}}$ eliminates imaginaries down to $T$-imaginaries.
% \end{cor}

% \noindent Hrushovksi~\cite{HGroupoids} shows that if $T$ does not have $3$-uniqueness then $T^*_{\mathrm{Aut}}$ does not have weak elimination of imaginaries, see \cite[Remark 3.7]{Bays2017} for an explicit non-eliminable imaginary in $\mathrm{CCM}^*_{\mathrm{Aut}}$. So Corollary~\ref{cor:aut-wei}, and Theorem~\ref{thm:ei}, is sharp. 

% \medskip \noindent
% There is a natural example of a theory $T$ such that $T$ does have have $3$-uniqueness and $T^*_{\text{Aut}}$ exists.
% Let $\mathrm{CCM}$ be the usual theory of compact complex manifolds.
% Bays, Gavrilovich, and Hils~\cite{Bays2013} shows that $\mathrm{CCM}^*_{\text{Aut}}$ exists and Bays, Hils, and Moosa~\cite{Bays2017} show that $\mathrm{CCM}$ does not have $3$-uniqueness.

\subsection{Generic selections}
\label{sec: genericselection}
We describe the generic selection of a definable equivalence relation, and then discuss several specific cases: generic tournaments, functions, and Skolemizations.
Let $E$ be an equivalence relation on a set $X$.
A \textbf{selection} of $E$ is a set that contains exactly one element from each $E$-class and a \textbf{quotient function} for $E$ is a (not necessarily surjective) function $f : X \to Y$ such that for all $a,a' \in X$ we have $f(a) = f(a')$ if and only if $E(a,a')$.

\medskip \noindent
In this section $T$ is a one-sorted, complete, and model complete $L$-theory, and $\phi(x)$ and $\psi(x,y)$ are $L$-formulas such that $\psi(x,y)$ defines an equivalence relation $E_{\sM,\psi}$ on $\phi(M^k)$ for every $\sM \models T$, here $|x| = k$.
(Everything easily generalizes to the case when $T$ is multi-sorted and not model complete.)
Let $L_{\text{Sel}}$ be the expansion of $L$ by an $n$-ary predicate $P$ and $T_{\text{Sel}}$ be the $L_{\text{Sel}}$-theory such that $(\sM;P) \models T_{\text{Sel}}$ if and only if $\sM \models T$ and $P$ is a selection of $E_{\sM,\psi}$.
We say that $T$ eliminates $\psi(x,y)$ if there is a formula $\varphi(x,z), |z| = l$ such that for any $\sM \models T$, $\varphi(M^k, M^l)$ is the graph of a quotient function $\phi(M^k) \to M^l$ for $E_{\sM,\psi}$.

\begin{thm}
\label{thm:generic-selection}
Suppose one of the following holds:
\begin{enumerate}
\item $T^\eq$ eliminates $\exists^\infty$, or
\item $T$ eliminates $\exists^\infty$ and $T$ eliminates $\psi(x,y)$.
\end{enumerate}
Then $T_{\mathrm{Sel}}$ has a  model companion $T^*_{\mathrm{Sel}}$ and there are $T_\cap$, $T_1$, $T_2$ such that 
\begin{enumerate}
\item  $T^*_{\mathrm{Sel}}$ is bi-interpretable with $T^*_\cup$,
\item $T_1$ is bi-interpretable with $T$,
\item $T_2$ is interpretable in the theory of equality, and
\item if each $E_{\sM,\psi}$-class is finite then $\acl_2$ agrees with $\acl_\cap$ and forking in $T_2$ agrees with forking in $T_\cap$.
\end{enumerate}  
\end{thm}

\begin{proof}
For each $n \geq 1$  let $e_n$ be the number of $E_{\sM,\psi}$-classes with exactly $n$-elements if there are finitely many such $E_{\sM,\psi}$-classes and declare $e_n = \infty$ otherwise.
Let $e_\infty$ be the number of infinite $E_{\sM,\psi}$-classes if there are finitely many infinite $\psi$-classes and declare $e_\infty = \infty$ otherwise.

Let $L_\cap$ be the three sorted language with  a  binary relation $E$ on the second sort and a binary relation $\rho$.
Let $T_\cap$ be the $L_\cap$-theory such that $(M,N,Q;E,\rho) \models T_\cap$ if $M,N$ are infinite, $E$ is an equivalence relation on $N$ such that for all $n \ge 1$
\begin{enumerate}
\item if $e_n < \infty$ then there are $e_n$ $E$-classes with exactly $n$ elements,
\item if $e_n = \infty$ then there are infinitely many $E$-classes with $n$ elements,
\item if $e_\infty < \infty$ then there are $e_\infty$ infinite $E$-classes, and
\item if $e_\infty = \infty$ then there are infinitely many infinite $E$-classes,
\end{enumerate}
and $\rho$ is the graph of a quotient function $N  \to Q$ for $E$.
Note that in either case $(1)$ or $(2)$ we have $e_n = 0$ for sufficiently large $n$.
It easily follows that $T_\cap$ is interpretable in the theory of equality.
The quotient function $\rho$ is included to ensure weak elimination of imaginaries.

Let $L_1$ be the expansion of $L_\cap$ by $L$, where all relations and functions are on the first sort, together with an $(k + 1)$-ary relation symbol $\pi$.
Let $T_1$ be the $L_1$-theory such that $(\sM,N,Q;E,\rho,\pi) \models T_1$ when $(M,N;E,\rho) \models T_\cap$, $\sM \models T$, and $\pi$ is the graph of an isomorphism $(\phi(M^k); E_{\sM,\psi}) \to (N;E)$.
It is easy to see that $T_1$ is bi-interpretable with $T$.

Let $L_2$ be the expansion of $L_\cap$ by a unary predicate $P$ of the second sort.
Let $T_2$ be the $L_2$-theory such that $(M,N,Q;E,\rho,P) \models T_2$ if $(M,N,Q;E,\rho) \models T_\cap$ and $P$ contains exactly one element from each $E$-class.
Note that $T_2$ is interpretable in the theory of equality.
It is easy to see that if every $E_{\sM,\psi}$-class is finite then forking and algebraic closure in $T_2$ agree with forking and algebraic closure in $T_\cap$.

Suppose $(\sM,N,Q;E,\rho,\pi,P) \models T_\cup$.
Let $S$ be the preimage of $P$ under $\pi$.
Then $S$ is a selection for $E_{\sM,\psi}$, hence $(\sM;S) \models T_{\text{Sel}}$.
Suppose that $(\sM;S) \models T_{\text{Sel}}$.
Let $Q$ be $\phi(M^{k})/E_{\sM,\psi}$, $\rho$ be the quotient $\phi(M^{k}) \to Q$, $N$ be a copy of $\phi(M^k)$, $\pi$ be a bijection $\phi(M^k) \to N$, $E$ be the pushforward of $E_{\sM,\psi}$ by $\pi$, and $P$ be the image of $S$ under $\pi$.
Then $(\sM,N,Q;E,\rho,\pi,P) \models T_\cup$.
These observations may be formalized to construct an existential bi-interpretation between $T_{\text{Sel}}$ and $T_\cup$.

As $T_\cap$ is interpretable in the theory of equality we see that $T_\cap$ is $\aleph_0$-stable and $\aleph_0$-categorical.
It is also easy to see that $T_\cap$ weakly eliminates imaginaries.
Note that in either case $(1)$ or $(2)$ $T_1$ eliminates $\exists^\infty$.
Furthemore $T^\eq_2$ eliminates $\exists^\infty$ as $T_2$ is interpretable in the theory of equality.
So $T^*_\cup$ exists by Fact~\ref{fact:minh}.
Finally, an application of Fact~\ref{fact:existential} shows that $T^*_{\text{Sel}}$ exists and is bi-interpretable with $T^*_\cup$.
\end{proof}

\noindent
The reader may now apply the results of \cite{firstpaper} to obtain an explicit $\forall\exists$-axiomization of $T^*_{\text{Sel}}$.
Corollary~\ref{cor:selection} follows by Corollary~\ref{cor:triv-forking} as $T_\cap$ is interpetable in the theory of equality.

\begin{cor}
\label{cor:selection}
If $T$ is $\NSOP_1$ and $T^\eq$ eliminates $\exists^\infty$, then $T^*_{\mathrm{Sel}}$ is $\NSOP_1$.
\end{cor}

\noindent
We now discuss some specific cases.

\subsubsection{The generic tournament}
\label{section:tournament}
A tournament is a set $V$ together with a binary relation $E$ such that $(a,a) \notin E$ for all $a \in V$ and for all distinct $a,b \in V$ we either have $(a,b) \in E$ or $(b,a) \in E$, but not both.
Finite tournaments form a \Fraisse class, we refer to the \Fraisse limit of this class is the \textbf{generic tournament}.
Let $T_{\text{Tour}}$ be the theory of tournaments.
The theory of the generic tournament is the model companion $T^*_{\text{Tour}}$ of $T_{\text{Tour}}$.

\medskip\noindent
A tournament is a selection.
Fix a set $V$, let $\Delta$ be the set of $(a,b) \in V^2$ such that $a \neq b$ and $\sim$ be the equivalence relation on $\Delta$ where $(u,v) \sim (u',v')$ if and only if $\{u,v\} = \{u',v'\}$.
Note that a tournament $E$ on $V$ is a selection of $\sim$.
So Proposition~\ref{prop:tournament} is a special case of Theorem~\ref{thm:generic-selection}.

\begin{prop}
\label{prop:tournament}
There are $T_\cap,T_1,T_2$ such that the theory of the generic tournament is bi-interpretable $T^*_\cup$, $T_1$ and $T_2$ are both interpretable in the theory of equality, $\acl_2$ agrees with $\acl_\cap$, and forking in $T_2$ agrees with forking in $T_\cap$.
\end{prop}

\noindent
Combining Theorem~\ref{thm:simplepreservation} and Proposition~\ref{prop:tournament} we recover simplicity of the generic tournament.
This is sharp as the generic tournament is unstable.
We recover $\aleph_0$-categoricity of the generic tournament from Corollary~\ref{cor:ccpres} and Proposition~\ref{prop:tournament}.

\subsubsection{The generic $n$-ary function}
\label{section:generic-function}
Fix $n \geq 1$, let $L_f$ be the language containing a single $n$-ary function symbol $f$, and $T_f$ be the empty $L_f$-theory.
The model companion $T^*_f$ of $T_f$ is the theory of a generic $n$-ary function.

\medskip\noindent
An $n$-ary function is a selection.
Let $M$ be a set.
Let $\sim$ be the equivalence relation on $M^n \times M$ where $(a,b) \sim (a',b')$ if and only if $a = a'$.
A selection for $\sim$ is the graph of a function $M^n \to M$.
So Proposition~\ref{prop:generic-function} is a special case of Theorem~\ref{thm:generic-selection}.

\begin{prop}
\label{prop:generic-function}
Then there are $T_\cap$, $T_1$, $T_2$ such that $T^*_f$ is bi-interpretable with $T^*_\cup$ and $T_1$ and $T_2$ are interpretable in the theory of equality.
\end{prop}

\noindent
Apply Theorem~\ref{thm:preservation} we recover the fact that the theory of a generic $n$-ary function is $\NSOP_1$.
This is sharp as the theory of a generic $n$-ary function is $\mathrm{TP}_2$ \cite[Proposition 3.14]{KRExp}.
We describe another presentation of the generic $n$-ary function as a fusion of stable theories in Section~\ref{sec:empty-companion}.

%\medskip\noindent
%We describe a simpler presentation of the theory of the generic $n$-ary function as a fusion.
%This will be used in Section~\ref{sec:empty-companion}.
%Let $L_\cap$ be the empty two-sorted language and $T_\cap$ be the $L_\cap$-theory such that $(M,N) \models T_\cap$ when $M,N$ are both infinite.
%Let $L_1$ be the expansion of $L_\cap$ by a $n$-ary function symbol $\pi$ and $T_1$ be the $L_1$-theory such that $(M,N;\pi) \models T_1$ if $M,N$ are both infinite and $\pi$ is a bijection $M^n \to N$.
%Let $L_2$ be the expansion of $L_\cap$ by an unary function symbol $g$ and $T_2$ be the $L_2$-theory such that $(M,N;g) \models T_2$ if $M,N$ are both infinite and $g$ is a function $N \to M$ with infinite and co-infinite range and $g^{-1}(a)$ is either empty or infinite for all $a \in M$.

\subsubsection{Generic Skolem functions}
\label{section:generic-skolem}
%This generalizes the case of a single $n$-ary function.
Suppose that $T$ is a complete, model complete, $L$-theory with infinite models.
Let $\theta(x,y)$ be an $L$-formula and $\Gamma(x) = \exists y \theta(x,y)$.
Given $\sM \models T$ a Skolem function for $\theta(x,y)$ is a function $f : \Gamma(M^{|x|}) \to M^{|y|}$ such that $\sM \models \theta(a,f(a))$ for all $a \in M^{|x|}$.
Let $L_{\theta-\text{Skol}}$ be the expansion of $L$ by an $(|x| + |y|)$-ary predicate $f$ and $T_{\theta-\text{Skol}}$ be the $L_{\theta-\text{Skol}}$-theory such that $(\sM;f) \models T_{\theta-\text{Skol}}$ when $f$ is the graph of a Skolem function for $\theta(x,y)$.
The first claim of Fact~\ref{fact:winkler-skolem} is due to Winkler~\cite{Winkler}.
The second claim is due to Kruckman and Ramsey~\cite{KRExp}.

\begin{fact}
\label{fact:winkler-skolem}
Suppose $T$ eliminates $\exists^\infty$.
Then $T_{\theta-\mathrm{Skol}}$ has a model companion $T^*_{\theta-\mathrm{Skol}}$.
If $T$ is $\NSOP_1$ then $T^*_{\theta-\mathrm{Skol}}$ is $\NSOP_1$.
\end{fact}

\noindent
Let $|x'| = |x|$, $|y'| = |y|$ and $\psi(x,y,x',y')$ be $\theta(x,y) \land \theta(x',y') \land (x = x')$.
So if $\sM \models T$ then $\psi$ defines an equivalence relation on $\theta(M^{|x|}, M^{|y|}$) and any selection for this equivalence relation is the graph of a Skolem function for $\theta(x,y)$.
So Fact~\ref{fact:winkler-skolem} follows from Theorem~\ref{thm:generic-selection}. 

\medskip\noindent
Fact~\ref{fact:nub-simple} is due to N\"ubling~\cite{Nbling2004}.

\begin{fact}
\label{fact:nub-simple}
Suppose that $\theta(x,y)$ is bounded in $y$.
If $T$ is simple then $T^*_{\theta-\mathrm{Skol}}$ is simple.
If $T$ is $\aleph_0$-categorical then any completion of $T^*_{\theta-\mathrm{Skol}}$ is $\aleph_0$-categorical.
\end{fact}

\noindent
The first claim of Fact~\ref{fact:nub-simple} follows from Theorem~\ref{thm:simplepreservation} and Theorem~\ref{thm:generic-selection}.
The second claim folows from Theorem~\ref{thm:ccpres} and Theorem~\ref{thm:generic-selection}.
The second claim of Fact~\ref{fact:nub-simple} is reasonably sharp, see \cite[2.3.1]{Nbling2004}.

\medskip\noindent
One can prove Fact~\ref{fact:winkler-skolem-strong} by iterating Fact~\ref{fact:winkler-skolem}.
The first claim is also due to Winkler~\cite{Winkler} and the second claim is due to Kruckman and Ramsey~\cite{KRExp}.

\begin{fact}
\label{fact:winkler-skolem-strong}
Suppose that $T$ eliminates $\exists^\infty$.
Then there is a language $L \subseteq L_{\mathrm{Skol}}$ and a model complete $L_{\mathrm{Skol}}$-theorem $T \subseteq T^*_{\mathrm{Skol}}$ which has definable Skolem functions.
If $T$ is $\NSOP_1$ then $T^*_{\mathrm{Skol}}$ is $\NSOP_1$.
\end{fact}

\noindent
One constructs $T^*_{\text{Skol}}$ as the union of a countable chain of theories, each of which defines Skolem functions for the previous theory and then applies the fact that a union of a chain of $\NSOP_1$ theories is $\NSOP_1$.

\subsection{The model companion of the empty theory}
\label{sec:empty-companion}
Let $L$ be a first order language and $\varnothing_L$ be the empty $L$-theory.
Winkler~\cite{Winkler} showed that $\varnothing_L$ has a model companion $\varnothing^*_L$.
Fact~\ref{fact:empty-companion} is due to Kruckman and Ramsey~\cite{KRExp} and independently Je{\v{r}}{\'{a}}bek~\cite{Jebek2019}.

\begin{fact}
\label{fact:empty-companion}
$\varnothing^*_L$ is $\NSOP_1$.
\end{fact}

\noindent
Fact~\ref{fact:empty-companion} is sharp as $\varnothing^*_L$ is $\mathrm{TP}_2$ and hence not simple when $L$ contains a function symbol of airity at least two~\cite[Proposition 3.14]{KRExp}.
Proposition~\ref{prop:empty-companion} allows us to realize $\varnothing^*_L$ as a fusion of stable theories. In particular, this allows us to recover Fact~\ref{sec:empty-companion} using Theorem~\ref{thm:preservation} and Proposition~\ref{prop:stablethreebody}.

\begin{prop}
\label{prop:empty-companion}
There is $T_\cap$, and $(T_i)_{i \in I}$ such that each $T_i$ is interpretable in the theory of equality and $\varnothing^*_L$ is bi-interpretable with $T^*_\cup$. 
\end{prop}

\noindent
So existence of $\varnothing^*_L$ follows by Corollary~\ref{cor:minh}.
Fact~\ref{fact:empty-companion} follows from Proposition~\ref{prop:empty-companion} and Corollary~\ref{cor:triv-forking}.

\begin{proof}
Let $\mathcal{C}$ be the set of constant symbols in $L$, $\mathcal{R}$ be the set of relation symbols, and $\mathcal{F}$ be the set of function symbols.
Let $T$ be the $L$-theory such that $\sM \models T$ if 
\begin{enumerate}
\item every $n$-ary $R \in \mathcal{R}$ defines an infinite and co-infinite subset of $M^n$, and
\item whenever $f \in \mathcal{F}$ is $n$-ary then $f : M^n \to M$ is surjective and $f^{-1}(a)$ is infinite for all $a \in M$.
\end{enumerate}
It is easy to see that $\varnothing^*_L$ is a model companion of $T$.
We construct $(T_i)_{i \in I}$, show that $T_\cup$ is existentially bi-interpretable with $T$, and apply Fact~\ref{fact:existential}.

Let $I = L \cup \{1\}$.
Let $L_\cap$ be the countably-sorted empty language and $T_\cap$ be the $L_\cap$-theory such that $(M,(N_k)_{k \geq 1}) \models T_\cap$ if $M$ and each $N_k$ is infinite.
%We refer to $M$ as the ``home sort".

For each $c \in \mathcal{C}$ let $L_c$ be the expansion of $L_\cap$ by a constant symbol $c$ and $T_c$ be the $L_c$-theory such that $(M,(N_k)_{k \geq 1};c) \models T_c$ when $c$ defines an element of $M$.
So each $T_c$ is interpretable in the theory of equality.

Let $L_1$ be the expansion of $L_\cap$ by function symbols $(\pi_k)_{k \geq 1}$ such that each $\pi_k$ is $k$-ary.
Let $T_1$ be the $L_1$-theory such that $(M,(N_k)_{k \geq 1};(\pi_k)_{k \geq 1}) \models T_1$ if each $\pi_k$ is a bijection $M^k \to N_k$.
It is easy to see that $T_1$ is interpretable in the theory of equality.

For each $n$-ary $R \in \mathcal{R}$ let $L_R$ be the expansion of $L_\cap$ by a unary predicate $P_R$ and let $T_R$ be the $L_R$-theory such that $(M,(N_k)_{k \geq 1}; P_R) \models T_R$ if and only if $P_R$ defines an infinite and co-infinite subset of $N_n$.
For each $n$-ary $f \in \mathcal{F}$ let $L_f$ be the expansion of $T_\cap$ by a unary function symbol $g_f$ and $T_f$ be the $L_f$-theory such that $(M,(N_k)_{k \ge 1};g_f) \models T_f$ when $g_f$ is a surjection $N_k \to M$ and $g^{-1}_f(a)$ is infinite for every $a \in M$.
It is easy to see that each $T_R$ and $T_f$ is interpretable in theory of equality.

Suppose that $\sM \models T$.
We describe a $T_\cup$-model which is existentially interpretable in $\sM$.
For each $n \geq 1$ let $N_n$ be $M^n$ and $\pi_n$ be the identity $M^n \to N_n$.
Interpret constant symbols in the obvious way.
For each $n$-ary $R \in \mathcal{R}$ let $P_R$ be $\{ \pi_n(a) : \sM \models R(a) \}$.
For each $n$-ary $f \in \mathcal{F}$ let $g_f : N_n \to M$ be $f \circ \pi^{-1}_n$.
Then $(M,(N_k)_{k \geq 1};\ldots)$ is a model of $T_\cup$.

Now suppose that $\sM_\cup = (M,(N_k)_{k \geq 1}; \ldots) \models T_\cup$.
We describe a $T$-model with domain $M$ which is existentially interpretable in $\sM_\cup$.
Again interpret constant symbols in the obvious way.
Let each $n$-ary $R \in \mathcal{R}$ define $\{ a \in M^n : \sM_\cup \models P_R(\pi_n(a)) \}$.
Let each $n$-ary $f \in \mathcal{F}$ be $g_f \circ \pi_n$.
This is easy seen to produce a $T$-model.
These observations are easily formalized to construct an existential bi-interpretation between $T$ and $T_\cup$.
\end{proof}

% \noindent
% Fact~\ref{fact:wei-empty} is \cite[Proposition 3.22]{KRExp}.
% It is also easy to see that each $T_i$ in the construction above has weak elimination of imaginaries.
% So Fact~\ref{fact:wei-empty} follows by Theorem~\ref{thm:ei} and Proposition~\ref{prop:empty-companion}.

% \begin{fact}
% \label{fact:wei-empty}
% $\emptyset^*_L$ weakly eliminates imaginaries.
% \end{fact}

% \noindent
% \textcolor{blue}{check bi-interpretation result works}

% \noindent
% Fact~\ref{fact:empty-simple} is proven in \textcolor{blue}{????}

\begin{fact}
\label{fact:empty-simple}
Suppose that $L$ does not contain any function symbol of airity at least two.
Then $\varnothing^*_L$ is simple.
\end{fact}

\noindent
We sketch a proof of Fact~\ref{fact:empty-simple}.
We will use some elementary facts about the theory of a generic unary function which we leave to the reader.
Suppose that $L$ does not contain any function symbol of airity at least two.
Let $\mathcal{C}, \mathcal{R}$, and $\mathcal{F}$ be as above.
Let $I = L$.
Let $L_\cap$ be the one-sorted empty language and $T_\cap$ be the theory of an infinite set.
For each $c \in \mathcal{C}$ let $T_c$ be the theory of an infinite set with a distinguished element.
For each $n$-ary $R \in \mathcal{R}$ let $T_R$ be the theory of a generic $n$-ary predicate and for each $f \in \mathcal{F}$ let $T_f$ be the theory of a generic unary function.
So $\varnothing^*_L$ is $T^*_\cup$.
Each $T_c$ and $T_f$ is stable and each $T_R$ is simple.
We observe that each $T_i$ is disintegrated, an application of Theorem~\ref{thm:rel-dis-simple} shows that $\varnothing^*_L$ is simple.
Each $T_c$ and $T_R$ has trivial algebraic closure.
In each $T_f$ the algebraic closure of a set is the substructure generated by that set, hence $T_f$ is unary.

\subsubsection{The generic expansion}
\label{section:generic-expansion}
Suppose that $T$ is a one-sorted model complete $L'$-theory and $L$ is a language containing $L'$.
We consider $T$ to be an $L$-theory.
If $T$ eliminates $\exists^\infty$ then $T$, considered as an $L$-theory, has a model companion $T^*_L$~\cite{Winkler}.
If $L'$ is the empty language and $T$ is the theory of equality then $T^*_{L}$ is $\varnothing^*_L$.
If $L \setminus L'$ contains a single unary predicate then $T^*_L$ is the expansion of $T$ by a generic unary predicate.
Fact~\ref{fact:generic-expansion} is due to Kruckman and Ramsey~\cite{KRExp}.

\begin{fact}
\label{fact:generic-expansion}
Suppose that $T$ eliminates $\exists^\infty$.
If $T$ is $\NSOP_1$ then $T^*_L$ is $\NSOP_1$.
\end{fact}

\noindent
Suppose that $T$ eliminates $\exists^\infty$.
Let $I = \{1,2\}$, $L_\cap$ be the one-sorted empty language, $T_\cap$ be the theory of an infinite set, $L_1$ be $L'$ , $T_1$ be $T$, $L_2$ be $L \setminus L'$, and $T_2$ be $\varnothing^*_{L \setminus L'}$.
Then $T^*_L$ agrees with $T^*_\cup$.
So Fact~\ref{fact:generic-expansion} follows from Corollary~\ref{cor:triv-forking}.

\medskip\noindent
One can also adapt the proof of Proposition~\ref{prop:empty-companion} to show that if $T$ is stable then $T^*_L$ is bi-interpretable with the fusion of a family of stable theories.
%We leave this to the reader.

\subsection{The generic variation} \label{section:generic variation}
Suppose that $L$ is a relational language and $T$ is a complete, one-sorted $L$-theory with infinite models.
%The restriction to a one-sorted relational language is unnecessary and adopted to avoid minor technicalities.
The generic variation $T^*_{\text{Var}}$ of $T$ was defined by Baudisch~\cite{baudisch}.
This theory exists when $T$ eliminates $\exists^\infty$.
This example is essentially a natural generalization of one of the most basic examples of a theory that is $\NSOP_1$ and not simple, see \ref{section:tfeq}.

\medskip \noindent
Let $L_{\text{Var}}$ be the language containing an $(n+1)$-ary relation symbol $R_{\text{Var}}$ for each $n$-ary relation symbol $R$ in $L$.
Suppose $\sN$ is an $L_{\text{Var}}$-structure.
For each $a \in N$ we put an $L$-structure $\sN[a]$ on $N$ by declaring $\sN[a] \models R(b)$ if and only if $\sN \models R_{\text{Var}}(a,b)$ for any relation symbol $R$ from $L$.
It is easy to describe an $L_{\text{Var}}$-theory $T_{\text{Var}}$ such that $\sN \models T_{\text{Var}}$ if and only if $\sN[a] \models T$ for every $a \in N$.
Fact~\ref{fact:baud} is proven in \cite{baudisch}.

\begin{fact}
\label{fact:baud}
If $T$ eliminates $\exists^\infty$ then $T_{\mathrm{Var}}$ has a model companion $T^*_{\mathrm{Var}}$.
\end{fact}

\noindent
If $T$ is the theory of the generic $n$-ary relation then $T^*_{\text{Var}}$ is the theory of the generic $(n + 1)$-ary relation, so one may produce the theory of a generic $n$-ary relation from the (stable) theory of an infinite and co-infinite unary relation by iterating the construction $n - 1$ times.

\begin{prop}
\label{prop:baud-1}
Suppose that $T$ eliminates $\exists^\infty$.
Then there are $T_\cap$, $T_1$, $T_2$ such that $T^*_{\mathrm{Var}}$ is b-interpretable with $T^*_\cup$, $T_1$ is bi-interpretable with the theory of equality, and $T_2$ is mutually interpretable with $T$.
\end{prop}

\begin{proof}
Let $L_\cap$ be the two sorted language with a single unary function symbol $\pi_1$.
Let $T_\cap$ be the $L_\cap$ theory such that $(M,N;\pi_1) \models T_\cap$ if $M,N$ are both infinite and $\pi_1 : N \to M$ is a surjection such that $\pi^{-1}_1(b)$ is infinite for all $b \in N$.
Let $L_1$ be the expansion of $L_\cap$ by a unary function $\pi_2$.
Let $T_1$ be the $L_1$-theory such that $(M,N;\pi_1,\pi_2) \models T_1$ if the map $N \to M^2$ given by $a \mapsto (\pi_1(a), \pi_2(a))$ is a bijection.
It is easy to see that $T_2$ is bi-interpretable with the theory of equality.

Let $L_2$ be the expansion of $L_\cap$ by an $(n+1)$-ary relation $R_2$ for each $n$-ary relation $R \in L$ such that the first variable of $R_2$ is of sort $M$ and the last $n$ variables are of sort $N$.
Suppose $\sM = (M,N;\pi_1, (R_2)_{R \in L_\cap})$ is an $L_2$-structure.
We let $M_a$ be the set of $b \in N$ such that $\pi_1(b) = a$ for all $a \in M$.
For each $a \in M$ we equip $M_a$ with an $L$-structure $\sM(a)$ by letting the interpertation of each $n$-ary $R \in L$ be $\{ b \in (M_a)^n : \mathscr{N} \models R_2(a,b)\}$.
Then $T_2$ is the $L_2$-theory such that $\mathscr{M} \models T_2$ when $\sM(a) \models T$ for all $a \in Y$ and 
$$ T_2 \models \forall x,y_1,\ldots,y_n \left( R_2(x,y_1,\ldots,y_n) \to [ \pi_1(y_1) = x \land \ldots \land \pi_1(y_n) = x] \right)  $$
for every $n$-ary $R \in L$.
Note that $T_2$ interprets $T$.
We show that $T$ interprets $T_2$.
Suppose $\sM \models T$.
Let $\rho : M^2 \to M$ be the projection onto the first coordinate and for each $n$-ary relation symbol $R \in L$ declare $R_2(a,b)$ if $a \in M$, $b \in M^n$, and $\sM \models R(a)$.
Then $(M,M^2;\rho, (R_2)_{R \in L}) \models T_2$.

We show that $T_\cup$ and $T_{\mathrm{Var}}$ are existentially bi-interpretable and apply Fact~\ref{fact:existential}.
Suppose that $ \sN = (M,N;\pi_1,\pi_2, (R_2)_{R \in L})$ is a $T_\cup$-model.
For each $n$-ary $R \in L$ and $a \in M$, $b_1,\ldots,b_n \in N$ we declare $R_{\text{Var}}(a,b_1,\ldots,b_n)$ if and only if  $\sN \models R_2(a,c_1,
\ldots,c_n)$ where each $c_k$ is the unique element of $N$ such that $\pi_1(c_k) = a$ and $\pi_2(c_k) = b_k$.
Then $(M; (R_{\text{Var}})_{R \in L}) \models T_{\text{Var}}$.

Now suppose that $\sM =  (M;(R_{\text{Var}})_{R\in L}) \models T_{\text{Var}}$.
Let $\pi_1,\pi_2$ be the projections $M^2 \to M$ onto the first and second coordinates, respectively.
For each $n$-ary $R \in L$ and $a \in M$, $b_1,\ldots,b_n \in M^2$ we declare $R_2(a,b_1,\ldots,b_n)$ if and only if $\pi_1(b_1) = \ldots = \pi_1(b_n) = a$ and $\sM \models R_{\text{Var}}(a,\pi_2(b_1),\ldots,\pi_2(b_n))$.
These observations may be formalized to construct an existential bi-interpretation between $T_\cup$ and $T_{\text{Var}}$.
\end{proof}

\noindent
We also observe that Fact~\ref{fact:baud} follows from Corollary~\ref{cor:minh}.
Suppose that $T$ eliminates $\exists^\infty$.
It easily follows that $T_1$ eliminates $\exists^\infty$.
Note that $T_1$ eliminates $\exists^\infty$ as $T_1$ is interpretable in the theory of equality.
Finally, $T_\cap$ is interpretable in the theory of equality.

\medskip\noindent
Dobrowolski~\cite{Dobro} showed that $T^*_{\text{Var}}$  is $\mathrm{NTP}_1$ when $T$ is $\mathrm{NTP}_1$.
It is conjectured that $\NSOP_1$ and $\mathrm{NTP}_1$ are equivalent.
We obtain Corollary~\ref{cor:variation} by combining Corollary~\ref{cor:triv-forking} and Proposition~\ref{prop:baud-1}.

\begin{cor}
\label{cor:variation}
If $T$ is $\NSOP_1$ and eliminates $\exists^\infty$ then $T^*_{\mathrm{Var}}$ is $\NSOP_1$.
\end{cor}

\noindent
Corollary~\ref{cor:variation} generalizes the result of Baudisch that if $T$ is stable and eliminates $\exists^\infty$ then $T^*_{\text{Var}}$ does not have the strict order property~\cite[Theorem 4.4]{baudisch}.
Corollary~\ref{cor:variation} is sharp.
Baudisch shows that if there is an $L$-formula $\varphi(x,y)$, $\sM \models T$, and a sequence $(a_i)_{i \in \omega}$ from $\sM$ such that the $\varphi(M^{|x|},a_i)$ are pairwise disjoint and infinite then $T^*_{\text{Var}}$ is $\mathrm{TP}_2$~\cite[Lemma 4.2]{baudisch}.
The same argument shows that if $T$ has $U$-rank at least two then $T^*_{\text{Var}}$ is $\mathrm{TP}_2$.

% \medskip\noindent
% It is easy to see that $T_1$ weakly eliminates imaginaries and that if $T$ weakly eliminates imaginaries then $T_2$ weakly eliminates imaginares.
% Corollary~\ref{cor:var-wei} follows by \textcolor{blue}{checking the bi-interpretation, this also depends on the wei result}.

% \begin{cor}
% \label{cor:var-wei}
% If $T$ is stable and weakly eliminates imaginaries then $T^*_{\mathrm{Var}}$ weakly eliminates imaginaries.
% \end{cor}

\subsubsection{$T^*_{\mathrm{Feq}}$} 
\label{section:tfeq}
Let $L_{\text{Feq}}$ be the two sorted theory containing a single ternay relation $D$.
Let $T_{\text{Feq}}$ be the $L_{\text{Feq}}$-theory such that $(M,N;D) \models T_{\text{Feq}}$ if $M,N$ are infinite, $D(x,y,z)$ is a ternary relation where $x$ is of sort $M$ and $y,z$ are of sort $N$, and $R(a,x,y)$ is an equivalence relation for any $a \in M$.
Then $T_{\text{Feq}}$ has a model companion $T^*_{\text{Feq}}$.
Chernikov and Ramsey~\cite{CR} show that $T^*_{\text{Feq}}$ is $\NSOP_1$.

\medskip\noindent
Proposition~\ref{prop:tfeq} follows by a slight modification of Proposition~\ref{prop:baud-1} and the observation that the theory of an equivalence relation with infinitely many infinite classes is interpretable in the theory of equality.
We leave the details to the reader.

\begin{prop}
\label{prop:tfeq}
There are $T_\cap, T_1, T_2$ such that $T^*_{\mathrm{Feq}}$ is bi-interpretable with $T^*_\cup$ and $T_1$ and $T_2$ are both interpretable in the theory of equality.
\end{prop}

\noindent
Combining Theorem~\ref{thm:preservation} and Proposition~\ref{prop:tfeq} we recover the fact that $T^*_{\text{Feq}}$ is $\NSOP_1$.
This is sharp as $T^*_{\mathrm{Feq}}$ is easily seen to be $\mathrm{TP}_2$.

\subsection{Generic subspaces}
\label{sec:vector}
Fix a finite field $\F$.
Let $T_{\mathrm{Vec}}$ be the theory of $\F$-vector spaces, $L$ be a language extending the language of $\F$-vector spaces, $T$ be a complete, model complete, $L$-theory extending $T_{\mathrm{Vec}}$, and $T_{\mathrm{Sub}}$ be the theory of a $T$-model equipped with an $\F$-vector subspace.
Fact~\ref{fact:elbee} is a special case of work of d'Elb\'ee~\cite{elbee}.

\begin{fact}
\label{fact:elbee}
Suppose that $T$ eliminates $\exists^\infty$.
Then the model companion $T^*_{\mathrm{Sub}}$ of $T_{\mathrm{Sub}}$ exists.
If $T$ is stable then $T^*_{\mathrm{Sub}}$ is $\NSOP_1$.
\end{fact}

\noindent
We do not know if Fact~\ref{fact:elbee} remains true when ``stable" is replaced by ``$\NSOP_1$".
Fix a prime $p$.
A special case of Fact~\ref{fact:elbee} is that the theory of an algebraically closed field of characteristic $p$ equipped with an additive subgroup has an $\NSOP_1$ model companion.
This is sharp as this theory is $\mathrm{TP}_2$, see \cite{elbee}.

\medskip\noindent
The proof of Proposition~\ref{prop:vec} is easy and left to the reader.

\begin{prop}
\label{prop:vec}
Let $I = \{1,2\}$, $T_1 = T$, and $T_2$ be the theory of $(\sV;P)$ where $\sV$ is an infinite dimensional $\F$-vector space and $P$ is a unary predicate defining an infinite dimensional and infinite codimensional subspace of $\sV$.
Then $T^*_\cup$ exists and equals $T^*_{\mathrm{Sub}}$.
\end{prop}

\noindent
Note that that a $T_2$-model $(\sV;P)$ is determined up to isomorphism by the dimensions of $P$ and $\sV$.
It follows that $(\sV;P)$ is $\aleph_0$-categorical and $\aleph_0$-stable, so in particular $T_2$ eliminates $\exists^\infty$.
So the first claim of Fact~\ref{fact:elbee} is a special case of Fact~\ref{fact:minh} and the second claim is a special case of Corollary~\ref{cor:stable-case}.

\medskip\noindent
We describe a related construction of d'Elb\'ee.
Fix an algebraically closed field $K$ and an abelian variety $A$ defined over $K$.
Suppose that the ring of endomorphisms of $A$ is $\zz$.
Let $T$ be the theory of the expansion of $K$ by a subgroup of the $K$-points $A(K)$ of $A$.
Then d'Elb\'ee~\cite{elbee-variety} shows that $T$ has a model companion $T^*$ which is $\NSOP_1$ and $\mathrm{TP}_2$.
One can show that $T^*$ is bi-interpretable with a fusion of two stable theories, hence Corollary~\ref{cor:stable-case} implies that $T^*$ is $\NSOP_1$.
We leave the details of this to the reader.

\subsection{Fusion over the theory of equality}
\label{section:winkler}
Let $L_\cap$ be the language of equality and $T_\cap$ be the theory of an infinite set.
Fact~\ref{fact:wink0} is essentially due to Winkler~\cite{Winkler}. The first part can be recovered as a special case of Corollary~\ref{cor:minh}.

\begin{fact}
\label{fact:wink0}
 If each $T_i$ eliminates $\exists^\infty$ then $T^*_\cup$ exists.
Conversely, if $T_1$ is complete and does not eliminate $\exists^\infty$ then there is $T_2$ such that $T^*_\cup$ does not exist.
\end{fact}

\noindent
Fact~\ref{fact:winkler-nsop} is proven by Ramsey and the first author in~\cite{KRExp} and independently by Je\u{r}\'abek in ~\cite{Jebek2019}. It can be recovered as a special case of Corollary~\ref{cor:triv-forking}.

\begin{fact}
\label{fact:winkler-nsop}
If each $T_i$ is $\NSOP_1$ and $T^*_\cup$ exists, then $T^*_\cup$ is $\NSOP_1$.
\end{fact}

\noindent
Tsuboi conjectured in~\cite{Tsuboi} that if $L_1 \cap L_2 = \emptyset$ and $T_1,T_2$ are both simple and eliminate $\exists^\infty$ then $T_\cup$ has a simple completion. 
The following corollary of Fact~\ref{fact:wink0} and Fact~\ref{fact:winkler-nsop} is the $\NSOP_1$ analogue of Tsuboi's conjecture.

\begin{cor}
\label{cor:tsuboi}
If $L_i \cap L_j = \emptyset$ for distinct $i,j \in I$ and each $T_i$ is $\NSOP_1$ and eliminates $\exists^\infty$ then $T_\cup$ has an $\NSOP_1$ completion.
\end{cor}

\noindent
We now describe some special cases from the literature where more than $\NSOP_1$ is preserved.
Suppose that $I = \{1,2\}$, $T_1$ is an arbitrary theory which eliminates $\exists^\infty$, $L_2 = \{P\}$ for a unary predicate $P$, and $T_2$ is the theory of an infinite and co-infinite unary predicate.
Then $T^*_\cup$ is the expansion of $T_1$ by a generic predicate defined by Chatzidakas and Pillay~\cite{Cha-Pi}.
Fact~\ref{fact:cha-pi} is proven in \cite{Cha-Pi}.
Note that Fact~\ref{fact:cha-pi} is a special case of Theorem~\ref{thm:simplepreservation}.

\begin{fact}
\label{fact:cha-pi}
If $T_1$ is simple and eliminates $\exists^\infty$, then the expansion of $T_1$ by a generic unary predicate is simple.
\end{fact}

\noindent
Proposition~\ref{prop:tp2} describes a family of examples which shows that fusions over the theory of equality are typically $\mathrm{TP}_2$. 

\begin{prop}
\label{prop:tp2}
Suppose $I = \{1,2\}$, $T_1$ expands the theory of some infinite group, $T_2$ has $U$-rank at least two, and $T^*_\cup$ exists.
Then $T^*_\cup$ has $\mathrm{TP}_2$.
\end{prop}

\noindent
So for example the fusion of a vector space and an equivalence relation with infinitely many infinite classes over the theory of equality is $\mathrm{TP}_2$.
This once again demonstrates the sharpness of the $\NSOP_1$ conclusion in Corollary~\ref{cor:triv-forking}.

\begin{proof}
Let  $\sM_\cup \models T^*_\cup$ be $\aleph_0$-saturated.  Morleyizing if necessary, we can arrange that $T_i$ admits quantifier elimination. By Fact~\ref{fact: first Theorem}, $T^*_\cup$ is model complete and $\sM_\cup$ is an existentially closed model of $T^*_\cup$.
As $T_2$ has $U$-rank at least two there is an $L_2$-formula $\varphi(x,y)$ and a sequence $(b_k)_{k \in \omega}$ of tuples such that $\varphi(M^{|x|},b_k)$ is infinite for all $k \in \omega$ and $\{ \varphi(x,b_k) \mid k \in \omega\}$ is $n$-inconsistent for some $n$. 
Let $(a_l)_{l \in \omega}$ be a sequence of distinct elements of $M$.
Using $\cdot$ to denote the group operation in $T_1$, we define $\psi(x,y,z) = \varphi(x \cdot z^{-1},y)$ so that $\psi(M^{|x|},b_k,c_l) = \varphi(M^{|x|},b_k) \cdot c_l$ for all $k,l$.
Note that $\{ \psi(x,b_k,c_l) \mid k \in \omega\}$ is $n$-inconsistent for all $l$.
Fix $m \in \omega$ and $f : m \to \omega$.
An easy application of existential closedness produces $a \in M^{|x|}$ such that we have $\sM_\cup \models \psi(a,b_{f(l)},c_l)$ for all $l \leq m$.
So $\psi(x,y,z)$ together with the array $(b_k,c_l) \mid k,l \in \omega)$ together witness that $T_\cup^*$ has $\mathrm{TP}_2$.
\end{proof}

\noindent
Simon and Shelah~\cite{Shelah-Simon} show that if $T$ is $\NIP$ and algebraic closure in $T$ is trivial then the expansion of $T$ by a generic dense linear order is $\NIP$.
Proposition~\ref{prop:ss} generalizes this.
Proposition~\ref{prop:ss} follows from Fact~\ref{fact:wink0} and Proposition~\ref{prop:stabnip}.
\begin{prop}
\label{prop:ss}
Suppose that each $T_i$ is $\NIP$ and that each $\acl_i$ is trivial.
Then $T^*_\cup$ exists and is $\NIP$.
\end{prop}

\subsection{Fusions of blowups}
\label{section:blow-up}
Example~\ref{ex:eqrel} is a member of a more general family of examples which we now discuss.
These examples witness sharpness of Proposition~\ref{prop:relaclcomp}.
Suppose $L$ is a relational language and $T$ is a complete $L$-theory with infinite models.
We define the \textbf{blow up} of $T$.
Let $L_{\text{Blow}}$ be the language containing a binary relation symbol $E$ and an $n$-ary relation symbol $R_{\text{Blow}}$ for each $n$-ary relation symbol $R \in L$.
Suppose $\sM \models T$.
A blow up of $\sM$ is an $L_{\text{Blow}}$-structure $\sN$ such that there is a surjection $\rho : N \to M$ satisfying the following:
\begin{enumerate}
\item For all $a,b \in N$ we have $\rho(a) = \rho(b)$ if and only if $\sN \models E(a,b)$,
\item $\rho^{-1}(a)$ is infinite for all $a \in N$,
\item for every $n$-ary relation symbol $R \in L$ and $a_1,\ldots,a_n \in N$ we have\\ $\sN \models R_{\text{Blow}}(a_1,\ldots,a_n)$ if and only if $\sM \models R(\rho(a_1),\ldots,\rho(a_n))$.
\end{enumerate}

\noindent
Let $T_{\text{Blow}}$ be the $L_{\text{Blow}}$-theory of blowups of models of $T$.
It is easy to see that $T_{\text{Blow}}$ is complete and algebraic closure in $T_{\text{Blow}}$ is trivial.
Given $i \in \{1,2\}$ let $L_i$ be a relational language, $T_i$ be a complete $L_i$-theory with infinite models, suppose that $T^\eq_i$ eliminates $\exists^\infty$ and $L_1 \cap L_2 = \emptyset$.
So $T^*_\cup$ exists by Fact~\ref{fact:minh}.
Let $S_\cap$ be the $\{E\}$-theory of an equivalence relation with
infinitely many infinite classes and given $i \in \{1,2\}$ let $S_i$ be $(T_i)_{\text{Blow}}$.
It is easy to see that each $S^\eq_i$ eliminates $\exists^\infty$ so $S^*_\cup$ exists by Fact~\ref{fact:minh}.
It is also easy to see that $S^*_\cup$ is the blow up of $T^*_\cup$.
It follows that algebraic closure in $S^*_\cup$ is trivial and $S^*_\cup$ interprets $T^*_\cup$.

\medskip\noindent
We show that $S^*_\cup$ is in general not $\acl$-closed.
So if $S^*_\cup$ is $\acl$-complete then every $S^*_\cup$-formula is equivalent to a boolean combination of $S_1$ and $S_2$ formulas.
As stable formulas are closed under boolean combinations it would follow that if $S_1$ and $S_2$ are stable then $S^*_\cup$ is stable.
It is easy to produce examples of stable $T_1,T_2$ such that $T^*_\cup$ has $\mathrm{TP}_2$, see Proposition~\ref{prop:tp2}.
For example if $T_1$ is theory of an equivalence relation with infinitely many infinite classes and $T_2$ is the theory of infinite vector spaces over a fixed finite field $\F$, then $T^*_\cup$ has $\mathrm{TP}_2$, so $S^*_\cup$ is not $\acl$-closed.

\subsection{Multiple valuations}
\label{section:multi-val}
Let $T_\cap$ be the theory of algebraically closed fields and for each $i \in I$ let $T_i$ be the theory of an algebraically closed field with a nontrivial valuation $v_i$.
The $T^*_\cup$ exists and is the theory of an algebraically closed field with an independent family $(v_i)_{i \in I}$ of valuations.
It is well known that the theory of an algebraically closed valued field is $\NIP$.
Fact~\ref{fact:acf-val} is due to Johnson~\cite{Johnson-thesis}.

\begin{fact}
\label{fact:acf-val}
The theory $T^*_\cup$ has $\mathrm{NTP}_2$ and $\mathrm{IP}$.
\end{fact}

\noindent
This witnesses sharpness of Proposition~\ref{prop:stabnip}.
Let $K$ be an algebraically closed field and $(v_i)_{i \in I}$ be an independent family of non-trivial valuations on $K$.
Each $(K;v_i)$ is $\NIP$ and algebraic closure in each $(K;v_i)$ agrees with algebraic closure in $K$, see for example \cite{lou-dimension}.
However, algebraic and definable closure do not agree in $K$.

\medskip\noindent
We discuss a seemingly similar structure.
Let $P$ be the set of primes.
Let $L_\cap$ be the language of abelian groups and $T_\cap$ be the theory of $(\zz;+)$.
For each $p \in P$ let $L_p$ be the expansion of $L_\cap$ by a binary relation $\elesub_p$ and let $T_p$ be the theory of $(\zz;+,\elesub_p)$ where $k \elesub_p k'$ if and only if the $p$-adic valuation of $k$ is less than or equal to the $p$-adic valuation of $k'$.
Alouf and d'Elb\'ee~\cite{AldE} prove Fact~\ref{fact:p-adic}.

\begin{fact}
\label{fact:p-adic}
The model companion $T^*_\cup$ of $T_\cup$ exists and agrees with the theory of $(\zz;+,(\elesub_p)_{p \in P})$.
Furthermore every $L_\cup$-formula is $T^*_\cup$-equivalent to a boolean combination of formulas from the various $L_p$.
\end{fact}

\noindent
It is also shown in \cite{AldE} that each $T_p$ is $\NIP$ and furthermore that $T^*_\cup$ is $\NIP$.
Note that that the second claim follows from the first claim and Fact~\ref{fact:p-adic} as a boolean combination of $\NIP$ formulas is $\NIP$.

\medskip\noindent
This is a natural example where the conditions of Proposition~\ref{prop:stabnip} are satisfied.
Algebraic closure in each $(\zz;+,\elesub_p)$ agrees with algebraic closure in $(\zz;+)$~\cite{AldE} and it is well known that algebraic and definable  closure agree in $(\zz;+)$.

\subsection{Cyclic multiplicative orders on $\F$.}
\label{section:cyclic}
We now describe the original motivating example of a fusion.
Let $G$ be an abelian group.
A \textbf{cyclic group order} on $G$ is a ternary relation $C$ such that for all $a,b,c,d \in G$
\begin{enumerate}
\item $C(a,b,c)$ implies $C(b,c,a)$,
\item $C(a,b,c)$ implies $\neg C(c,b,a)$,
\item $C(a,b,c)$ and $C(a,c,d)$ implies $C(a,b,d)$,
\item if $a,b,c$ are distinct then either $C(a,b,c)$ or $C(c,b,a)$,
\item $C(a,b,c)$ implies $C(da,db,dc)$.
\end{enumerate}
$(1) - (4)$ are the usual axioms of a cyclic order on a set.
Let $(F;+,\times)$ be a field.
A \textbf{cyclic multiplicative order} on $(F;+,\times)$ is a ternary relation $C$ on $F$ such that

\begin{enumerate}
\setcounter{enumi}{5}
\item $C(a,b,c)$ implies $a,b,c \in F^\times$ for all $a,b,c \in F$, and
\item the restriction of $C$ to $F^\times$ is a cyclic group order on $(F^\times;\times)$.
\end{enumerate}
We are primarily interested in the case when $F$ is an algebraic closure of a finite field.
In this case any cyclic multiplicative order on $(F;+,\times)$ is the pullback of the usual cyclic group order on $(\mathbb{R}/\mathbb{Z}; +)$ via an injective character $(F^\times;\times) \to (\mathbb{R}/\mathbb{Z}; +)$.

\medskip\noindent
Let $L_\cap$ be the language containing a binary function symbol $\times$ and $T_\cap$ be the $L_\cap$ theory of $(F;\times)$, where $(F;+,\times)$ is an algebraically closed field.
(Note that $T_\cap$ is bi-interpretable with the theory of the multiplicative group of an algebraically closed field.)
Let $L_1$ be the expansion of $L_\cap$ by a binary function symbol $+$ and $T_1$ be $\mathrm{ACF}$.
Let $L_2$ be the expansion of $L_\cap$ by a ternary function symbol $C$ and $T_2$ be the $L_2$-theory such that $(F;\times,C) \models T_2$ if and only if $(F;\times) \models T_\cap$, $C(a,b,c)$ implies $a,b,c \in F^\times$ for all $a,b,c \in F$, and the restriction of $C$ to $F^\times$ is a cyclic group order on $(F^\times;\times)$.
%Our description is somewhat informal, see \cite{Minh-1} for precise definitions.
%Fix an algebraic closure $\F$ of a finite field and an injective character $\chi : \F^\times \to \mathbb{S}^1$ to the unit circle.
%Let $C$ be the cyclic order on $\F^\times$ induced by $\chi$.
%Let $I = \{1,2\}$, $L_\cap$ be the language of multiplicative groups, $T_\cap$ be the $L_\cap$-theory of $\F^\times$, $L_1$ be the language of rings, $T_1$ be the $L_1$-theory of $\F$, $L_2$ be the language of cyclically ordered groups, and $T_2$ be the $L_2$-theory of $(\F^\times;C)$.
Fact~\ref{fact:cyclic-1} is essentially proven in \cite{Minh-1}.

\begin{fact}
\label{fact:cyclic-1}
The fusion $T^*_\cup$ exists and agrees the theory of an algebraic closure of a finite field equipped with a cyclic multiplicative group order.
\end{fact}

\noindent
It is easy to see that $T_\cap$ is stable and has weak elimination of imaginaries.
Fact~\ref{fact:cyclic-2} follows by Corollary~\ref{cor:bclpres}.
Fact~\ref{fact:cyclic-2} is also proven in \cite{Minh-1}.

\begin{fact}
\label{fact:cyclic-2}
Every $L_\cup$-formula $\psi(x)$ is $T^*_\cup$-equivalent to a a finite disjunction of b.e.\ formulas of the form $\exists y [\varphi_1(x,y) \land \varphi_2(x,y)] $ where $\varphi_i$ is an $L_i$-formula for $i \in \{1,2\}$.
\end{fact}

\noindent
Note that $T$ is $\SOP$.
An argument similar to the proof of Proposition~\ref{prop:tp2} shows that $T$ has $\mathrm{TP}_2$.
So $T$ does not satisfy any known classification theoretic property.
However, $T_\cap$ is strongly minimal, $T_1$ and $T_2$ are both $\NIP$, and we have a reasonable description of $T$-definable sets in terms of $T_1$- and $T_2$-definable sets.
So we conjecture that $T$ satisfies some as-yet-undiscovered classification theoretic property.
More generally, we make the following vague conjecture.
Recall that a set system $(X;\mathcal{S})$ consists of a set $X$ and a family $\mathcal{S}$ of subsets of $X$.

\begin{v-conj}
There is a classification-theoretic property $\bigstar$ such that:
\begin{enumerate}
\item there is a property $P$ of set systems such that $T$ is $\bigstar$ if and only if \\ $(M^x ; \{ \delta(M^x,b) : b \in M^y\})$ has $P$ for every formula $\delta(x,y)$ and $\sM \models T$,
\item both $\NSOP_1$ and $\mathrm{NTP}_2$ imply $\bigstar$,
\item $\NSOP_1$ is equivalent to the conjunction of $\bigstar$ and $\NSOP$,
\item $\bigstar$ is preserved by fusions under reasonably general conditions, and
\item any fusion of $\NIP$ theories over a stable base is $\bigstar$.
\end{enumerate}
\end{v-conj}

\noindent
We describe two other examples of theories which are $\SOP$ and $\mathrm{TP}_2$, and should have $\bigstar$.
One example is the model companion of the theory of a real closed ordered field $(R;+,\times,<)$ equipped with a subgroup of $(R^\times;\times)$, this theory was shown to exist in \cite{alexi-subgroup}.
The second example requires slightly more explanation.
Let $\mathbb{F}$ be a field.
Granger studied the theory of an $\mathbb{F}$-vector space equipped with a generic symmetric bilinear form~\cite{Granger-thesis}.
This theory always has $\mathrm{TP}_2$.
When $\mathbb{F}$ is algebraically closed this theory is $\NSOP_1$~\cite{CR}, we expect that it is $\bigstar$ when $\mathbb{F}$ is real closed.

\section{Dense linear orders are indecomposable}
\label{section:dlo}
\noindent
We have given many examples of theories which ``decompose" as fusions.
We finish with an example of a theory which does not.
There should be many such examples, but it appears difficult to establish this.
Let $T$ be a complete $L$-theory.
We say that $T$ is \textbf{interpolatively decomposable} if there are $(T_i)_{i \in I}$ such that $T^*_\cup$ exists, $T$ is bi-interpretable with $T^*_\cup$, and $T$ is not interpretable in any $T_i$.
For example an unstable theory which is bi-interpretable with a fusion of a family of stable theories is interpolatively decomposable.
We say that $T$ is \textbf{interpolatively indecomposable} if it is not interpolatively decomposable.
We expect that $\ACF$ and $\mathrm{RCF}$ are interpolatively indecomposable, but we do not have a proof at present.
We do not even have a proof that the theory of $(\nn,<)$ is interpolatively indecomposable.

\begin{thm}
\label{thm:dlo}
The theory $\mathrm{DLO}$ of dense linear orders without endpoints is interpolatively indecomposable.
\end{thm}

\noindent
We will use a deep result on $\NIP$ structures.
Suppose that $T$ is $\aleph_0$-categorical, unstable, and $\NIP$.
Simon~\cite{Pierre2} shows that $T$ interprets an infinite linear order.
While this is not explicitly stated in that paper, the proof shows that $T$ interprets a dense linear order.
(It is also not difficult to show directly that an infinite $\aleph_0$-categorical linear order interprets a dense linear order.)
So we have Fact~\ref{fact:pierre}.

\begin{fact}
\label{fact:pierre}
If $T$ is $\aleph_0$-categorical, $\NIP$, and unstable, then $T$ interprets $\mathrm{DLO}$.
\end{fact}

%\noindent
%We also apply Lemma~\ref{lem:dlo}.

%\begin{lem}
%\label{lem:dlo}
%Suppose that $(L,<)$ is an infinite linear order which eliminates $\exists^\infty$.
%Then $(L,<)$ interprets an infinite dense linear order without endpoints.
%\end{lem}

%\begin{proof}
%Note  first that it suffices to show that $(L,<)$ interprets an infinite dense linear order.
%Recall that an infinite discrete linear order does not eliminate $\exists^\infty$.
%So if there is an infinite interval $I \subseteq L$ such that $(I,<)$ is discrete
%\end{proof}

\noindent
We now prove Theorem~\ref{thm:dlo}.

\begin{proof}
Suppose that $\mathrm{DLO}$ is bi-interpretable with $T^*_\cup$.
So each $T_i$ is interpretable in $\mathrm{DLO}$.
By Corollary~\ref{cor:nip} some $T_j$ is unstable.
Note that $T_j$ is $\aleph_0$-categorical and $\NIP$ as $\mathrm{DLO}$ is $\aleph_0$-categorical and $\NIP$.
Apply Fact~\ref{fact:pierre}.
\end{proof}

\noindent
The most basic example of an unstable $\NSOP$ theory (the random graph) decomposes as a fusion of two stable theories and the most basic example of an unstable $\NIP$ theory ($\mathrm{DLO}$) is indecomposable. 

\medskip\noindent
Finally we let $\mathrm{ACF}_p$ be the theory of algebraically closed fields of characteristic $p$ (possibly $p = 0$) and sketch a heuristic argument that $\mathrm{ACF}_p$ is interpolatively indecomposable.
Suppose that $\mathrm{ACF}_k$ is bi-interpretable with $T^*_\cup$.
So each $T_i$ is interpretable in $\mathrm{ACF}_p$.
It is a conjecture of Zil'ber that any structure interpretable in an algebraically closed field is either one-based or interprets an infinite field.
Any infinite field interpretable in an algebraically closed field $K$ is definably isomorphic to $K$~\cite{bouscaren}.
We expect that a fusion of a family of one-based theories cannot interpret an infinite field, but we do not have a proof.

\appendix

\section{$\bcl$-completeness }\label{app:Kcomp}

\noindent Recall the definitions of b.e.\ formula and $\bcl$ from Section~\ref{sec:Kcomplete}. 

\begin{rem}\label{rem:beconj}
The class of b.e.\ formulas is closed (up to $T$-equivalence) under conjunction: if $\exists y_1\, \psi_1(x,y_1)$ and $\exists y_2\, \psi_2(x,y_2)$ are b.e.\ with bounds $k_1$ and $k_2$ respectively, then \[(\exists y_1\, \psi_1(x,y_1))\land (\exists y_2\, \psi_2(x,y_2))\] is $T$-equivalent to \[\exists y_1y_2\, (\psi_1(x,y_1)\land \psi_2(x,y_2)),\] which is b.e.\ with bound $k_1\cdot k_2$.
\end{rem}

\begin{rem}\label{rem:bcl}
A formula $\varphi(x,y,z)$ is bounded in $yz$ if and only if it is bounded in $z$ and $\exists z\, \varphi(x,y,z)$ is bounded in $y$. As a consequence, $b\in \bcl(A)$ if and only if $b$ satisfies a formula $\exists z\, \varphi(a,y,z)$, such that $a\in A^x$ and $\varphi(x,y,z)$ is quantifier-free and bounded in $yz$. 
\end{rem}

\begin{lem} \label{lem: bclisclosureoperator}
If $A \subseteq \sM$ then $\langle A\rangle \subseteq \bcl(A)$.
Furthermore, $\bcl$ is a closure operator.
\end{lem}
\begin{proof}
Fix $A \subseteq \sM$.
Suppose $b\in \langle A\rangle$.
Then $t(a) = b$ for a term $t(x)$ and a tuple $a$ from $A$. 
Then the formula $t(x) = y$ is b.e.\ (taking $z$ to be the empty tuple of variables) and bounded in $y$ (with bound $1$), so it witnesses $b\in \bcl(A)$.

It follows that $A\subseteq \bcl(A)$, and it is clear that $A\subseteq B$ implies $\bcl(A)\subseteq \bcl(B)$.
It remains to show $\bcl$ is idempotent. 

Suppose $b\in \bcl(\bcl(A))$. Then by Remark~\ref{rem:bcl}, $\sM\models \exists z\, \varphi(a,b,z)$ for some quantifier-free formula $\varphi(x,y,z)$ which is bounded in $yz$ and some tuple $a = (a_1,\dots,a_n)$ from $\bcl(A)$. For each $1\leq j\leq n$, since $a_j$ is in $\bcl(A)$, $\sM\models \exists z_j\, \psi_j(d_j,a_j,z_j)$ for some quantifier-free formula $\psi_j(w_j,x_j,z_j)$ which is bounded in $x_jz_j$, and some tuple $d_j$ from $A$.

Then the quantifier-free formula \[\left(\bigwedge_{j=1}^n \psi_j(w_j,x_j,z_j)\right)\land \varphi(x_1,\dots,x_n,y,z)\] is bounded in $x_1\dots x_n y z_1\dots z_n z$ (by the product of the bounds for $\varphi$ and the $\psi_j$), and \[\sM\models \exists x_1\dots x_n z_1\dots z_n z\, \left(\bigwedge_{j=1}^n \psi_j(d_j,x_j,z_j)\right)\land \varphi(x_1,\dots,x_n,b,z),\] so $b\in \bcl(A)$ by Remark~\ref{rem:bcl}.
\end{proof}

\begin{thm}\label{thm:bclaclappendix}
The following are equivalent:
\begin{enumerate}
\item Every $L$-formula is $T$-equivalent to a finite disjunction of b.e.\ formulas.
\item $T$ is $\acl$-complete and $\acl = \bcl$ in $T$-models.
\item $T$ is $\bcl$-complete.
\end{enumerate}
\end{thm}
\begin{proof}
We assume (1) and prove (2).
We first show $\acl$ and $\bcl$ agree.
Suppose $A \subseteq \sM \models T$ and $b\in \acl(A)$, witnessed by an algebraic formula $\varphi(a,y)$ with parameters $a$ from $A$. 
Suppose there are exactly $k$ tuples in $M^y$ satisfying $\varphi(a,y)$. 
Let $\varphi'(x,y)$ be the formula \[\varphi(x,y)\land \exists^{\leq k} y'\, \varphi(x,y'),\]
and note $\varphi'(x,y)$ is bounded in $y$.
By assumption, $\varphi'(x,y)$ is equivalent to a finite disjunction of boundedly existential formulas, so there is some boundedly existential formula $\psi(x,y)$ such that $T\models \psi(x,y)\rightarrow \varphi'(x,y)$ and $\sM\models \psi(a,b)$. Since $\varphi'(x,y)$ is bounded in $y$, so is $\psi(x,y)$, and hence $b\in \bcl(A)$.

We continue to assume (1) and show $T$ is $\acl$-complete.
Suppose $\sA$ is an algebraically closed substructure of $\sM\models T$ and $f\colon \sA\to \sN\models T$ is an embedding. 
We show that for any formula $\varphi(x)$, if $\sM\models \varphi(a)$, where $a\in A^x$, then $\sN\models \varphi(f(a))$. By our assumption, $\varphi(x)$ is equivalent to a finite disjunction of boundedly existential formulas, so there is some boundedly existential formula $\exists y\, \psi(x,y)$ such that $$T\models (\exists y\, \psi(x,y))\rightarrow \varphi(x)\quad \text{and}\quad \sM\models \exists y\, \psi(a,y).$$ Let $b\in M^y$ be a witness for the existential quantifier. Then each component of the tuple $b$ is in $\acl(a)\subseteq A$, since $A$ is algebraically closed. And $\psi$ is quantifier-free, so $\sN\models \psi(f(a),f(b))$, and hence $\sN\models \varphi(f(a))$. 

It is clear that (2) implies (3).

We now assume (3) and prove (1).
For any finite tuple of variables $x$, let $\Delta_x$ be the set of boundedly existential formulas with free variables from $x$.

\medskip

\emph{Claim:} For all models $\sM$ and $\sN$ of $T$ and all tuples $a\in M^x$ and $a'\in N^x$, if $\tp_{\Delta_x}(a) \subseteq \tp_{\Delta_x}(a')$, then $\tp(a) = \tp(a')$.

\emph{Proof of claim:} Suppose that $\sM$ and $\sN$ are models of $T$, $a\in M^x$, $a'\in N^x$, and $\tp_{\Delta_x}(a)\subseteq \tp_{\Delta_x}(a')$. Let $y$ be a tuple of variables enumerating the elements of $\bcl(a)$ which are not in $a$. Let $p(x,y) = \qftp(\bcl(a))$, and let $q(x) = \tp(a')$. We claim that $T\cup p(x,y)\cup q(x)$ is consistent. 

Let $b = (b_1,\dots, b_n)$ be a finite tuple from $\bcl(a)$ which is disjoint from $a$, and let $\psi(x,y')$ be a quantifier-free formula such that $\sM\models \psi(a,b)$ (where $y' = (y_1,\dots,y_n)$ is the finite subtuple of $y$ enumerating $b$). 

For each $1\leq j\leq n$, the fact that $b_j\in \bcl(a)$ is witnessed by $\sM\models \exists z_j\, \varphi_j(a,b_j,z_j)$, where $\varphi_j(x,y_j,z_j)$ is quantifier-free and bounded in $y_jz_j$ (by Remark~\ref{rem:bcl}). Letting $z = (z_1, \dots, z_n)$, the conjunction $\bigwedge_{j=1}^n \varphi_j(x,y_j,z_j)$ is a quantifier-free formula $\varphi(x,y',z)$ which is bounded in $y'z$. It follows that $\varphi(x,y',z)\land \psi(x,y')$ is also bounded in $y'z$, and $\sM\models \exists z\, (\varphi(a,b,z)\land \psi(a,b))$. Then \[\exists y'z\,(\varphi(x,y',z)\land \psi(x,y'))\in \tp_{\Delta_x}(a)\subseteq \tp_{\Delta_x}(a'),\] so $\sN\models \exists y'z\, (\varphi(a',y',z)\land \psi(a',y'))$. Letting $b'\in N_{y'}$ be a witness for the first block of existential quantifiers, $\sN\models \psi(a',b')$, so $T\cup \{\psi(x,y')\}\cup q(x)$ is consistent.

By compactness, $T\cup p(x,y)\cup q(x)$ is consistent, so there exists a model $\sN'\models T$, a tuple $a''\in (N')^x$ realizing $q(x)$, and an embedding $f\colon \bcl(a)\to \sN'$ such that $f(a) = a''$. By $\bcl$-completeness, we have $\tp(a) = \tp(a'') = \tp(a')$, as was to be shown.

\medskip

Having established the claim, we conclude with a standard compactness argument. Let $\varphi(x)$ be an $L$-formula. Suppose $\sM\models T$ and $\sM\models \varphi(a)$. Let $p_a(x) = \tp_{\Delta_x}(a)$. By the claim, $T\cup p_a(x)\cup \{\lnot \varphi(x)\}$ is inconsistent. Since $p_a(x)$ is closed under finite conjunctions (up to equivalence) by Remark~\ref{rem:beconj}, there is a formula $\psi_a(x)\in p_a(x)$ such that $T\models \psi_a(x)\rightarrow \varphi(x)$. 

Now 
\[T\cup \{\varphi(x)\}\cup \{\lnot \psi_a(x)\mid \sM\models T\text{ and } \sM\models \varphi(a)\}\]
is inconsistent, so there are finitely many $a_1,\dots,a_n$ such that $T\models \varphi\rightarrow (\bigvee_{i=1}^n \psi_{a_i}(x))$. Since also $T\models (\bigvee_{i=1}^n \psi_{a_i}(x))\rightarrow \varphi(x)$, we have shown that $\varphi(x)$ is $T$-equivalent to $\bigvee_{i=1}^n \psi_{a_i}(x)$.
\end{proof}

\noindent It may be surprising that $\acl$-completeness does not already imply every formula is equivalent to a finite disjunction of b.e.\ formulas, i.e., $\acl$-completeness is not equivalent to $\bcl$-completeness. 
We give a counterexample.

\begin{example}
Let $L$ be the language with a single unary function symbol $f$. We denote by $E(x,y)$ the equivalence relation defined by $f(x) = f(y)$. We say an element of an $L$-structure is \textbf{special} if it is in the image of $f$. Let $T$ be the theory asserting the following:
\begin{enumerate}
\item Models of $T$ are nonempty.
\item There are no cycles, i.e., for all $n\geq 1$, $\forall x\, f^n(x)\neq x$. 
\item Each $E$-class is infinite and contains exactly one special element. 
\end{enumerate}
Every $T$-model can be decomposed into a disjoint union of \textbf{connected components}, each of which is a chain of $E$-classes, $(C_n)_{n\in \mathbb{Z}}$, such that each class $C_n$ contains a unique special element $a_n$, and $f(b) = a_n$ for all $b\in C_{n-1}$.

Let $A$ be a subset of a $T$-model. Then $\acl(A)$ consists of $A$, together with the $\mathbb{Z}$-indexed chain of special elements in each connected component which meets $A$. But $\bcl(A)$ is just the substructure generated by $A$: it only contains the special elements from $E$-classes further along in the chain than some element of $A$. Indeed, if $a_n$ is the unique special element in class $C_n$, $a_n\notin A$, and no element of $A$ is in any class $C_m$ with $m<n$ in the same connected component, then $a_n$ does not satisfy any bounded and b.e.\ formula with parameters from $A$.

It is not hard to show that $T$ is $\acl$-complete (and hence complete, since $\acl(\emptyset) = \emptyset$), but not $\bcl$-complete. For an explicit example of a formula which is not equivalent to a finite disjunction of b.e.\ formulas, consider the formula \[
\exists y\, f(y) = x
\]
defining the special elements.
\end{example}

\section{Extendability of forking independence}\label{app:stationary}

\noindent In this appendix, we adopt the notation and terminology of Section~\ref{sec:stationary}. Our goal is to understand when forking independence $\eind[f]$ in a theory $T$ is extendable to an arbitrary expansion $T'$. In particular, we show that when $T$ is stable with weak elimination of imaginaries, $\eind[f]$ is always stationary and extendable to $T'$.

\medskip\noindent We recall a few variations on the notion of elimination of imaginaries (see~\cite{CF}).
\begin{enumerate}
\item $T$ has \textbf{elimination of imaginaries} if every $a \in \monster^{\eq}$ is interdefinable with some $b \in \monster$, i.e., $a \in \dcl^{\eq}(b)$ and $b \in \dcl^{\eq}(a)$. 
\item $T$ has \textbf{weak elimination of imaginaries} if for every $a\in \monster^{\eq}$ there is some $b\in \monster$ such that $a \in \dcl^{\eq}(b)$ and $b \in \acl^{\eq}(a)$. 
\item $T$ has \textbf{geometric elimination of imaginaries} if every $a \in \monster^{\eq}$ is interalgebraic with some $b \in \monster$, i.e., $a \in \acl^{\eq}(b)$ and $b \in \acl^{\eq}(a)$. 
\end{enumerate}

\noindent Let $\delta(x,y)$ be a formula. An \textbf{instance} of $\delta$ is a formula $\delta(x,b)$ with $b\in \monsterset^y$, and a \textbf{$\delta$-formula} is a Boolean combination of instances of $\delta$. A \textbf{global $\delta$-type} is a maximal consistent set of $\delta$-formulas with parameters from $\monsterset$. We denote by $S_\delta(\monsterset)$ the Stone space of global $\delta$-types. 

\medskip \noindent The following lemma is a well-known fact about the existence of weak canonical bases for $\delta$-types when $\delta(x,y)$ is stable. 

\begin{lem}\label{lem:weakcb}
Suppose $T$ has geometric elimination of imaginaries, and $\delta(x,y)$ is a stable formula. For any $q(x)\in S_\delta(\monsterset)$, there exists a tuple $d$ such that:
\begin{enumerate}
    \item $q(x)$ has finite orbit under automorphisms of $\monster$ fixing $d$.
    \item $d$ has finite orbit under automorphisms of $\monster$ fixing $q(x)$.
    \item $q(x)$ does not divide over $d$.
\end{enumerate}
If $T$ has weak elimination of imaginaries, we can arrange that $d$ is fixed by automorphisms of $\monster$ fixing $q(x)$. And if $T$ has elimination of imaginaries,  we can further arrange that $q(x)$ is fixed by automorphisms of $\monster$ fixing $d$.
\end{lem}

\begin{proof}
Let $e\in \monster^\eq$ be the canonical base for $q(x)$. Then $q(x)$ is fixed by all automorphisms fixing $e$, $e$ is fixed by all automorphisms fixing $q(x)$, and $q(x)$ does not divide over $e$. By geometric elimination of imaginaries, $e$ is interalgebraic with a real tuple $d$, and (1), (2), and (3) follow immediately. The cases when $T$ has elimination of imaginaries or weak elimination of imaginaries are similar.  
\end{proof}

\noindent $T$ has \textbf{stable forking} if whenever a complete type $p(x)$ over $B$ forks over $A\subseteq B$, then there is a stable formula $\delta(x,y)$ such that $\delta(x,b) \in p(x)$ and $\delta(x,b)$ forks over $A$. Every theory with stable forking is simple; the converse is the Stable Forking Conjecture, which remains open (see~\cite{Kim-Pillay-Stable-Forking}).

\medskip \noindent The following lemma is essentially the same idea as~\cite[Lemma 3]{PillayTsuboi}, which itself makes use of key ideas from~\cite[Lemmas 5.5 and 5.8]{HrushovskiPillay}.  

\begin{lem}\label{lem:forkingreduct}
Suppose $T$ has stable forking and geometric elimination of imaginaries. Then $\eind[f]$ in $T$ is extendable to $T'$, i.e., has full existence over algebraically closed sets in $T'$.
\end{lem}
\begin{proof}
Suppose towards a contradiction that there exist sets $A$, $B$, and $C$ in $\monster'$ such that $C = \acl_{L'}(C)$, and for any $A^*$ with $\tp_{L'}(A^*/C) = \tp_{L'}(A/C)$, $A^*\enind[f]_C B$ in $\monster$. We may assume $C\subseteq B$. Let $p(x) = \tp_{L'}(A/C)$. Since $T$ has stable forking, the fact that $\tp_{L}(A^*/B)$ forks over $C$ is always witnessed by a stable $L$-formula. 
So the partial type 
\begin{align*}
p(x)\cup \{\lnot \delta(x,b) \mid\,\, & \delta(x,y)\in L\text{ is stable, and } \delta(x,b)\text{ forks over $C$ in } \monster\}
\end{align*} 
is not satisfiable in $\monster'$. By saturation and compactness, we may assume that $A$ is finite and $x$ is a finite tuple of variables. And as stable formulas and forking formulas are closed under disjunctions, there is an $L'(C)$-formula $\varphi(x)\in p(x)$, a stable $L$-formula $\delta(x,y)$, and $b\in \monsterset^y$ such that $\delta(x,b)$ forks over $C$, and
\[ \monster'\models \forall x\,(\varphi(x)\rightarrow \delta(x,b)).\]
Since forking and dividing agree in simple theories~\cite[Prop.\ 5.17]{Casanovas}, $\delta(x,b)$ divides over $C$.

Let $[\varphi]$ be the set of all $\delta$-types in $S_\delta(\monsterset)$ which are consistent with $\varphi(x)$. 
This is a closed set in $S_\delta(\monsterset)$: it consists of all global $\delta$-types $r(x)$ such that $\chi(x)\in r(x)$ whenever $\chi(x)$ is a $\delta$-formula and $\varphi(\monster')\subseteq \chi(\monster')$. In particular, if $r(x)\in [\varphi]$, then $\delta(x,b)\in r(x)$. 
Since $\delta$ is stable, $[\varphi]$ contains finitely many points of maximal Cantor-Bendixson rank.
Let $q(x)$ be such a point.

Let $d$ be the weak canonical base for $q(x)$ obtained in Lemma~\ref{lem:weakcb}. Since $[\varphi]$ is fixed setwise by any $L'$-automorphism fixing $C$, $q(x)$ has finitely many conjugates under such automorphisms. It follows that $d$ too has finitely many conjugates, so $d \in C$, as $C$ is algebraically closed in $\monster'$. But then $q(x)$ does not divide over $C$, contradicting the fact that $\delta(x,b)\in q(x)$.
\end{proof}

\begin{rem}
The following counterexample shows the assumptions of geometric elimination of imaginaries in $T$ and $C = \acl(C)$ in $\monster'$ (not just in $\monster$) are necessary in Lemma~\ref{lem:forkingreduct}.  Let $T$ be the theory of an equivalence relation with infinitely many infinite classes. Let $T'$ be the expansion of this theory by a single unary predicate $P$ naming one of the classes. Let $a$ and $b$ be two elements of the class named by $P$ in $\monster'$, and let $C = \emptyset$ (which is algebraically closed in $\monster$ and $\monster'$). For any $a^*$ such that $\tp_{L'}(a^*/\emptyset) = \tp_{L'}(a/\emptyset)$, we have $a^*Eb$, and $xEb$ forks over $\emptyset$ in $\monster$. To fix this, we move to $\monster^{\eq}$, so we have another sort containing names for all the $E$-classes. Note that $\acl^{\eq}(\emptyset)$ in $\monster$ still doesn't contain any of these names. But $\acl^{\eq}(\emptyset)$ in $\monster'$ contains the name for the class named by $P$, since it is fixed by $L'$-automorphisms. And we recover the lemma, since $xEb$ does not fork over the name for the $E$-class of $b$.
\end{rem}

\begin{rem}
The conclusion of Lemma~\ref{lem:forkingreduct} also  fails when there are unstable forking formulas in $T$.
Let $T$ be be the theory of $(\mathbb{Q},<)$ and $T'$ be the expansion of $T$ by a unary predicate $P$ defining an open interval $(p,p')$, where $p<p'$ are irrational real numbers.
Let $b_1<a<b_2$ be elements of $\monster'$ such that $a\in P$ and $b_1,b_2\notin P$. Let $C = \emptyset$ (which is algebraically closed in $\monster'$). 
Then for any realization $a^*$ of $\tp_{L'}(a/\emptyset)$, we have $a^*\enind[f]_\emptyset b_1b_2$ in $\monster$, witnessed by the formula $b_1<x<b_2$.
\end{rem}

\begin{rem}\label{rem:stationarityisnecessary}
It is not possible to strengthen the conclusion of Lemma~\ref{lem:forkingreduct} to the following: For all small sets $A$, $B$, and $C$, such that $C = \acl_{L'}(C)$, and for any $A''$ such that $\tp_L(A''/C) = \tp_L(A/C)$ and $A''\eind[f]_C B$ in $\monster$, there exists $A'$ with $\tp_{L'}(A'/C) = \tp_{L'}(A/C)$ and $\tp_{L}(A'/BC) = \tp_{L}(A''/BC)$. 

That is, while it is possible to find a realization $A'$ of $\tp_{L'}(A/C)$ such that $\tp_L(A'/BC)$ is a nonforking extension of $\tp_L(A/C)$, it is not possible in general to obtain an arbitrary nonforking extension of $\tp_L(A/C)$ in this way. 

For a counterexample, consider the theories $T$ and $T'$ from Example~\ref{ex:rg}. $T$ has stable forking and geometric elimination of imaginaries. Let $a$ and $b$ be elements of the clique defined by $P$ in $\monster'$, and let $C = \emptyset$ (which is algebraically closed in $\monster'$). Let $a''$ be any element such that $\monster'\models \lnot a''Eb$, and note that $a''\eind[f]_\emptyset b$ and $\tp_L(a''/\emptyset) = \tp_L(a/\emptyset)$ (there is only one $1$-type over the empty set with respect to $T$). But for any $a'$ with $\tp_{L'}(a'/\emptyset) = \tp_{L'}(a/\emptyset)$, $\monster'\models P(a')$, so $a'Eb$, and $\tp_{L}(a'/b)\neq \tp_{L}(a''/b)$.
\end{rem}

\noindent We have seen that the hypotheses of stable forking (and hence simplicity) and geometric elimination of imaginaries in $T$ are sufficient to ensure that $\eind[f]$ has full existence over algebraically closed sets in $T'$, with no further assumptions on $T'$. But we would also like $\eind[f]$ to be a stationary independence relation in $T$. 

\medskip\noindent In a simple theory $T$, $\eind[f]$  satisfies stationarity over $\acl^{\eq}$-closed sets if and only if $T$ is stable~\cite[Chapter 11]{Casanovas}. And a stable theory has weak elimination of imaginaries if and only if it has geometric elimination of imaginaries and $\eind[f]$ satisfies stationarity over $\acl$-closed sets~\cite[Propositions 3.2 and 3.4]{CF}. So we have proven the following proposition, under very natural hypotheses. 

\begin{prop}\label{prop:stationarityappendix}
Suppose $T$ is stable with weak elimination of imaginaries. Then $\eind[f]$ is a stationary and extendable independence relation in $T$.
\end{prop}

\section{Abstract independence without base monotonicity}\label{app:reasonable}

\noindent
Throughout this section, we assume $T$ is a complete theory, $\monster \models T$ is a monster model, and $\eind[*]$ is a ternary relation on subsets of $\monster$ satisfying: 
\begin{enumerate}
\item \textbf{Invariance}: If $A\eind[*]_C B$ and $ABC \equiv A'B'C'$, then $A'\eind[*]_{C'} B'$. 
\item \textbf{Monotonicity}: If $A\eind[*]_C B$, $A'\subseteq A$, and $B'\subseteq B$, then $A'\eind[*]_C B'$. 
\item \textbf{Symmetry}: If $A\eind[*]_C B$, then $B\eind[*]_C A$. 
\item \textbf{Transitivity}: Suppose $C\subseteq B \subseteq A$. If $A\eind[*]_B D$ and $B\eind[*]_C D$, then $A\eind[*]_C D$. 
\item \textbf{Normality}: If $A\eind[*]_C B$, then $AC\eind[*]_C B$. 
\item \textbf{Full existence}: For any $A$, $B$, and $C$, there exists $A'\equiv_C A$ such that $A'\eind[*]_C B$. 
\item \textbf{Finite character}: If $A'\eind[*]_C B$ for all finite $A'\subseteq A$, then $A\eind[*]_C B$. 
\item \textbf{Strong local character}: For every cardinal $\lambda$, there exists a cardinal $\kappa$ such that for all $A$ with $|A| = \lambda$, all $B$, and all $D\subseteq B$, there exists $D\subseteq C\subseteq B$ with $|C|\leq\max(|D|,\kappa)$ and $A\eind[*]_C B$. 
\end{enumerate}

This list of axioms is very similar to Adler's axioms for independence relations in~\cite{Adler}, with the following differences:
\begin{itemize}
\item We do not assume base monotonicity. 
\item We assume symmetry. Adler proves that symmetry follows from his other axioms, but this proof uses base monotonicity. 
\item Our formulation of local character is stronger and serves as a partial replacement for base monotonicity: If $A\eind[*]_C B$ and $C\subseteq B' \subseteq B$, we don't necessarily have $A\eind[*]_{B'} B$. But by strong local character there is $B'\subseteq B'' \subseteq B$, with $B''$ not too much bigger than $B'$, such that $A\eind[*]_{B''} B$. 
\item We assume full existence instead of extension. But extension follows from our axioms (see Remark~\ref{rem:axioms} below).
\end{itemize}

\begin{example}\label{ex:forking}
Suppose $T_0$ is a reduct of $T$ which is simple with stable forking and geometric elimination of imaginaries, and let $\monster_0$ be the corresponding reduct of $\monster$. 
As in Section~\ref{sec:indreduct}, we define:
$$A\eind[r]_C B \iff \acl(AC)\eind[f]_{\acl(C)} \acl(BC) \text{ in } \monster_0.$$
where $\acl$ is the algebraic closure operator in $\monster$. Then $\eind[r]$ satisfies our axioms.

Invariance, monotonicity, symmetry, normality, and finite character are clear from the definition and the corresponding properties for $\eind[f]$ in simple theories.

Transitivity: Suppose $C\subseteq B \subseteq A$. If $A\eind[r]_B D$ and $B\eind[r]_C D$, we have 
$$\acl(A)\eind[f]_{\acl(B)} \acl(BD) \quad \text{and} \quad \acl(B)\eind[f]_{\acl(C)} \acl(CD)\quad \text{in} \quad \monster_0.$$

Since $\acl(CD)\subseteq \acl(BD)$, by monotonicity for $\eind[f]$, $\acl(A)\eind[f]_{\acl(B)} \acl(CD)$ in $\monster_0$, so by transitivity for $\eind[f]$, $\acl(A)\eind[f]_{\acl(C)} \acl(CD)$ in $\monster_0$, so $A\eind[r]_C D$. 

Full existence: Let 
$$ p(x) = \tp(\acl(AC)/\acl(C)) \quad \text{and} \quad q(y) = \tp(\acl(BC)/\acl(C)).$$
By~\ref{lem:forkingreduct} (this is where we use the assumptions of stable forking and geometric elimination of imaginaries), there are realizations $\widehat{A}$ of $p(x)$ and $\widehat{B}$ of $q(y)$ such that $\widehat{A}\eind[f]_{\acl(C)} \widehat{B}$ in $\monster_0$. Since $\widehat{B}\equiv_{\acl(C)} \acl(BC)$, we can move $\widehat{B}$ back to $\acl(BC)$ by an automorphism $\sigma$ fixing $\acl(C)$. Let $A'\subseteq\sigma(\widehat{A})$ be the image under $\sigma$ of the copy of $A$ in $\widehat{A}$, so $A'\equiv_{\acl(C)} A$. In particular, $A'\equiv_C A$, and $\sigma(\widehat{A}) = \acl(A'C)$, so $\acl(A'C)\eind[f]_{\acl(C)} \acl(BC)$, so $A'\eind[r]_C B$. 

Strong local character: Since $T_0$ is simple, there is a cardinal $\kappa(T)$ such that for all finite $a$ and all $B$, there exists $B'\subseteq B$ with $|B'|\leq \kappa(T)$ and $a\eind[f]_{B'} B$ in $\monster_0$. Given the cardinal $\lambda$, let $\kappa = \max(\kappa(T),\lambda,|L|)$. I claim first that for all $A$ with $|A|\leq \max(\lambda,|L|)$, for all $B$, and for all $D\subseteq B$, there exists $C$ with $D\subseteq C \subseteq B$ and $|C|\leq \max(|D|,\kappa)$ such that $A\eind[f]_{\acl(C)} \acl(B)$ in $\monster_0$. Indeed, for every finite tuple $a$ from $A$, we can find $B_a\subseteq \acl(B)$ with $|B_a|\leq \kappa(T)$ such that $a\eind[f]_{B_a} \acl(B)$ in $\monster_0$. Letting $B' = \bigcup_a B_a$, by base monotonicity and finite character for $\eind[f]$, $A\eind[f]_{B'} \acl(B)$, and $|B'|\leq \kappa$, since there are $|A|$-many finite tuples from $A$. Now we obtain $B''\subseteq B$ with $B'\subseteq \acl(B'')$ by replacing each element $b\in B'$ with a finite tuple $b_1,\dots,b_n\in B$ such that $b\in \acl(b_1,\dots,b_n)$. We still have $|B''|\leq \kappa$, and $A\eind[f]_{\acl(B'')} \acl(B)$. Finally, let $C = B''\cup D$, so $D\subseteq C\subseteq B$, and $|C|\leq \max(|D|,\kappa)$. By base monotonicity for $\eind[f]$, $A\eind[f]_{\acl(C)}\acl(B)$. 

Now suppose we are given $A$ with $|A| = \lambda$, $B$, and $D\subseteq B$. Build a sequence $(C_i)_{i\in \omega}$ such that for all $i\in \omega$, $|C_i|\leq \max(|D|,\kappa)$, $D\subseteq C_i\subseteq C_{i+1}\subseteq B$, and $\acl(AC_i)\eind[f]_{\acl(C_{i+1})} \acl(B)$. For the base case, set $C_0 = D$, and for the inductive step, we can use the claim in the last paragraph, since $|\acl(AC_i)|\leq \max(\lambda,|L|)$. 

Let $C = \bigcup_{i\in\omega} C_i$. Then $|C| \leq \max(|D|,\kappa)$ and $D\subseteq C\subseteq B$. I claim that $A\eind[r]_C B$, i.e., $\acl(AC)\eind[f]_{\acl(C)} \acl(B)$. By finite character for $\eind[f]$, it suffices to show that for every finite tuple $a$ from $\acl(AC)$, $a\eind[f]_{\acl(C)}\acl(B)$. But $a$ is already contained in $\acl(AC_i)$ for some $i\in \omega$, and by base monotonicity for $\eind[f]$, $\acl(AC_i)\eind[f]_{\acl(C)}\acl(B)$, since $\acl(C_{i+1})\subseteq \acl(C)\subseteq \acl(B)$. 

\end{example}

\begin{example}\label{ex:kim}
If $T$ is $\NSOP_1$, then Kim-independence $\eind[K]$ is only defined over models. However, if we restrict our axioms to the cases where all the sets in the base of $\eind[*]$ are models, then they are all satisfied by $\eind[K]$ in $\NSOP_1$ theories. Strong local character follows from~\cite[Theorem 1.1]{KRS} and transitivity from~\cite[Theorem 3.4]{KRTransitivity}. Moreover, in all known examples, $\eind[K]$ agrees over models with an independence relation defined over arbitrary sets, which satisfies all our axioms. Recently, some progress has been made toward showing this holds in general, see~\cite{DKR}. For an example in which $\eind[K]$ does not arise from forking independence in a simple reduct, take the independence relation $\eind[I]$ in the theory $T_{m,n}$ of generic $K_{m,n}$-free incidence structures ($m,n\geq 2$) defined in~\cite{CK}. 

\medskip\noindent
In such a situation, where $\eind[*]$ agrees with Kim-independence over models, we can view the main results of this section (reasonable extension, Theorem~\ref{thm:reasext*}, the reasonable chain condition, Theorem~\ref{thm:reaschain*}, and the reasonable independence theorem, Theorem~\ref{thm:reasind*}) as saying that we can obtain certain instances of base monotonicity. For example, reasonable extension says that for all $a\eind[K]_M b$ and for all $c$, there exists $a'$ such that $a'\equiv_{Mb} a$, $a'\eind[K]_M bc$, and $a'\eind[*]_{Mb} c$. The last assertion would be immediate in the presence of base monotonicity. 
\end{example}

\begin{question}\label{q:kim}
Does every $\NSOP_1$ theory admit an independence relation $\eind[*]$ satisfying our axioms, such that $\eind[*]_M = \eind[K]_M$ for all $M\prec \monster$?
\end{question}

\noindent
We now return to the case of general $\eind[*]$. When working with this independence relation, we will almost always use the key Lemmas~\ref{lem:morley} and~\ref{lem:longseq} below, instead of appealing directly to the axioms. 

\begin{rem}\label{rem:axioms}
Note that by monotonicity, normality, and transitivity, we have the following form of transitivity: If $A\eind[*]_{BC} D$ and $B\eind[*]_C D$, then $A\eind[*]_C D$. 

\medskip\noindent
Though we will not need to use it, it may be worth noting that extension follows from our axioms: Suppose $A\eind[*]_C B$ and $B\subseteq B'$. By full existence there exists $A'\equiv_{BC} A$ such that $A'\eind[*]_{BC} B'$. By invariance, symmetry, and the form of transitivity just observed, $A'\eind[*]_C B'$. 
\end{rem}

\begin{defn}
A sequence $(b_\alpha)_{\alpha<\mu}$ is $\eind[*]$-independent over $C$ if $b_\alpha \eind[*]_{C} b_{<\alpha}$ for all $0<\alpha<\mu$. The sequence is $\eind[*]$-independent over $C$ in the type $p(y)$ if additionally $b_\alpha$ realizes $p$ for all $\alpha<\mu$. The sequence is a $\eind[*]$-Morley sequence over $C$ if additionally it is $C$-indiscernible. 
\end{defn}

\begin{lem}\label{lem:morley}
Let $p(y)$ be a type over $C$, and let $\mu$ be a cardinal. Then there exists a $\eind[*]$-Morley sequence over $C$ in $p$ of length $\mu$. 
\end{lem}
\begin{proof}
Define a sequence $(c_i)_{i\in \omega}$ by induction: pick $b=c_0$ realizing $p$, and for all $n>0$, pick $c_n\equiv_C c_0$ such that $c_n\eind[*]_C c_{<n}$ by full existence. Now let $(b_\alpha)_{\alpha<\mu}$ be a $C$-indiscernible sequence based on $(c_i)_{i\in \omega}$. By finite character, symmetry, and invariance, $b_\alpha\eind[*]_C b_{<\alpha}$ for all $0<\alpha<\mu$. 
\end{proof}

\begin{lem}\label{lem:longseq}
Let $\lambda$ be a cardinal, and let $\kappa$ be the cardinal provided for $\lambda$ by strong local character. If $\mu$ is a regular cardinal greater than $\max(\kappa,|C|)$, and $(b_\alpha)_{\alpha<\mu}$ is an $\eind[*]_C$-independent sequence over $C$, then for any $a$ with $|a| = \lambda$, there exists $\beta<\mu$ such that $a\eind[*]_C b_{\beta'}$ for all $\beta\leq \beta' <\mu$.  
\end{lem}

\begin{proof}
By strong local character, there exists $B$ with $C\subseteq B\subseteq C(b_\alpha)_{\alpha<\mu}$ and $|B| = \max(\kappa,|C|)$ such that $a\eind[*]_{B} C(b_\alpha)_{\alpha<\lambda}$. Then there exists $\beta<\mu$ such that $B\subseteq C(b_\alpha)_{\alpha<\beta}$, and by monotonicity, for all $\beta\leq \beta'<\mu$,  $b_{\beta'}\eind[*]_{C} B$, so by symmetry $B\eind[*]_{C} b_{\beta'}$. By monotonicity again, $a\eind[*]_B b_{\beta'}$, and by transitivity, $a\eind[*]_{C} b_{\beta'}$.
\end{proof}

\noindent
In applying Lemma~\ref{lem:longseq}, we will typically just write that a Morley sequence is ``long enough'', meaning that its length $\mu$ satisfies the hypotheses. The only exception is in the proof of Theorem~\ref{thm:reasind*} below, where we have to be a bit more careful with our choice of cardinal. 

\medskip\noindent
\textbf{From now on, we assume $T$ is $\NSOP_1$} and examine the relationship between $\eind[K]$ and $\eind[*]$. Our ultimate goal is to prove the ``reasonable independence theorem'' (Theorem~\ref{thm:reasind*} below). In~\cite{KRExp}, Ramsey and the first-named author proved this theorem in the special case $\eind[*] = \eind[a]$, and our proof closely follows that in~\cite{KRExp}. 

\begin{thm}[Reasonable extension]\label{thm:reasext*}
For all $a\eind[K]_M b$ and for all $c$, there exists $a'$ such that $a'\equiv_{Mb} a$, $a'\eind[K]_M bc$, and $a'\eind[*]_{Mb} c$. 
\end{thm}
\begin{proof}
Let $(c_\alpha)_{\alpha< \mu}$ be a long enough $\eind[*]_{Mb}$-independent sequence in $\tp(c/Mb)$. By extension for $\eind[K]$, there exists $a''\equiv_{Mb} a$ such that $a''\eind[K]_M b(c_\alpha)_{\alpha< \mu}$. By Lemma~\ref{lem:longseq}, there exists $\beta<\mu$ with $a''\eind[*]_{Mb} c_\beta$. By monotonicity for Kim-independence, we also have $a''\eind[K]_M bc_\beta$. Let $\sigma$ be an automorphism moving $c_\beta$ to $c$  and fixing $Mb$, and let $a' = \sigma(a'')$. Then $a'\eind[K]_M bc$ and $a'\eind[*]_{Mb}c$. 
\end{proof}

\noindent
Next, we wish to build $\eind[*]$-Morley sequences over $Ma$ which are also Kim-independent from $a$ over $M$. Lemma~\ref{lem:weirdextension}, which handles the inductive step of the construction, uses the improved independence theorem from~\cite{KRExp}.

\begin{thm}[\cite{KRExp}, Theorem 2.13]\label{thm:strongind}
If \(a_0 \eind[K]_{M} b\), \(a_1 \eind[K]_{M} c\), \(b \eind[K]_{M} c\) and $a_0 \equiv_{M} a_1$, then there exists $a$ with $a \equiv_{Mb} a_0$, $a \equiv_{Mc} a_1$, $a \eind[K]_{M} bc$, $b\eind[K]_M ac$, and $c\eind[K]_M ab$.
\end{thm}

\begin{lem} \label{lem:weirdextension}
If $a \eind[K]_{M} b$ and $a \eind[K]_{M} c$, then there exists $b' \equiv_{M a} b$ such that $a \eind[K]_{M} b'c$ and $b' \eind[*]_{Ma} c$.  
\end{lem}

\begin{proof}
We build a sequence $(c_i)_{i<\omega}$ by induction, such that for all $i<\omega$, the following conditions hold: 
\begin{enumerate}
\item $c_i \equiv_{Ma} c$.
\item $c_i \eind[K]_M ac_{< i}$.
\item $c_i \eind[*]_{Ma} c_{< i}$.
\end{enumerate}

Set $c_0 = c$, and the induction step follows from Theorem~\ref{thm:reasext*}. 

Let $I = (c_\alpha)_{\alpha<\kappa}$ be a long enough $\eind[*]$-Morley sequence over $Ma$ based on $(c_i)_{i<\omega}$. I claim that there exists $b''\equiv_{Ma} b$ such that $a\eind[K]_M b''c_\alpha$ for all $\alpha<\kappa$. By compactness, it suffices to show that for all $n<\omega$, there exists $b_n\equiv_{Ma} b$ such that $a\eind[K]_M b_nc_i$ for all $i<n$. We argue by induction on $n$, additionally ensuring that $b_n\eind[K]_M ac_{<n}$ for all $n$.

In the base case, we may take $b_0 = b$. Suppose we are given $b_n$. By extension for Kim-independence, choose $b_n' \equiv_{M} b_n$ with $b_n' \eind[K]_{M} c_{n}$.  Then since we also have $b_n\eind[K]_M ac_{< n}$ and $c_{n} \eind[K]_{M} ac_{< n}$, we may apply the strengthened independence theorem (Theorem \ref{thm:strongind}), to find $b_{n+1}$ such that $b_{n+1}\equiv_{Mac_{< n}} b_n$, $b_{n+1}\equiv_{Mc_{n}} b_n'$, $b_{n+1} \eind[K]_{M} ac_{\leq n}$, and $ac_{< n}\eind[K]_M b_{n+1}c_{n}$.  In particular, we also have $b_{n+1}\equiv_{Ma} b_n \equiv_{Ma} b$ and $a\eind[K]_M b_{n+1}c_{i}$ for all $i<n+1$.

Having obtained our $b''$ as described above, by Lemma~\ref{lem:longseq} there exists $\beta<\kappa$ such that $b''\eind[*]_{Ma} c_\beta$. Let $\sigma$ be an automorphism fixing $Ma$ and moving $c_\beta$ to $c$, and let $b' = \sigma(b'')$. 
\end{proof}

\begin{thm}\label{thm:goodseq}
If $a \eind[K]_{M} b$, then for any cardinal $\kappa$, there exists a $\eind[*]$-Morley sequence over $Ma$, $I = (b_\alpha)_{\alpha<\kappa}$, with $b_0 = b$ and $a\eind[K]_M I$.
\end{thm}

\begin{proof}
We mimic the proof of Lemma~\ref{lem:morley}, building a sequence $(c_i)_{i < \omega}$ in $\tp(b/Ma)$ by induction, such that $c_i\equiv_{Ma} b$,  $c_{i} \eind[*]_{Ma} c_{<i}$, and additionally $a \eind[K]_{M} c_{\leq i}$ for all $i<\omega$. Set $c_{0} = b$, and for the induction step we can use Lemma~\ref{lem:weirdextension}, since $a\eind[K]_M c_{<i}$ and $a\eind[K]_M b$. 

Now let $I = (b_\alpha)_{\alpha < \kappa}$ be a $\eind[*]$-Morley sequence over $Ma$ based on $(c_i)_{i<\omega}$. After an automorphism fixing $Ma$, we may assume $b_0 = b$. And $a\eind[K]_M I$ by invariance and the finite character of Kim-independence.
\end{proof}

\noindent
Theorem~\ref{thm:kimchain} is a chain condition for Kim-independence.

\begin{thm}[\cite{Kim} Proposition 3.20]\label{thm:kimchain}
Suppose $a\eind[K]_M b$ and $I = (b_\alpha)_{\alpha<\kappa}$ is a $q$-Morley sequence over $M$ with $b_0 = b$, for some global $M$-invariant type $q(y)$ extending $\tp(b/M)$. Then there exists $a'$ such that $a'\equiv_{Mb} a$, $a'\eind[K]_M I$, and $I$ is $Ma'$-indiscernible. 
\end{thm}

\noindent
The next theorem strengthens the conclusion that $I$ is $Ma'$-indiscernible to the conclusion that $I$ is $\eind[*]$-Morley over $Ma'$, with the caveat that the indexing of $I$ has to be reversed. For convenience in the application of this theorem, we assume to start with that $I$ is a reverse $q$-Morley sequence.

\medskip\noindent
Let $p(x)$ be a global $M$-invariant type.
Then a {\bf reverse $p$-Morley sequence over $M$} is a sequence $(a_\alpha)_{\alpha<\kappa}$ such that $a_\alpha$ realizes $p|_{Ma_{>\alpha}}$ for all $\alpha<\kappa$. 
Note any two reverse $p$-Morley sequences of the same length realize the same type over $M$, since for any $\alpha_1<\alpha_2<\dots<\alpha_n<\kappa$, $\tp(a_{\alpha_n}a_{\alpha_{n-1}}\dots a_{\alpha_1}/M) = p^{\otimes n}|_M$. 

\begin{thm}[Reasonable chain condition]\label{thm:reaschain*}
Suppose $a\eind[K]_M b$ and $I = (b_\alpha)_{\alpha<\kappa}$ is a reverse $q$-Morley sequence over $M$ with $b_0 = b$, for some global $M$-invariant type $q(y)$ extending $\tp(b/M)$. Then there exists $a'$ such that $a'\equiv_{Mb} a$, $a'\eind[K]_M I$, and $I$ is a $\eind[*]$-Morley sequence over $Ma'$.
\end{thm}
\begin{proof}
We show by induction on $n$ that there exists $(c_0,\dots,c_n) \models q^{\otimes (n+1)}|_M$ such that $c_i\equiv_{Ma} b$ for all $i\leq n$, $c_i\eind[*]_{Ma} c_{>i}$ for all $0\leq i<n$, and $a\eind[K]_M (c_i)_{i\leq n}$.

When $n= 0$, taking $c_0 = b$ suffices. So suppose we are given the tuple $(c_0,\dots,c_n)$ by induction. By Theorem~\ref{thm:goodseq}, there is a long enough $\eind[*]$-Morley sequence $J = (d_{0,\alpha})_{\alpha<\kappa}$ over $Ma$ in $\tp(b/Ma)$, such that $a\eind[K]_M J$.

Let $(d_1,\dots,d_{n+1})$ realize $q^{\otimes(n+1)}|_{MJ}$. Since $d_{0,\alpha}\models q|_M$ for all $\alpha$, we have $(d_{0,\alpha},d_1,\dots,d_{n+1})\models q^{\otimes(n+2)}|_M$ for all $\alpha$. Let $\sigma\in \Aut(\monster/M)$ be such that $\sigma(c_i) = d_{i+1}$ for all $i\leq n$, and let $a' = \sigma(a)$. Now $a\equiv_M a'$, $a\eind[K]_M J$ (by choice of $J$), $a'\eind[K]_M d_1 \dots d_{n+1}$ (by invariance and induction), and $d_1\dots d_{n+1}\eind[K]_M J$ (since $\tp(d_1 \dots d_{n+1}/MJ)$ extends to a global $M$-invariant type). 

Applying the independence theorem, we find $a''$ with $a''\equiv_{MJ} a$, $a''\equiv_{Md_1\dots d_{n+1}} a'$, and $a''\eind[K]_M Jd_1\dots d_{n+1}$. By invariance, we have $d_i \eind[*]_{Ma''} d_{>i}$ for all $1\leq i < n+1$, and by Lemma~\ref{lem:longseq}, there is some $\beta <\kappa$ such that $d_1\dots d_{n+1} \eind[*]_{Ma''} d_{0,\beta}$. Defining $d_0 = d_{0,\beta}$, by symmetry we have $d_i \eind[*]_{Ma''} d_{>i}$ for all $0\leq i<n+1$.

It remains to move $a''$ back to $a$ by an automorphism $\sigma\in \Aut(\monster/M)$. The tuple $\sigma(d_0),\dots,\sigma(d_{n+1})$ has the desired properties: $a\sigma(d_0)\equiv_M a''d_0 \equiv_M ad_0 \equiv_M ab$, $a\sigma(d_i)\equiv_M a''d_i \equiv_M a'd_i \equiv ac_{i-1} \equiv ab$ for all $0<i\leq n+1$, $\sigma(d_i)\eind[*]_{Ma} \sigma(d_{>i})$ for all $0\leq i < n+1$, and $a\eind[K]_M \sigma(d_{\leq (n+1)})$. 

By compactness, and reversing our indexing, we can find a reverse $q$-Morley sequence over $M$, $(c_\alpha)_{\alpha<\kappa}$, such that $c_\alpha\equiv_{Ma} b$ for all $\alpha<\kappa$, $c_\alpha\eind[*]_{Ma} c_{<\alpha}$ for all $\alpha<\kappa$, and $a \eind[K]_M I'$. In fact, we can assume $I'$ is $Ma$-indiscernible, by replacing it with an $Ma$-indiscernible sequence based on it. As $I'$ and $I$ are both reverse $q$-Morley sequences over $M$, we can move $I'$ to $I$ by an automorphism $\sigma\in \Aut(\monster/M)$, and take $a' = \sigma(a)$.
\end{proof}

\noindent
Combining the reasonable chain condition with the usual chain condition gives us Lemma~\ref{lem:doublesequence}, which is the key ingredient in the proof of the reasonable independence theorem. 

\begin{lem}
\label{lem:doublesequence}
Suppose $a\eind[K]_M b$, and let $p(x)$ and $q(y)$ be global invariant types such that $\tp(a/M)\subseteq p(x)$ and $\tp(b/M)\subseteq q(y)$. Then for any $\kappa$ and $\kappa'$, there exist sequences $I = (a_\alpha)_{\alpha<\kappa}$ and $J = (b_\beta)_{\beta<\kappa'}$ such that:
\begin{enumerate}
\item $I$ is a reverse $p$-Morley sequence over $M$.
\item $J$ is a reverse $q$-Morley sequence over $M$.
\item $a_0 = a$ and $b_0 = b$.
\item $J$ is $\eind[*]$-Morley over $Ma$.
\item $I$ is $MJ$-indiscernible.
\item $I\eind[K]_M J$. 
\end{enumerate} 
In particular, $J$ is $\eind[*]$-Morley over $Ma_{\alpha}$ for all $\alpha<\kappa$, and $a_\alpha b_\beta \equiv_M ab$ for all $\alpha<\kappa$ and $\beta<\kappa'$. 
\end{lem}
\begin{proof}
First, note that the last claims follow from (3), (4) and (5). For all $\alpha<\kappa$, since $J$ is $\eind[*]$-Morley over $Ma$ and $a_\alpha\equiv_{MJ} a$, by invariance $J$ is $\eind[*]$-Morley over $Ma_\alpha$. And for all $\alpha<\kappa$ and $\beta<\kappa'$, $a_\alpha b_\beta \equiv_M ab_\beta \equiv_M ab$, since $b_0 = b$ and $J$ is $Ma$-indiscernible.

So let $J' = (b'_\beta)_{\beta<\kappa'}$ be any reverse $q$-Morley sequence over $M$ with $b'_0 = b$.   By Theorem~\ref{thm:reaschain*}, we can find $a'\equiv_{Mb} a$, $a'\eind[K]_M J'$, and $J'$ is a $\eind[*]$-Morley sequence over $Ma'$. Now let $I' = (a'_\alpha)_{\alpha<\kappa}$ be any reverse $p$-Morley sequence over $M$ with $a'_0 = a'$. Since $J\eind[K]_M a'$, by Theorem~\ref{thm:kimchain}, we can find $J'' = (b''_\beta)_{\beta<\kappa'}$ such that $J''\equiv_{Ma'}J'$ such that $J''\eind[K]_M I'$ and $I'$ is $MJ''$-indiscernible.  

Now $a'_0b''_0 = a'b''_0 \equiv_M a'b'_0 = a'b \equiv_M ab$, so we can pick an automorphism $\sigma$ fixing $M$ and moving $a'_0$ to $a$ and $b''_0$ to $b$. Let $I = \sigma(I')$ and $J = \sigma(J'')$. Conditions (1)-(6) follow from invariance.
\end{proof}

\begin{thm}[Reasonable independence theorem]\label{thm:reasind*}
If $a\eind[K]_M b$, $a'\eind[K]_M c$, $b\eind[K]_M c$, and $a\equiv_M a'$, then there exists $a''$ such that $a''\equiv_{Mb} a$, $a''\equiv_{Mc} a'$, and $a''\eind[K]_M bc$, and further $a''\eind[*]_{Mc} b$, $a''\eind[*]_{Mb} c$, and $b\eind[*]_{Ma''} c$. 
\end{thm}
\begin{proof}
Let $\kappa_a$, $\kappa_b$, and $\kappa_c$ be the cardinals provided by strong local character for $|a|$, $|b|$, and $|c|$, respectively. Let $\mu_c$ be a regular cardinal greater than $\max(\kappa_a,|Mb|)$, let $\mu_b$ be a regular cardinal greater than $\max(\kappa_c,|Ma|,\mu_c)$, and let $\mu_a$ be a regular cardinal greater than $\max(\kappa_b, |Mc|,\mu_b,\mu_c)$. 

Let $p(x)$, $q(y)$, and $r(z)$ be global $M$-invariant types extending $\tp(a/M) = \tp(a'/M)$, $\tp(b/M)$, and $\tp(c/M)$, respectively.

Apply Lemma~\ref{lem:doublesequence} to $q(y)$ and $r(z)$, obtaining sequences $B = (b_\beta)_{\beta<\mu_b}$ and $C = (c_\gamma)_{\gamma<\mu_c}$ where $C$ is $\eind[*]$-Morley over $Mb$ and $B$ is $MC$-indiscernible. Then apply it to $p(x)$ and $q(y)$, obtaining sequences $A = (a_\alpha)_{\alpha<\mu_a}$ and $\widehat{B} = (\widehat{b_\beta})_{\beta<\mu_b}$, where $\widehat{B}$ is $\eind[*]$-Morley over $Ma$ and $A$ is $M\widehat{B}$-indiscernible. Finally, apply it to $r(z)$ and $p(x)$, obtaining sequences $\widehat{C} = (\widehat{c_\gamma})_{\gamma<\mu_c}$ and $A' = (a_i')_{\alpha<\mu_a}$, where $A'$ is $\eind[*]$-Morley over $Mc$ and $\widehat{C}$ is $MA'$-indiscernible. 

Since $B$ and $\widehat{B}$ are both reverse $q$-Morley sequences over $M$, and $C$ and $\widehat{C}$ are both reverse $r$-Morley sequences over $M$, we may assume $\widehat{B} = B$ and $\widehat{C} = C$ at the expense of moving $A$ and $A'$ by automorphisms over $M$. Note that $A$ and $A'$ are both reverse $p$-Morley sequences over $M$, so $A\equiv_M A'$. We also have $A\eind[K]_M B$, $A'\eind[K]_M C$, and $B\eind[K]_M C$. 

So we can apply the independence theorem, obtaining a sequence $A'' = (a_\alpha'')_{\alpha<\mu_a}$ such that $A''\equiv_{MB} A$, $A''\equiv_{MC} A'$, and $A''\eind[K]_M BC$. This, together with the conclusions of Lemma~\ref{lem:doublesequence}, ensure that for any $\alpha<\mu_a$, $\beta<\mu_b$, and $\gamma<\mu_c$, we have:
\begin{enumerate}
\item $a''_\alpha \eind[K]_M b_\beta c_\gamma$.
\item $a''_\alpha b_\beta \equiv_M ab$, $a''_\alpha c_\gamma\equiv_M a'c$, and $b_\beta c_\gamma\equiv_M bc$. 
\item $A$ is $\eind[*]$-Morley over $Mc_\gamma$, $B$ is $\eind[*]$-Morley over $Ma''_\alpha$, and $C$ is $\eind[*]$-Morley over $Mb_\beta$. 
\end{enumerate}

For all $\beta<\mu_b$ and $\gamma<\mu_c$, since $A''$ is $\eind[*]$-Morley over $Mc_\gamma$ and $\mu_a > \max(\kappa_b,|Mc_\gamma|)$, there exists a $\delta(\beta,\gamma)<\mu_a$ such that $b_\beta\eind[*]_{Mc_\gamma}a''_{\delta'}$ for all $\delta'>\delta(\beta,\gamma)$. Now $\Delta = \{\delta(\beta,\gamma)\mid \beta<\mu_b, \gamma<\mu_c\}$ is a set of ordinals of cardinality $\leq \max(\mu_b,\mu_c)$. But $\mu_a$ is regular and greater than $\max(\mu_b,\mu_c)$, so $\Delta$ is not cofinal in $\mu_a$. Let $\delta'' = \sup\Delta<\mu_a$, and let $a_* = a''_{\delta''}$.

Similarly, for all $\gamma<\mu_c$, since $B$ is $\eind[*]$-Morley over $Ma_*$ and $\mu_b > \max(\kappa_c,|Ma_*|)$, there exists a $\delta(\gamma)<\mu_b$ such that $c_\gamma\eind[*]_{Ma_*} b_{\delta'}$ for all $\delta'>\delta(\gamma)$. Now $\Delta = \{\delta(\gamma)\mid \gamma<\mu_c\}$ is a set of ordinals of cardinality $\leq \mu_c$. But $\mu_b$ is regular and greater than $\mu_c$, so $\Delta$ is not cofinal in $\mu_b$. Let $\delta'' = \sup \Delta<\mu_b$, and let $b_* = b_{\delta''}$. 

Finally, since $C$ is $\eind[*]$-Morley over $Mb_*$ and $\mu_c>\max(\kappa_a,|Mb_*|)$, there exists a $\delta<\mu_c$ such that $a_*\eind[*]_{Mb_*} c_\delta$. Let $c_* = c_{\delta}$. 

In total, we have $a_*\eind[K]_M b_*c_*$, $a_*b_*\equiv_M ab$, $a_*c_*\equiv_M a'c$, $b_*c_*\equiv_M bc$, $a_*\eind[*]_{Mb_*} c_*$, $c
_*\eind[*]_{Ma_*} b_*$, and $b_*\eind[*]_{Mc_*} a_*$. It remains to move $b_*c_*$ back to $bc$ by an automorphism $\sigma$ fixing $M$, and let $a'' = \sigma(a_*)$. 
\end{proof}

\bibliographystyle{amsalpha}
\bibliography{the}

\end{document}